\documentclass[11pt]{amsart}

\usepackage{amsmath,amsthm,verbatim,amssymb,amsfonts,amscd, graphicx}
\usepackage{graphics}
\usepackage{tensor}
\usepackage{enumitem}
\usepackage[utf8]{inputenc}
\usepackage{graphicx}
\usepackage{xcolor}
\usepackage{hyperref}
\usepackage{tikz-cd}
\usepackage{mathtools}
\usepackage{esint}
\usepackage{marvosym}
\usepackage{tikz}
\usetikzlibrary{%
  matrix,%
  calc,%
  arrows%
}
\newcommand{\rom}[1]{\uppercase\expandafter{\romannumeral #1\relax}}

\newcounter{Counter}
\theoremstyle{plain}
\newtheorem{theorem}{Theorem}[section]

\newtheorem{maintheorem}{Main Theorem}[Counter]

\newtheorem{corollary}[theorem]{Corollary}
\newtheorem{lemma}[theorem]{Lemma}
\newtheorem{example}[theorem]{Example}
\newtheorem{remark}[theorem]{Remark}

\newtheorem{proposition}[theorem]{Proposition}

\theoremstyle{definition}
\newtheorem{definition}[theorem]{Definition}
\theoremstyle{definition}

\numberwithin{equation}{section}

\definecolor{darkblue2}{RGB}{0, 0, 200}
\definecolor{darkblue}{RGB}{0, 0, 100}
\definecolor{darkestblue}{RGB}{0, 0, 50}
\colorlet{lightblue}{darkblue!50!}
\definecolor{darkred}{RGB}{100, 0, 0}
\colorlet{lightred}{darkred!50!}
\definecolor{darkgreen}{RGB}{0,100,0}

\def\Ric{\operatorname{Ric}}

\def\N{\mathbb{N}}
\def\Z{\mathbb{Z}}

\def\R{\mathbb{R}}

\def\sup{\operatorname{sup}}
\def\inf{\operatorname{inf}}

\def\Tr{\operatorname{Tr}}

\def\sup{\operatorname{sup}}

\def\dist{\operatorname{dist}}

\def\div{\operatorname{div}}

\title{Rigid comparison geometry for Riemannian bands and open incomplete manifolds}
\author{Sven Hirsch}
\address{Department of Mathematics, Duke University, Durham, NC, 27708, USA}
\email{sven.hirsch@duke.edu}

\author{Demetre Kazaras}
\address{Department of Mathematics, Duke University, Durham, NC, 27708, USA}
\email{demetre.kazaras@duke.edu}

\author{Marcus Khuri}
\address{Department of Mathematics, Stony Brook University, Stony Brook, NY, 11794, USA}
\email{khuri@math.sunysb.edu}

\author{Yiyue Zhang}
\address{Department of Mathematics, University of California, Irvine, 92697, USA}
\email{yiyuez4@uci.edu}

\thanks{M. Khuri acknowledges the support of NSF Grant DMS-2104229.}

\begin{document}

\maketitle

\begin{abstract}
Comparison theorems are foundational to our understanding of the geometric features implied by various curvature constraints. This paper considers manifolds with a positive lower bound on either scalar, 2-Ricci, or Ricci curvature, and contains a variety of theorems which provide sharp relationships between this bound and notions of {\em{width}}. Some inequalities leverage geometric quantities such as boundary mean curvature, while others involve topological conditions in the form of linking requirements or homological constraints. In several of these results open and incomplete manifolds are studied, one of which partially addresses a conjecture of Gromov in this setting. The majority of results are accompanied by rigidity statements which isolate various model geometries -- both complete and incomplete -- including a new characterization of round lens spaces, and other models that have not appeared elsewhere. As a byproduct, we additionally give new and quantitative proofs of several classical comparison statements such as Bonnet-Myers' and Frankel's Theorem, as well as a version of Llarull's Theorem and a notable fact concerning asymptotically flat manifolds. The results that we present vary significantly in character, however a common theme is present in that the lead role in each proof is played by \emph{spacetime harmonic functions}, which are solutions to a certain elliptic equation originally designed to study mass in mathematical general relativity. 
\end{abstract}

\tableofcontents

\vspace{-.5in}

\section{Introduction}\label{S:Intro}

In comparison geometry \cite{CheegerEbin} one relates a geometric feature of a general Riemannian manifold to that of some appropriate model geometry, often taken to be a simply connected space form. These statements form a significant part of our geometric-analytic understanding of various curvature conditions, and commonly come with associated {\emph{rigidity statements}} which provide a revealing characterization of the model spaces. Prototypical examples include Toponogov's triangle comparison for manifolds with a lower sectional curvature bound, and the Bonnet-Myers diameter estimate for manifolds possessing uniformly positive Ricci curvature. Comparison and rigidity results leveraging the much weaker condition of a {\emph{positive lower bound on scalar curvature}} are more ephemeral and underlie many research programs in contemporary differential geometry. 

Historically speaking, scalar curvature comparison theorems employ techniques and ideas from the a priori disparate areas of spin geometry, minimal surface theory, and mathematical general relativity. See \cite{Brendle} for a partial survey. In the present work, we strengthen a number of known scalar and Ricci curvature comparison theorems from a novel and uniform perspective. The primary role in our argumentation is played by a certain elliptic PDE known as the {\emph{spacetime harmonic equation}}, the basic facts of which are reviewed in Section \ref{S:Preliminaries}. This equation originally appeared in \cite{HKK}, wherein its solutions were used to study the total mass of initial data sets for Einstein's equations. The tools we develop here can be considered as generalizing Stern's method of circle-valued harmonic maps on $3$-manifolds \cite{BrayStern, Stern}. Moreover, there is a strong parallel to draw between our work and the emergent techniques of Gromov's $\mu$-bubbles \cite{gromov2019four} and Cecchini-Zeidler's Callias operators \cite{CecchiniZeidler}, which generalize minimal surface and Dirac operator methods, respectively. 

Compared to methods based on the analysis of Jacobi fields along geodesics or on $\mu$-bubbles, the techniques developed here always yield quantitative statements involving bulk integrals of geometric quantities. 
Methods based on Dirac operator methods also yield quantitative statements, but have the potential downside of relying on index-theoretic arguments to produce solutions to the relevant elliptic equations. Solutions to the spacetime Laplace equation, on the other hand, are readily obtained by a fixed point procedure with robust applicability.

Though the results presented here appear diverse, nearly all belong to a class of comparison theorems we refer to as {\emph{band-width inequalities}}. These inequalities analyze compact Riemannian manifolds $(M,g)$ whose boundary components are separated into two disjoint and non-empty collections $\partial M=\partial_- M\sqcup \partial_+ M$. The data $(M,\partial_\pm M,g)$ is referred to as a {\emph{Riemannian band}}. A band-width inequality -- or simply band inequality -- is an upper bound for a Riemannian band's {\emph{width}}, or rather the distance $d(\partial_-M,\partial_+M)$, in terms of the mean curvature of $\partial M$ and a positive curvature lower bound of some type. Taking a limit of band inequalities applied to regions exhausting a given closed manifold allows one to prove interesting comparison theorems of a more classical flavor. One advantage of our band-mentality is that it only requires analysis on compact regions within the manifold of interest. We are able to exploit this feature and prove new theorems about {\emph{open and incomplete}} Riemannian manifolds. To our knowledge, Theorems \ref{Meyer boundary cor}, \ref{t:2RicIncomplete}, and \ref{LlarullIncomplete} are among the first comparison theorems for general incomplete manifolds with full rigidity statements.

Below and throughout this work all manifolds are assumed to be connected, oriented, Hausdorff, second-countable, and smooth. In Section \ref{SS:non-oriented} we discuss the alterations one must make to each theorem presented below in the nonorientable case.

\section{Statement of Results}\label{s:statements}

There are a wide range of results presented in this paper, some of which are known facts proven by new quantitative means, some of which are refinements of previous results, and others are novel. For convenience of the reader, we have highlighted four results that are labelled as \emph{Main Theorems} to indicate the relative level of significance. 

\subsection{Manifolds with positive Ricci curvature}


In Section \ref{SS:RicciIncomplete} we conduct an analysis of spacetime harmonic functions on manifolds with a positive lower bound on Ricci curvature. The technical centerpiece of that section, Lemma \ref{L:Meyer integral formula}, is a fundamental integral identity. This identity is used to establish a band-width inequality, originally obtained by Croke-Kleiner \cite[Theorem 3]{CrokeKleiner} via distance function comparison. Below and throughout, $H$ denotes the mean curvature of a manifold's boundary with respect to the outward pointing unit normal vector. 

\begin{theorem}\label{Meyer boundary}
Let $(M^n,\partial_\pm M^n,g)$, $n\geq 3$, be an $n$-dimensional Riemannian band. Consider $H_\pm:=-\min_{\partial_\pm M^n}H$, where $H$ is the mean curvature with respect to the outer normal. If $\Ric\ge (n-1)g$, then the width satisfies
\begin{equation}\label{e:meyerwidth}
   d(\partial_-M^n,\partial_+M^n) \le \arctan\frac{H_-}{n-1}+\arctan\frac{H_+}{n-1}.
\end{equation}
If additionally equality occurs in \eqref{e:meyerwidth}, then $M^n$ splits isometrically as a warped product 
\begin{equation}
(M^n,g)\cong\left(\left[-\arctan\frac{H_-}{n-1},\arctan\frac{H_+}{n-1}\right]\times\Sigma^{n-1},d\theta^2+\cos^2\theta g_\Sigma\right),
\end{equation}
where $(\Sigma^{n-1},g_\Sigma)$ is a closed $(n-1)$-manifold with $\Ric_{\Sigma}\geq (n-2)g_\Sigma$.
\end{theorem}

The next result concerns open Riemannian manifolds. If the ends of an open $(M^n,g)$ are separated into disjoint and non-empty classes $E_-$ and $E_+$, the {\emph{width $d(E_-,E_+)$ of $(M^n,E_\pm,g)$}} is the distance between $E_-$ and $E_+$, by which we mean the minimal length of paths traveling between $E_-$ and $E_+$. Related definitions and facts are discussed in Appendix \ref{s:openappendix}. By carefully taking limits within the proof of Theorem \ref{Meyer boundary} in the context of certain compact bands exhausting an open manifold, we establish the following.

\begin{maintheorem}\label{Meyer boundary cor}
Let $(M^n,g)$, $n\geq 3$, be an open $n$-dimensional Riemannian manifold with a closed hypersurface $\Sigma^{n-1}$ separating its ends into two disjoint nonempty classes $E_-$ and $E_+$. If $\mathrm{Ric}\geq(n-1)g$, then 
\begin{equation}\label{e:Meyersincomplete}
    d(\Sigma^{n-1},E_-)+d(\Sigma^{n-1},E_+)\leq\pi .
\end{equation} 
If additionally equality is achieved in \eqref{e:Meyersincomplete}, then $M^n$ splits isometrically as a warped product
\begin{equation}
(M^n,g)\cong\left(\left(-\frac{\pi}{2},\frac{\pi}{2}\right)\times\Sigma^{n-1},d\theta^2+\cos^2\theta g_\Sigma\right)
\end{equation}
where $g_\Sigma$ is a metric on $\Sigma^{n-1}$ satisfying $\Ric_\Sigma\geq (n-2)g_\Sigma$.
\end{maintheorem}

The fundamentally new content of Main Theorem \ref{Meyer boundary cor} is that it imposes no hypothesis of completeness or control on the geometry of the manifold's ends. Such results are rare -- rigidity theorems characterizing open and incomplete Riemannian manifolds include Theorems \ref{Meyer boundary cor}, \ref{t:2RicIncomplete}, \ref{LlarullIncomplete}, Zhu's Theorem \cite[Theorem 1.4]{Zhu}, and Gromov's \cite[Section 3.9]{gromov2019four}. In a similar spirit, Lee-Lesourd-Unger \cite{LLU} and Lesourd-Unger-Yau \cite{LUY} use the method of $\mu$-bubbles to study the the Riemannian positive mass theorem on incomplete asymptotically flat manifolds. We also mention the work of Cecchini-R{\"a}de-Zeidler \cite{CecchiniRadeZeidler} making use of both Dirac and $\mu$-bubble methods to study the scalar curvature of open and complete manifolds with at least two ends.

As an application, in Corollary \ref{c:myerscheng} we apply Main Theorem \ref{Meyer boundary cor} to open manifolds obtained by removing a pair of points from a given closed Riemannian manifold. This leads to new proofs of the classical Bonnet-Myers diameter estimate \cite{Myers} and Cheng's rigidity theorem \cite{Cheng}.

In Section \ref{SS:RicciAF} we give a new and relatively simple argument showing the nonexistance of nonflat asymptotically flat manifolds with nonnegative Ricci curvature.
This fact was previously shown by Zhang \cite[Theorem 3.4]{Zhang} under even weaker decay assumptions on $g$. 
Moreover, Anderson \cite[Theorem 3.5]{Anderson} established a similar result as a consequence of Bishop-Gromov's volume comparison, see also \cite{BKN} and \cite{Bartnik2} for related results. In the special case that $\Ric=0$, a version of this result is due to Schoen \cite[Proposition 2]{Schoen}, who employed harmonic coordinates and Bochner's formula.  

\begin{theorem}\label{t:AFricci}
Let $(M^n,g)$, $n\geq 2$ be a complete Riemannian manifold which is asymptotically flat\footnote{The definition of asymptotically flat manifolds is given in Section \ref{SS:RicciAF}.} of order $q>0$. If $\mathrm{Ric}\geq 0$, then $(M^n,g)$ is isometric to Euclidean space $(\R^n,\delta)$.
\end{theorem}

The proof of this result utilizes Green's functions, which may be viewed as a special case of spacetime harmonic functions. In this regard, we mention the recent work of Munteanu-Wang \cite{MW1,MW2} which also employs the use of Green's functions to obtain rigid comparison theorems, but for 3-manifolds with scalar curvature bounded from below.

\subsection{$3$-Manifolds of positive 2-Ricci curvature}

Given a whole number $k$ and a number $\kappa$, a Riemannian manifold is said to have {\emph{$k$-Ricci curvature at least $\kappa$}} if at each point the sum of the $k$ smallest eigenvalues of its Ricci endomorphism is at least $\kappa$. In Section \ref{S:HopfSphere}, we analyse $3$-manifolds with positive $2$-Ricci curvature. In general, the geometry of manifolds satisfying this curvature condition is less constrained compared to positive Ricci curvature ($1$-Ricci curvature), but more constrained compared to positive scalar curvature ($3$-Ricci curvature). 

\begin{remark}
Interestingly, the class of closed $3$-manifolds which support metrics of positive $2$-Ricci curvature coincides with those admitting positive scalar curvature metrics, which in the orientable case is the collection
\begin{equation}\label{e:pscmanifolds}
    \left(\#_{i=1}^kS^3/\Gamma_i\right)\#\left(\#_{j=1}^l S^1\times S^2\right)
\end{equation}
where $S^3/\Gamma_i$ are spherical space forms. To see this, first note that \eqref{e:pscmanifolds} contains all orientable positive scalar curvature $3$-manifolds, a fact which follows from the Poincar{\'e} conjecture \cite{Perelman} and earlier work of Gromov-Lawson \cite{GromovLawsonBig} and Schoen-Yau \cite{SY3}. On the other hand, the round metrics on $S^3/\Gamma_i$ and the product metric on $S^1\times S^2$ have positive $2$-Ricci curvature. Finally, the surgery theorems of Hoelzel \cite{Hoelzel} and Wolfson \cite{Wolfson} apply in this context and show that all of \eqref{e:pscmanifolds} admit such metrics.
\end{remark}

We next describe the appropriate band-width inequality tailored to the positive $2$-Ricci curvature condition. 

\begin{maintheorem}\label{HopfSphereBoundary}
Let $(M^3,\partial_\pm M^3,g)$ be a $3$-dimensional Riemannian band with no spherical classes in $H_2(M^3;\mathbb{Z})$. Consider the sign reversed minimal outward mean curvature $H_0=-\min_{\partial M^3}H$. If $(M^3,g)$ has $2$-Ricci curvature at least $4$, then $H_0>0$ and the width of the band satisfies
\begin{equation}\label{eq:widthinequality2}
w:=d(\partial_-M^3,\partial_+M^3)\leq\mathrm{arctan}(H_0/2).
\end{equation}
If additionally $\Ric\ge2g$ and equality is achieved in \eqref{eq:widthinequality2}, then the universal cover of $(M^3,g)$ is isometric to $([-\frac w2,\frac w2]\times\R^2,g_{\scriptscriptstyle{\Upsilon}})$
where
\begin{equation}\label{e:g'metric}
g_{\scriptscriptstyle{\Upsilon}}=d\rho^2+\phi_{\scriptscriptstyle{\Upsilon}}^2(\rho)dx^2+\psi_{\scriptscriptstyle{\Upsilon}}^2(\rho)dy^2,\qquad \rho \in \left[-\frac w2,\frac w2\right], \text{ }\text{ } (x,y)\in \mathbb{R}^2,
\end{equation}
and
\begin{equation}
\label{phi alpha}
    \begin{cases}
    \phi_{\scriptscriptstyle{\Upsilon}}(\rho)
    =2^{\frac{1-\Upsilon}{2}}\cos^{1-\Upsilon}(\rho+\frac{\pi}{4})\cos^{\frac{\Upsilon}{2}}(2\rho)
    \\ 
    \psi_{\scriptscriptstyle{\Upsilon}}(\rho)
    =2^{\frac{1-\Upsilon}{2}}\sin^{1-\Upsilon}(\rho+\frac{\pi}{4})\cos^{\frac{\Upsilon}{2}}(2\rho)
    \end{cases}
\end{equation}
for some $\Upsilon\in[0,1]$.
\end{maintheorem}

The prototypical example of a band achieving the maximal width of Main Theorem \ref{HopfSphereBoundary} is a neighborhood of the Clifford torus in the round $3$-sphere, which is associated with the case $\Upsilon=0$. It is interesting to point out that even though the curvature condition here is weaker than $\Ric\geq 2g$, inequality \eqref{eq:widthinequality2} is stronger than the width inequality in Theorem \ref{Meyer boundary}. The homological restriction of Main Theorem \ref{HopfSphereBoundary} is responsible for this strengthening. Notice also the more restrictive curvature condition $\Ric\geq2g$ for the rigidity statements of the theorems stated in this subsection. This condition is in fact necessary, as is illustrated by Example \ref{example} below.  Similarly to the previous band inequalities, we establish a version of Main Theorem \ref{HopfSphereBoundary} for open and incomplete manifolds.

\begin{maintheorem}\label{t:2RicIncomplete}
Let $(M^3,g)$ be an open $3$-dimensional Riemannian manifold with a closed surface $\Sigma^2$ separating its ends into two disjoint nonempty classes $E_-$ and $E_+$. Assume that $H_2(M^3;\mathbb{Z})$ contains no spherical classes. If the $2$-Ricci curvature of $(M^3,g)$ is at least $4$, then 
\begin{equation}\label{e:2RicIncomplete}
d(\Sigma^2,E_-)+d(\Sigma^2,E_+)\leq\frac{\pi}{2}.
\end{equation} 
If additionally $\mathrm{Ric}\geq2g$ and equality is achieved in \eqref{e:2RicIncomplete}, then the universal cover of $(M^3,g)$ is isometric to $((-\frac \pi4,\frac \pi4)\times\R^2,g_\Upsilon)$ for some $\Upsilon\in[0,1]$, where $g_\Upsilon$ is as in Main Theorem \ref{HopfSphereBoundary}.
\end{maintheorem}


The above result is of a similar spirit to Zhu's \cite[Theorem 5.1]{Zhu}, which was inspirational to the present work. We point out that Main Theorem \ref{t:2RicIncomplete} fully settles the case of rigidity, which was only partially addressed in the aforementioned paper, thereby resolving an inconsistency in the rigidity statement of \cite[Theorem 5.1]{Zhu}.

Let us take a moment to place the previous two results in some context. Prior to the maturation of the techniques mentioned here, it was unclear to what extent Llarull's Theorem held for {\emph{non-spin}} manifolds. To this end, Gromov \cite{Gromov} adopted the following strategy. First construct wide toric bands embedded as codimension-$0$ submanifolds in the round sphere. Then, given a map from a positive scalar curvature manifold to the round sphere, apply a version of the torus band inequality (see Theorem \ref{ToricBand} below) -- known to hold in dimensions less than $8$, irrespective of the spin condition -- to the preimage of this band and thereby obtain an estimate on the map's Lipschitz constant. In this method, the sharpness of this Lipschitz constant estimate is highly dependent on the first step wherein one seeks a codimension-$1$ embedded torus in the round sphere with a large normal injectivity radius. In dimensions above $3$, there appears to be no canonical choice of this torus and though the construction of embedded high-dimensional tori in \cite{Gromov} is quite clever, the resulting Lipschitz constant estimate is far from sharp. In dimension $3$, however, there is an embedded torus of maximal normal injectivity radius: the minimal Clifford torus in $S^3$ and their lens space quotients. 

As a consequence of Main Theorem \ref{t:2RicIncomplete}, we obtain the following corollary. It is a refinement of the result \cite[Corollary 1.7]{Zhu}, which itself addressed a conjecture of Gromov \cite{Gromov} arising from the above discussion.  Below, given an embedded submanifold $\Sigma^2$ in a Riemannian manifold $(M^3,g)$, its {\em{normal injectivity radius}} will be denoted by $\textit{Inj}_n(\Sigma^2)$. Recall that this quantity represents the largest distance from $\Sigma^2$, such that all points within the radius admit a unique geodesic minimizing distance to the surface. 

\begin{corollary}\label{HopfSphere}
Let $(M^3,g)$ be a closed Riemannian manifold with $2$-Ricci curvature at least $4$. If $\Sigma^2\subset M^3$ is a connected embedded closed surface of positive genus, then 
\begin{align}\label{radius}
    \text{Inj}_n(\Sigma^2)\le \frac\pi4.
\end{align}
If additionally $\Ric\ge2g$ and equality occurs in \eqref{radius}, then the universal cover of $(M^3,g)$ is isometric to the round sphere and $\Sigma^2$ lifts to the Clifford torus. Moreover, in this case $(M^3,g)$ is isometric to a round sphere or a round lens space.
\end{corollary}

As a second application of Main Theorem \ref{t:2RicIncomplete}, one can give a rigid upper bound for the distance between knots. Recall that a rational homology $3$-sphere $M^3$ is a closed oriented $3$-manifold with vanishing first and second Betti numbers. If $K_1$ and $K_2$ are embedded circles in such a manifold, there is a $2$-cycle $C_i$ such that $\partial C_i$ is a multiple of the torsion class $[K_i]\in H_1(M^3;\mathbb{Z})$ for $i=1,2$. The knots $K_1$ and $K_2$ are {\emph{linked}} if the signed transverse intersection of $K_1$ and $C_2$ is nonzero.

\begin{corollary}\label{c:linkCor}
Let $(M^3,g)$ be a closed $3$-dimensional Riemannian manifold where $M^3$ is a rational homology sphere. Suppose that $K_1,K_2\subset M^3$ are two linked knots. If the $2$-Ricci curvature of $(M^3,g)$ is at least $4$, then $d(K_1,K_2)\leq\pi/2$. If additionally $\mathrm{Ric}\geq2g$ and $d(K_1,K_2)=\pi/2$, then the universal cover of $(M^3,g)$ is the round $3$-sphere and $K_1\cup K_2$ lifts to the Hopf link. Moreover, in this case $(M^3,g)$ is isometric to a round sphere or a round lens space.
\end{corollary}

We would like to emphasize that Corollaries \ref{HopfSphere} and \ref{c:linkCor} are, to the authors' knowledge, the only general comparison theorems which specifically characterize the class of round lens spaces. It is worth pointing out the unrelated characterization of the round $\mathbb{RP}^3$ by 
Bray-Brendle-Eichmair-Neves \cite{BBEN} in terms of the area of certain $2$-cycles.

\begin{example}\label{example}
There exists an explicit $1$-parameter family of smooth and nonround metrics on $S^3$ with $2$-Ricci curvature at least $4$, each of which contains an embedded torus of normal injectivity radius equal to $\pi/4$. Furthermore, utilizing the `Hopf link' or a thickening thereof, this family may be employed to obtain counterexamples to the rigidity statements in Main Theorems \ref{HopfSphereBoundary} and \ref{t:2RicIncomplete}, as well as Corollary \ref{c:linkCor}, when the Ricci lower bound is not assumed. See Section \ref{s:counterexample} for details.

\end{example}

\subsection{$3$-Manifolds of positive scalar curvature}
The methods used above to treat positive Ricci and 2-Ricci curvature, may also be applied in the context of scalar curvautre.
In Section \ref{S:Llarull} we study the Lipschitz constants of maps from positive scalar curvature manifolds to the unit round $3$-sphere $(S^3,g_{S^3})$. In particular, we show the following quantitative statement which recovers the illuminating theorem of Llarull \cite{Llarull} in the special case of $3$-manifolds satisfying a topological condition.

\begin{theorem}\label{Llarull}
Let $(M^3,g)$ be a closed 3-dimensional Riemannian manifold with $H_2(M^3;\Z)=0$. Suppose that $\ell:(M^3,g)\to (S^3,g_{S^3})$ is a $C^1$ map of nonzero degree with Lipschitz norm $\mathrm{Lip}(\ell)\le1$. Then there exists a nonconstant spacetime harmonic function $u$ on $M^3$ which is $C^{2,\alpha}$-smooth away from a finite set of points, for any $\alpha\in(0,1)$, such that
\begin{equation}\label{e:Llarullquant}
\int_{M^3}(6-R)|\nabla u|dV\geq\int_{M^3}\frac{{\big{|}}\nabla^2u+\cot(\theta\circ\ell)|\nabla u|\;g{\big{|}}^2}{|\nabla u|}dV
\end{equation}
where $\theta(x)$ denotes the spherical distance between a point $x\in S^3$ and the north pole.  If additionally the scalar curvature satisfies $R\geq6$, then $\ell$ is an isometry.
\end{theorem}

\begin{remark} 
In the statement of Theorem \ref{Llarull} we include the quantitative inequality \eqref{e:Llarullquant}. In fact, as a biproduct of the techniques employed, all results in this manuscript could be stated in such terms, but for brevity we only make this explicit here. It is worth noting that such quantitative statements involving Hessians could have implications for associated stability questions, see for instance \cite{CheegerColding, KKL}.
\end{remark}


Localized or rigid quantitative band-type versions of Llarull's theorem are also possible. See \cite{GS}, \cite{Listing}, and \cite{Lott} for very general results of this variety. In the following, let $N$ and $S$ denote the north and south poles in $S^3$. Given two numbers $r_1,r_2\in(0,\pi)$ such that $r_1<r_2$, we will write $A[r_1,r_2]$ for the closed annular region in the unit round sphere consisting of points $x\in S^3$ satisfying $\theta(x)=d_{S^3}(x,N)\in[r_1,r_2]$.

\begin{theorem}\label{LlarullBand}
Let $(M^3,\partial_\pm M^3,g)$ be a $3$-dimensional Riemannian band such that both $\partial_+M^3$ and $\partial_-M^3$ are connected, and $H_2(M^3,\partial M^3;\mathbb{Z})=0$. Suppose that $0<r_1<r_2<\pi$, and $\ell:M^3\to A[r_1,r_2]$  is a $C^1$ map of nonzero degree with $\mathrm{Lip}(\ell)\leq 1$. If $R\geq6$, then there must exist a point $x\in\partial M$ so that the outward normal mean curvature satisfies $H(x)\leq H_{g_{S^3}}(\ell(x))$. If additionally $H(x)\geq H_{g_{S^3}}(\ell(x))$ for all points $x\in\partial M^3$, then $(M^3,g)$ is isometric to an annular region in the unit round $3$-sphere.
\end{theorem}

Conjecturally, the extremal character of $(S^3,g_{S^3})$ articulated by Theorem \ref{Llarull} is even more robust -- Gromov has suggested  \cite[Conjecture D]{Gromov} that the open and incomplete manifold formed by removing finitely many points from the round sphere enjoys the same property. We confirm this statement in the next result, for dimension $3$ in the special case of a pair of antipodal points. See Gromov's four lectures on scalar curvature \cite[Section 3.9]{gromov2019four} for an extended discussion and an alternative set of arguments, for this and related extremality statements. See also the recent similar result \cite[Theorem 1.6]{HLS} of Hu-Liu-Shi.

\begin{maintheorem}\label{LlarullIncomplete}
Let $g$ be a Riemannian metric on $S^3\setminus\{N,S\}$. If $g\geq g_{S^3}$, then there is a point $x\in S^3\setminus\{N,S\}$ where the scalar curvature satisfies $R(x)\leq 6$. If additionally $R\geq6$, then $g$ agrees with the round metric $g_{S^3}$.
\end{maintheorem}

Prior to Llarull's work, Gromov-Lawson \cite{GromovLawson} developed a homotopy-theoretic obstruction to the existence of positive scalar curvature metrics on closed spin manifolds called {\emph{enlargability}}. The enlargability obstruction articulates the following heuristic: if scalar curvature is positive, at each point there is at least one positive eigenvalue of the Ricci endomorphism and so the geometric logic of Bonnet-Myers dictates that the manifold cannot expand dramatically in all directions simultaneously. The $n$-torus, for instance, may be viewed as expanding in all directions by passing to covers, and therefore cannot support a positive scalar curvature metric. Soon after this work, Gromov-Lawson \cite[Section 12]{GromovLawsonBig} qualified this obstruction with arguments that lead to a sharp upper bound on the width of torical bands $(T^{n-1}\times[0,1],g)$, $n\leq 7$ of uniformly positive scalar curvature, a result which would be made explicit and expounded upon by Gromov in \cite{Gromov} and dubbed the {\em{Torical $\frac{2\pi}{n}$-Inequality}}. In Section \ref{S:Band}, we establish a generalized form of this band-width inequality. We will refer to a nontrivial homology class ${\bf{c}}\in H_2(M^3)$ as {\emph{spherical}}, if there is an embedding $S^2\hookrightarrow M^3$ so that ${\bf{c}}$ is the image of the fundamental class $[S^2]$.

\begin{theorem}\label{ToricBand}
Let $(M^3,\partial_\pm M^3,g)$ be a $3$-dimensional Riemannian band such that $H_2(M^3;\mathbb{Z})$ contains no spherical classes. Consider the sign reversed minimal outward mean curvature $H_0=-\min_{\partial M^3}H$.
If $R\geq 6$, then $H_0>0$ and the width of the band satisfies
\begin{equation}\label{eq:widthinequality}
w:=d(\partial_-M^3,\partial_+M^3)\leq\frac{4}{3}\arctan(H_0/2).
\end{equation}
If additionally equality is achieved in \eqref{eq:widthinequality}, then $M^3$ splits isometrically as a warped product 
\begin{equation}
(M^3,g)\cong\left([-w/2,w/2]\times T^2,\text{ }\! ds^2 +\cos^{\frac{4}{3}}\!\left(\tfrac{3}{2}s \right)g_0\right),
\end{equation}
where $g_0$ is a flat metric on the torus $T^2$.
\end{theorem}

\begin{remark}
In analogy with Corollary \ref{c:linkCor}, a straightforward consequence of Theorem \ref{ToricBand} asserts that the distance between linked knots in a rational homology $3$-sphere with scalar curvature at least $6$ cannot exceed $2\pi/3$. Compared to Corollary \ref{c:linkCor}, however, this upper bound is never attained by a smooth Riemannian manifold.
\end{remark}

Contemporaneously to the present work, Chai-Wan \cite[page 8]{ChaiWan} have established a version of Theorem \ref{ToricBand} for $3$-dimensional initial data sets using a similar method. In \cite{Cecchini, CecchiniZeidler,Zeidler}, Cecchini and Zeidler proved a variety of band-width inequalities in arbitrary dimensions using Callias operators. We also point out R\"ade's article \cite{Rade} on band inequalities based on the method of $\mu$-bubbles, as well as the related rigidity statement \cite[Corollary 1.4]{EichmairGallowayMendes} of Eichmair-Galloway-Mendes.

Finally, in Section \ref{S:waist} we prove two {\emph{waist inequalities}}, which articulate the heuristic that $3$-dimensional positive scalar curvature manifolds macroscopically resemble $1$-dimensional complexes. 
Informally speaking, Theorem \ref{thm:waist} below asserts that a $3$-manifold with scalar curvature at least $R_0>0$ admits a map to an interval whose fibers have average area bounded above in terms of $R_0$. Results of this flavor go back to Gromov-Lawson \cite{GromovLawson}, and the state-of-the-art is due to Liokumovich-Maximo \cite{LiokumovichMaximo}. 

\begin{theorem}\label{thm:waist}
Let $(M^3,g)$ be a closed $3$-dimensional Riemannian manifold with scalar curvature bounded below by a positive constant, $R\ge R_0 >0$. Assume that $\textnormal{diam}(M^3)\geq \frac{4\pi}{\sqrt{3R_0}}$, and let $p,q\in M^3$ be points whose mutual distance achieves the diameter. Then there exists a spacetime harmonic function $u\in C^{2,\alpha}(M^3 \setminus\{p,q\})\cap C^{0,1}(M^3)$ for any $\alpha\in(0,1)$ with min/max values $u(p)=-1$ and $u(q)=1$, such that
\begin{equation}\label{e:waistest}
\mathrm{Avg}(\Sigma_t)\leq  \frac{16\pi}{R_0}\left(b_2 +1\right),
\end{equation}
where $\mathrm{Avg}(\Sigma_t)$ is the average $2$-dimensional Hausdorff measure of all $u$-level sets $\Sigma_t=u^{-1}(t)$, and $b_2$ denotes the second Betti number of $M^3$.
\end{theorem}

The next result decomposes a positive scalar curvature $3$-manifold into bands whose widths and boundary areas are bounded above in terms of the smallest value of its scalar curvature. Theorem \ref{thm:dice} closely resembles the ``Slice" part of Chodosh-Li's Slice-and-Dice construction \cite[Section 6.3]{ChodoshLi}. Below, given a function $u:M^3\to \mathbb{R}$, we use the notation $n(t)$ for the number of path components of the $t$-level set $u^{-1}(t)$. See Figure \ref{figV} for a schematic depiction of the following result.

\begin{theorem}\label{thm:dice}
Let $(M^3,g)$ be a closed $3$-dimensional Riemannian manifold with scalar curvature bounded below by a positive constant, $R\ge R_0 >0$. Then there exist a sequence of closed regions $\{V_i\}_{i=1}^I$ covering $M^3$ which satisfy the following:
\begin{enumerate}
\item The number of regions is bounded above by the diameter and scalar curvature lower bound 
$I-1\leq \tfrac{\sqrt{3R_0}}{4\pi}\mathrm{diam}(M^3)$.

\item For each $i$, the boundary $\partial V_i$ is a $C^2$ surface.

\item Elements of the sequence only have nontrivial intersection with their neighbors, and the intersection consists entirely of boundary components, that is, $V_i \cap V_j =\emptyset$ unless $j=i\pm 1$, in which case $V_i \cap V_j =\partial V_i \cap \partial V_j$.

\item The average area of components in each intersection surface satisfies
\begin{equation}
\frac{|\partial V_i \cap \partial V_{i+1}|}{N_i}\leq \frac{16\pi }{R_0},\quad\quad\quad i=1,\ldots, I-1,
\end{equation}
where $N_i$ denotes the number of components of $\partial V_i \cap \partial V_{i+1}$.

\item The intersection surfaces lie within a fixed distance to one another
\begin{equation}
d(\partial V_{i-1} \cap \partial V_i , \partial V_i \cap \partial V_{i+1})\leq \frac{16\pi}{\sqrt{3R_0}},\quad\quad\quad i=2,\ldots,I-1,
\end{equation}
with the first and last regions staying within a fixed Hausdorff distance to their boundaries
\begin{equation}
d_H(V_1,\partial V_1)+d_H(V_I,\partial V_I)\leq \frac{32\pi}{\sqrt{3R_0}}.
\end{equation}
\end{enumerate}
\end{theorem}

\medskip
\noindent\textbf{Acknowledgements.} The authors would like to thank Hubert Bray, Simon Brendle, Simone Cecchini, Richard Schoen, Rudolf Zeidler, and Jintian Zhu for insightful discussions. 
\section{Preliminaries}\label{S:Preliminaries}

\subsection{Initial data sets}\label{s:ids}
The main arguments of the present work rely on a common tool, namely the spacetime harmonic function. To describe these functions, their properties, and their relevance to comparison geometry, we require some preliminary notions. A triple $(M^n,g,k)$ consisting of a Riemannian manifold $(M^n,g)$ and a symmetric 2-tensor $k$ is called an {\emph{initial data set}}. In mathematical general relativity, such data represent a spatial slice of spacetime with induced metric $g$ and second fundamental form $k$, and can be thought of as initial conditions for the Einstein equations \cite{Lee}. The energy and momentum densities are important local invariants of the slice, and are given by
\begin{equation}
\mu=\frac12\left(R-|k|^2+(\text{Tr}_g k)^2\right),\quad\quad\quad J=\div\left(k-(\text{Tr}_g k)g\right).
\end{equation}
These expressions arise directly from the Gauss-Codazzi relations associated with the embedding into spacetime, by taking traces. From a physical perspective, these quantities agree with certain components of the stress-energy tensor, which encodes relevant information concerning the matter fields on spacetime. A typical requirement for the physical significance of $(M^n,g,k)$ is the dominant energy condition, which stipulates that $\mu\ge|J|$ holds across $M^n$.
Geometrically, this may be viewed as a type of lower bound for scalar curvature.

In \cite{HKK} the spacetime harmonic equation 
was introduced to study the total mass of asymptotically flat $3$-dimensional initial data sets. 
This equation is given by
\begin{equation}\label{e:spacetimeharmonicequation0}
    \Delta u+\left(\text{Tr}_g k\right)|\nabla u|=0,
\end{equation}
and solutions are referred to as \textit{spacetime harmonic functions}. A fundamental property of these functions is 
that they satisfy an integral inequality \cite[Proposition 3.2]{HKK}, which relates the dominant energy condition quantity $\mu-|J|$ with boundary geometry of the initial data set. The left-hand side of equation \eqref{e:spacetimeharmonicequation0} is equivalent to the trace of the {\emph{spacetime Hessian}}
\begin{equation}\label{e:spacetimehessian0}
{\overline{\nabla}}^2u:=\nabla^2u+|\nabla u|k,
\end{equation}
which also plays a significant role in the integral inequality. For a discussion of the geometric meaning of the spacetime Hessian, see the survey article \cite{BHKKZ}. In addition to providing a proof of the spacetime positive mass theorem in \cite{HKK}, there have been various applications of spacetime harmonic functions to both initial data sets and Riemannian geometry in \cite{AHK, BHKKZ, BHKKZ2, BKKS, Chai, ChaiWan, HZ, Tsang1, Tsang2}.

\subsection{Spacetime harmonic functions on bands and scalar curvature}

Let $(M^n,\partial_\pm M^n,g)$ be an $n$-dimensional Riemannian band. Given a function $f\in\mathrm{Lip}(M^n)$, we may define the symmetric $2$-tensor $k:=fg$ and consider the auxiliary initial data set $(M^n,g,k)$. Throughout the paper, our arguments will focus on the spacetime harmonic functions associated with $(M^n,g,k)$ having Dirichlet boundary conditions. The following proposition is an immediate consequence of the more general existence result discussed in \cite[Section 4]{HKK}. 

\begin{proposition}\label{p:existence} Let $(M^n,\partial_\pm M^n,g)$ be an $n$-dimensional Riemannian band, and consider a function $f\in \mathrm{Lip}(M^n)$, as well as constants $c_-<c_+$. Then for any $\alpha\in(0,1)$, there exists a unique solution $u\in C^{2,\alpha}(M^n)$ of the spacetime harmonic Dirichlet problem
\begin{equation}\label{e:bandspacetimeharmoniceq}
\begin{cases}
    \Delta u+nf|\nabla u|=0&\text{ in }M^n ,\\
    u= c_\pm&\text{ on }\partial_\pm M^n .
\end{cases}
\end{equation}
\end{proposition}

The next result expresses the fundamental integral inequality associated with spacetime harmonic functions in the current setting, which is specialized to dimension $3$. More precisely, if $(M^3,g,k=fg)$ is an initial data set as described above, then a brief calculation shows that its energy and momentum densities take the form
\begin{align}
\mu=\frac12R+3f^2,\quad\quad \quad J=-2\nabla f,
\end{align}
wherever $f$ is differentiable. With this observation in mind, the more general inequality derived in \cite[Proposition 3.2]{HKK} directly implies the next lemma.
We note that although \cite[Proposition 3.2]{HKK} is stated in coarea form, this may be re-expressed as a bulk integral since the critical point set 
$\{x\in\Omega:\nabla u(x)=0\}$ is of Hausdorff codimension at least $2$; this latter fact  is due to \cite[Theorem 1.1]{NaberV} after viewing  
the spacetime Laplace equation as a linear equation in which the first order coefficient are $L^\infty$.

\begin{lemma}\label{integralformula}
Let $(M^3,\partial_\pm M^3,g)$ be a $3$-dimensional Riemannian band, and let $f\in \mathrm{Lip}(M^3)$. If $u\in C^{2,\alpha}(M^3)$, $\alpha\in(0,1)$ solves boundary value problem \eqref{e:bandspacetimeharmoniceq}, then
\begin{align}\label{integralformula2}
\begin{split}
\int_{\partial_-M^3}2|\nabla u|(2f-H)&dA -\int_{\partial_+M^3}2|\nabla u|(2f+H)dA\\
\geq&\int_{M^3}\left(\frac{|\overline{\nabla}^2u|^2}{|\nabla u|}+(R+6f^2)|\nabla u|-4\langle\nabla f,\nabla u\rangle\right)dV -\int_{c_-}^{c_+} 4\pi\chi(\Sigma_t)dt
\end{split}
\end{align}
where 
$H$ is the outward mean curvature of $\partial M^3$, and $\chi(\Sigma_t)$ is the Euler characteristic of regular level sets $\Sigma_t:=u^{-1}(t)$.
\end{lemma}

\begin{remark}\label{alkjflkjah}
Even though the function $f$ is only Lipschitz, Rademacher's Theorem ensures that its derivatives exist almost everywhere, justifying the appearance of $\nabla f$ in \eqref{integralformula2}. Moreover, we claim that the Euler characteristic integrand is a measurable function.  To see this, note that as explained in \cite[Remark 3.3]{HKK}, the conclusion of Sard's theorem holds for the function $u$ even though it may fail to be $C^3$-smooth. 
Furthermore, $u$ is a proper map and so its regular values form an open set of full measure. 
Thus, for a regular value $t_0$ of $u$, we find that the function $t\mapsto\chi(\Sigma_t)$ is constant for all levels $t$ near $t_0$. It follows that $\chi(\Sigma_t)$ is continuous almost everywhere, and is therefore measurable. 
\end{remark}

We now state a technical lemma which holds in all dimensions. It shows that there are Lipschitz functions $f$ with certain desirable properties. These functions will be used in later sections to construct appropriate auxiliary spacetime second fundamental forms.

\begin{lemma}\label{construct f}
Let $(M^n,\partial_\pm M^n,g)$ be an $n$-dimensional Riemannian band and denote the width by $w=\dist(\partial_-M^n,\partial_+M^n)$. Let $a,b,\varepsilon>0$, and $C>\frac1a\tan\left(\frac{\pi}{2(1+\epsilon)}\right)$ be parameters. If $w= \frac{b\pi}{a(1+\varepsilon)}$, then there exists a function $f_\varepsilon\in\mathrm{Lip}(M^n)$ satisfying the following properties:
\begin{enumerate}
\item $f_\varepsilon$ is an increasing function of the the distance to $\partial_-M^n$,
\item there is a constant $C_1$ depending only on $a,b$ so that almost everywhere
\begin{align} \label{ODE f1}
\begin{split}
   a^2f_\varepsilon^2(x)-b|\nabla f_\varepsilon(x)|+1\ge -C_1\varepsilon,\quad\quad\;& \text{if }\text{ }\big|d(x,\partial_- M)-\frac{w}{2} \big|\leq \frac{w}{4},
    \\ 
    a^2f_\varepsilon^2(x)-b|\nabla f_\varepsilon(x)|+1\ge0,\quad\quad\;&\text{if }\text{ } \big|d(x,\partial_- M)-\frac{w}{2}\big|> \frac{w}{4},
    \end{split}
\end{align}
\item $f_\varepsilon \le -C$ on $\partial_-M^n$ and $f_\varepsilon\ge C$ on $\partial_+M^n$.
\end{enumerate}
\end{lemma}

\begin{proof}
Consider the $1$-variable function
\begin{align}
\phi_\varepsilon(\tau)=
    \begin{cases}
    (1+10\varepsilon)\tau&\text{ if }|\tau|\le c,\\
    \tau+\frac \tau{|\tau|}10\varepsilon c&\text{ if }|\tau|\geq c,
    \end{cases}
\end{align}
where
\begin{align}
   0< c=
\frac{b}{10\varepsilon a}\left(\tan^{-1}(aC)-\frac{wa}{2b}\right)< \frac w{20}.
\end{align}
Next define 
\begin{align}
    \tilde f_\epsilon(\tau)=
\begin{cases}
    \frac1a\tan\left(\frac ab\phi_{\varepsilon }
    \left(\tau-\frac w2\right)\right)&\text{if }\tau\le w ,\\
    L(\tau)&\text{if }\tau>w,
\end{cases}
\end{align}
where $L$ is the unique linear function making $\tilde f_\epsilon(\tau)$ differentiable at $\tau=w$.
The desired function may then be obtained by setting $f_\epsilon(x)=\tilde f_\epsilon(r(x))$, where
$r(x)=d(x,\partial_-M^n)$.
\end{proof}

\subsection{Nonorientable manifolds}\label{SS:non-oriented}
Throughout this work we assume that all manifolds are orientable. Although many of the results presented hold as stated in the nonorientable case by passage to the orientable double cover, some theorems require minor modifications which we address here. This later collection consists of Theorem \ref{ToricBand}, and Main Theorems \ref{HopfSphereBoundary} and \ref{t:2RicIncomplete}. Before describing those modifications, it is worth noting that nonorientable $3$-manifolds never satisfy the homological hypotheses of Theorem \ref{Llarull} and \ref{LlarullBand}, since their rank $2$ integral homology groups always contain a $\mathbb{Z}_2$ summand, a fact which can be derived from the universal coefficients theorem. Lastly, due to their specialized nature, we refrain from asserting nonorientable versions of Corollary \ref{c:linkCor}, and Theorems \ref{thm:waist} and \ref{thm:dice}.

We now briefly describe the changes required for the three results mentioned above in the nonorientable case. In this setting the assumptions of Theorem \ref{ToricBand}, Main Theorem \ref{HopfSphereBoundary}, and Main Theorem \ref{t:2RicIncomplete} must be strengthened by imposing that there are no \emph{immersed spherical classes}, which is to say there are no spherical classes in the orientable double cover.
In addition, the statement of Theorem \ref{ToricBand} should be adjusted further to account for the fact that a nonorientable band achieving equality in \eqref{eq:widthinequality} will split as a warped product with a Klein bottle instead of a torus.

\section{Spacetime Harmonic Functions and Ricci Curvature} \label{S:Ricci}

In this section we analyze the interactions between spacetime harmonic functions and Ricci curvature. 
As a precursor which illustrates the main approach, we give an elementary proof of Frankel's Theorem concerning minimal surfaces in manifolds with nonnegative Ricci curvature.

\subsection{Minimal surfaces and nonnegative Ricci curvature}


\begin{theorem}\label{T:Frankel}
Let $(M^n,g)$, $n\ge2$ be a compact Riemannian manifold with $\Ric\ge0$.
Suppose that its boundary $\partial M^n$ is minimal and consists of at least two components.
Then $M^n=I\times \Sigma^{n-1}$ for an interval $I$ and a hypersurface $\Sigma^{n-1}$, and the metric splits as a product $g=dt^2+g_{\Sigma^{n-1}}$.
\end{theorem}

\begin{proof}
Group the components of $\partial M^n$ into two nonempty disjoint collections $\partial_-M^n$ and $\partial_+M^n$ with $\partial_-M^n\cup\partial_+M^n=\partial M^n$. 
Let $u$ be the harmonic function on $M^n$ with $u=\pm 1$ on $\partial_\pm M^n$.
Then we have by Bochner's formula and the minimality of $\partial M^n_\pm$ that
\begin{align}
\begin{split}
    0\le&\int_{M^n}(|\nabla^2u|^2+ \Ric(\nabla u,\nabla u))dV
    \\=&\int_{M^n}\frac12\Delta|\nabla u|^2dV\\
    =&\int_{\partial M^n}|\nabla u|\partial_{\textbf n}|\nabla u|dA
    \\=&-\int_{\partial M^n}H|\nabla u|^2\\
    =&0,
    \end{split}
\end{align}
where $\textbf n$ is the unit outer normal to $\partial M^n$ and $H$ is the mean curvature with respect to $\textbf n$.
Hence $\nabla^2u=0$ on $M^n$, which implies that $M^n$ splits as a product.
\end{proof}

As a consequence to Theorem \ref{T:Frankel}, we obtain the following result due to Frankel \cite{Frankel}.

\begin{corollary}\label{Frankel Cor}
Any two smooth minimal hypersurfaces $\Sigma_1^{n-1},\Sigma_2^{n-1}$ in a closed Riemannian manifold $(M^n,g)$, $n\geq 2$ with $\Ric>0$ must intersect.
\end{corollary}

\begin{proof}
Proceeding by contradiction, suppose that there are two connected and disjoint smooth minimal hypersurfaces $\Sigma_1^{n-1}$ and $\Sigma_2^{n-1}$. Let $M_1^n$ be the component of $M^n\setminus \Sigma_1^{n-1}$ which contains $\Sigma_2^{n-1}$. Next, let $M_2^n$ be the metric completion of the component from $M_1^{n}\setminus\Sigma_2^{n-1}$ which contains at least two boundary components. Since its boundary is minimal, we may apply Theorem \ref{T:Frankel} to find that $M_2^n$ splits as a product. In particular, the Ricci curvature of $M_2^n$ vanishes in at least one direction, contradicting the assumption $\Ric>0$.
\end{proof}

\subsection{Ricci curvature and asymptotically flat manifolds}\label{SS:RicciAF}

Next, we study the Ricci curvature of asymptotically flat manifolds. We say that a complete Riemannian manifold $(M^n,g)$ is \textit{asymptotically flat} of order $q>0$, if there exists a compact set $\Omega\subset M^n$ and a diffeomorphism $\psi:M^n\setminus\Omega\to \sqcup_{i=1}^k\left(\R^n\setminus B\right)$ for a whole number $k$ and ball $B\subset \mathbb{R}^n$, such that $(\psi^{-1})^*g-\delta=O_2(\rho^{-q})$ as $\rho\rightarrow\infty$. Here $\rho$ is the radial coordinate of $\mathbb{R}^n$, and $\delta$ denotes the standard flat metric, while the subindex 2 indicates additional decay for derivatives in the usual manner. 

\begin{proof}[Proof of Theorem \ref{t:AFricci}]
First we reduce the theorem to the case in which $M^n$ has a single end. If $(M^n,g)$ has two or more ends, then one may produce a line in $M^n$ traversing two of its ends and apply the Cheeger-Gromoll Splitting Theorem \cite{CheegerGromoll} to show that $(M^n,g)$ splits isometrically as a product with a compact manifold. This contradicts the asymptotically flat condition and so we may assume $M^n$ has a single end.

In dimension 2 this result follows from the Gauss-Bonnet formula. More precisely, when applied to the interior of a large coordinate circle, we find that its Euler characteristic is no less than 1, since the total geodesic curvature of its boundary curve is $2\pi+ O(\rho^{-q})$. Then, using connectivity, the Euler characteristic must be exactly 1. It follows that the total Gaussian curvature is zero, and hence the manifold is flat, which yields the desired conclusion.

Now assume that $n\geq 3$. Fix a point $p\in M^n$ and let $v>0$ be the corresponding Green's function, which is to say $\Delta v=(2-n)\omega_{n-1}\delta_p$ and $v\to0$ as $\rho\rightarrow \infty$, where $\omega_{n-1}$ is the volume of the unit sphere $S^{n-1}\subset\R^n$.
Since $M^n$ is asymptotically flat, it is well-known that $v=\rho^{2-n}+O_{2}(\rho^{2-n-q})$, see for instance \cite[Corollary A.38]{Lee}.  
Define the function $u=v^{\frac1{2-n}}$.
Away from the point $p$, the function $u$ solves 
\begin{align}
    \Delta u=(n-1)|\nabla u|^2u^{-1}.
\end{align} 
Moreover, we have by Bochner's formula
\begin{align}\label{eq:3.1}
\begin{split}
   & \frac 12\Delta((|\nabla u|^2-1)u^{2-n})\\
   =&\frac12u^{2-n}\Delta|\nabla u|^2
   -\frac12(|\nabla u|^2-1) \Delta v
   -2(n-2)u^{1-n}|\nabla u|^2\nabla_{\nu\nu}u \\
   =&u^{2-n}\Ric(\nabla u,\nabla u)+u^{2-n}|\nabla^2u|^2+u^{2-n}\langle \nabla \Delta u,\nabla u\rangle-2(n-2)\nabla_{\nu\nu}u|\nabla u|^2u^{1-n}.
   \end{split}
\end{align}
Furthermore
\begin{align}\label{eq:3.2}
  \langle \nabla \Delta u,\nabla u\rangle=2(n-1)|\nabla u|^2u^{-1}\nabla_{\nu\nu}u-(n-1)|\nabla u|^4u^{-2},
\end{align}
and
\begin{align}\label{eq:3.3}
\begin{split}
    &|\nabla^2u-|\nabla u|^2u^{-1}g+u^{-1}\nabla u\otimes\nabla u|^2\\
    =&
    |\nabla^2u|^2+n|\nabla u|^4u^{-2}+|\nabla u|^4u^{-2}-2|\nabla u|^2u^{-1}\Delta u+2|\nabla u|^2u^{-1}\nabla_{\nu\nu}u-2|\nabla u|^4u^{-2}
    \\
    =&|\nabla^2u|^2
    +2|\nabla u|^2u^{-1}\nabla_{\nu\nu}u
    -(n-1)|\nabla u|^4u^{-2}.
    \end{split}
\end{align}

Let $\varepsilon>0$ and consider the domain $M_{\varepsilon}^{n}$ bounded between a geodesic sphere $S_{\varepsilon}^{n-1}$ of radius $\varepsilon$ around $p$, and a coordinate sphere $\mathcal{S}_{1/\varepsilon}^{n-1}$ of radius $\frac{1}{\varepsilon}$ in the asymptotic end. Then integrating \eqref{eq:3.1} by parts and applying \eqref{eq:3.2} and \eqref{eq:3.3} produces
\begin{align}
\begin{split}
    &\int_{M_{\varepsilon }^n}u^{2-n}\left(\big|\nabla^2u-|\nabla u|^2u^{-1}g+u^{-1}\nabla u\otimes\nabla u \big|^2+\Ric(\nabla u,\nabla u)\right)dV\\
    =&\int_{M_{\varepsilon}^n}\frac 12\Delta((|\nabla u|^2-1)u^{2-n})dV
    \\
    =&\int_{S_\varepsilon^{n-1}}\left( u^{2-n}|\nabla u|\partial_{\mathbf n}|\nabla u|-\frac{n-2}2u^{1-n}(|\nabla u|^2-1) \partial_{\mathbf n}u    \right)dA\\
   & +\int_{\mathcal{S}_{1/\varepsilon}^{n-1}}\left( u^{2-n}|\nabla u|\partial_{\mathbf n}|\nabla u|-\frac{n-2}2u^{1-n}(|\nabla u|^2-1) \partial_{\mathbf n}u    \right)dA.
   \end{split}
\end{align}
where $\mathbf n$ denotes the unit outer normal to $M_{\varepsilon}^n$.
Observe that $u=\rho+o_2(\rho)$ as $\rho\rightarrow\infty$ from the discussion above, and $u=r+ o_2(r)$ as $r\rightarrow 0$ where $r(x)=d(x,p)$ is the distance to $p$, see for instance  \cite[Theorem 2.4]{MRS}. 
Taking $\varepsilon\to0$ we find that
\begin{align} \label{hessian u}
    \int_{M^n}u^{2-n}\left(\big|\nabla^2u-|\nabla u|^2u^{-1}g+u^{-1}\nabla u\otimes\nabla u \big|^2+\Ric(\nabla u,\nabla u)\right)dV=0.
\end{align}

The nonnegative Ricci curvature assumption allows us to conclude that
\begin{equation}\label{jiihoi}
\nabla^2u-|\nabla u|^2u^{-1}g+u^{-1}\nabla u\otimes\nabla u =0
\end{equation}
away from the point $p$, which in turn implies that $\nabla^2u^2=2|\nabla u|^2g$. 
In particular, $\nabla |\nabla u|=0$ which according to the asymptotics shows that $|\nabla u|\equiv 1$. The manifold then splits topologically 
$M^n \setminus \{p\}=(0,\infty)\times \Sigma^{n-1}$, and the metric may be decomposed as
\begin{equation}
g=du^2+g_{u},
\end{equation}
where $g_{u}$ is the induced metric on level sets $\Sigma_u$.
Using \eqref{jiihoi}, the second fundamental form $II_u$ of these level sets is given by
\begin{equation}
    \frac{1}{2}\partial_u g_{u}=II_u=\frac{\nabla^2u|_{\Sigma_u}}{|\nabla u|}=u^{-1}g_{u}.
\end{equation}
It follows that $g_{u}=u^2g_{\Sigma}$ for some metric $g_{\Sigma}$ on $\Sigma^{n-1}$.
The structure of the metric indicates that $u(x)=r(x)$ is the distance function from $p$. It follows that $(\Sigma^{n-1},g_{\Sigma})$ is isometric to the round unit sphere, yielding the desired result.
\end{proof}





\subsection{A Bonnet-Myers Theorem with boundary}\label{SS:RicciIncomplete}

The proof of Theorem \ref{Meyer boundary} relies on the following fundamental integral formula for spacetime harmonic functions.

\begin{lemma}\label{L:Meyer integral formula}
Let $(M^n,\partial_\pm M^n,g)$, $n\geq 3$ be a Riemannian band and let $f\in\mathrm{Lip}(M^n)$.
Suppose that $u$ is a spacetime harmonic function solving the boundary value problem \eqref{e:bandspacetimeharmoniceq}.
Then 
\begin{align}\label{eq:Meyer integral formula}
\begin{split}
        &\int_{\partial_- M^n}|\nabla u|^{2-\frac{n}{n-1}}(n(n-2)f-H)dA
     -\int_{\partial_+M^n} |\nabla u|^{2-\frac{n}{n-1}}(n(n-2)f+H)dA\\
    \ge&
        \int_{M^n} 
       |\nabla u|^{-\frac n{n-1}} \left(
        |\overline \nabla^2u|^2-\frac{n}{n-1}|\nabla |\nabla u|+f\nabla u|^2
        \right)dV\\
&+
\int_{M^n}|\nabla u|^{-\frac n{n-1}}\left(\frac{n^2(n-2)^2}{n-1}f^2|\nabla u|^2+\Ric(\nabla u,\nabla u)
    -n(n-2)|\nabla u|\langle \nabla f,\nabla u\rangle
    \right)dV.
 \end{split}
\end{align}

\end{lemma}

\begin{remark}\label{Rm:Meyer integral formula}
Let $p\in M^n$ be such that $\nabla u(p)\ne0$, and
let $\{e_i\}_{i=1}^n$ be an orthonormal frame at $p$ with $e_n=\nu=\tfrac{\nabla u}{|\nabla u|}$.
Then at this point
\begin{align}
\begin{split}
   &|\overline \nabla^2u|^2-\frac{n}{n-1}|\nabla |\nabla u|+f\nabla u|^2
   \\ \ge
   &\sum_{i=1}^{n-1}|\overline{\nabla}_{ii} u|^2-\frac{1}{n-1}(\overline{\nabla}_{\nu\nu}u)^2
   +\frac{n-2}{n-1}\sum_{i=1}^n\sum_{j=i+1}^n|\overline\nabla_{ij}u|^2\\
   =&\sum_{i=1}^{n-1}|\overline{\nabla}_{ii} u|^2-\frac{1}{n-1}\left(\sum_{i=1}^{n-1}\overline{\nabla}_{ii} u\right)^2 
    +\frac{n-2}{n-1}\sum_{i=1}^n\sum_{j=i+1}^n|\overline\nabla_{ij}u|^2\\
       =&
       \frac{1}{2n-2}\sum_{i\ne j\le n-1} (\nabla_{ii}u-\nabla_{jj}u)^2
     +\frac{n-2}{2n-2}\sum_{i\ne j}|\nabla_{ij}u|^2.
             \end{split}
\end{align}
If $\nabla u=0$ and $|\nabla u|$ is differentiable (which holds almost everywhere), then $|\nabla |\nabla u|+f\nabla u|^2=0$. Thus, the second line of \eqref{eq:Meyer integral formula} is nonnegative.
\end{remark}

\begin{proof}
Proposition \ref{p:existence} implies that $u\in C^{2,\beta}(M^n)$ for some $\beta\in(0,1)$ so that $|\nabla u|$ is Lipschitz. 
Consequently, $|\nabla u|$ is differentiable outside a set of measure zero by Rademacher's Theorem, and hence $|\nabla u|$ is $W^{1,p}(M^n)$ for any $p>1$. 
It follows that $u\in W^{3,p}(M^n)$ by elliptic regularity.

Set $\alpha=\frac{n-2}{n-1}$ and for $\varepsilon>0$ consider $\psi_\varepsilon=\sqrt{|\nabla u|^2+\varepsilon^2}$. With the help of
of Bochner's formula we compute
\begin{align}
\begin{split}
    \Delta\psi_\varepsilon^\alpha=&    \frac{\alpha}{2}\psi_\varepsilon^{\alpha-2}\Delta \psi_\varepsilon^2
    +\frac{\alpha}{2}\left(\frac{\alpha}{2}-1\right)\psi_\varepsilon^{\alpha-4}|\nabla \psi_\varepsilon^2|^2
    \\=&
    \alpha \psi_\varepsilon^{\alpha-2}\left(|\nabla^2 u|^2+\Ric(\nabla u,\nabla u)+\langle \nabla u,\nabla \Delta u \rangle\right)
    +\frac{\alpha}{2}\left(\frac{\alpha}{2}-1\right)\psi_\varepsilon^{\alpha-4}|\nabla |\nabla u|^2|^2
    \label{Lp gradient u}
    \\ = &
    \alpha\psi_\varepsilon^{\alpha-2}(|\nabla^2 u|^2+\Ric(\nabla u,\nabla u)
    +\langle \nabla u,\nabla \Delta u \rangle+(\alpha-2)\psi_\varepsilon^{-2}|\nabla |\nabla u||^2|\nabla u|^2).
    \end{split}
    \end{align}
Using the spacetime Hessian identity
    \begin{align}
    \begin{split}
    |\overline \nabla^2u|^2
    :=|\nabla^2u+f|\nabla u|g|^2
    =|\nabla^2u|^2+nf^2|\nabla u|^2+2f|\nabla u|\Delta u
    =|\nabla^2u|^2-nf^2|\nabla u|^2,
    \end{split}
    \end{align}
as well as
\begin{equation}
    |\nabla|\nabla u||^2=\big|\nabla|\nabla u|+f\nabla u\big|^2
    -f^2|\nabla u|^2-2f\langle \nabla |\nabla u|,\nabla u\rangle,
\end{equation}
produces
\begin{align}\label{eq:L3.3proof}
\begin{split}
    \alpha^{-1}\Delta\psi_\varepsilon^\alpha=& \psi_\varepsilon^{\alpha-2} 
    \left(
    |\bar{\nabla}^2 u|^2
    +nf^2|\nabla u|^2
    +\Ric(\nabla u,\nabla u)
     +\langle \nabla u,\nabla \Delta u \rangle
    \right)\\
    &+  (\alpha-2)\psi_\varepsilon^{\alpha-4}
   |\nabla u|^2
    \left(\big|\nabla|\nabla u|+f\nabla u\big|^2
    -f^2|\nabla u|^2-2f\langle \nabla |\nabla u|,\nabla u\rangle\right)
   .
   \end{split}
\end{align}
To analyze the last term of the first line in \eqref{eq:L3.3proof}, we calculate this expression in two different ways, namely
\begin{align}\label{eq:L3.3proof2}
\begin{split}
    \psi_\varepsilon^{\alpha-2}\langle \nabla u,\nabla \Delta u\rangle
    =&
    \textnormal{div}(\Delta u\cdot \psi_\varepsilon^{\alpha-2}\nabla u)
    -(\Delta u)^2\psi_\varepsilon^{\alpha-2}
    -\Delta u\langle \nabla\psi_\varepsilon^{\alpha-2},\nabla u\rangle
    \\=&
    \textnormal{div}(\Delta u\cdot \psi_\varepsilon^{\alpha-2}\nabla u)
    -n^2f^2 |\nabla u|^{2}\psi_\varepsilon^{\alpha-2}+
    n(\alpha-2)f|\nabla u|^2\psi_\varepsilon^{\alpha-4}\langle \nabla|\nabla u|,\nabla u\rangle
    \end{split}
\end{align}
and
\begin{equation}\label{eq:L3.3proof3}
\psi_\varepsilon^{\alpha-2}\langle \nabla u,\nabla \Delta u\rangle
=-n\psi_\varepsilon^{\alpha-2}|\nabla u|\langle \nabla f,\nabla u\rangle
-nf\psi_\varepsilon^{\alpha-2}\langle \nabla u,\nabla|\nabla u|\rangle.    
\end{equation}
Using a parameter $\lambda $ to interpolate between \eqref{eq:L3.3proof2} and \eqref{eq:L3.3proof3}, and inserting this into \eqref{eq:L3.3proof} gives
\begin{align}
\begin{split}
    &\alpha^{-1}\psi_\varepsilon^{2-\alpha}\Delta\psi_\varepsilon^\alpha
\\ =&|\overline{\nabla}^2 u|^2+nf^2|\nabla u|^2
+\Ric(\nabla u,\nabla u)
\\&+(\alpha-2)\left(\big|\nabla|\nabla u|
+f\nabla u\big|^2
    -f^2|\nabla u|^2
    -2f\langle \nabla |\nabla u|,\nabla u\rangle\right)\psi_\varepsilon^{-2}|\nabla u|^2
    \\ & -n^2\lambda f^2|\nabla u|^2
    +n(\alpha-2)\lambda f|\nabla u|^2\psi_\varepsilon^{-2}\langle \nabla|\nabla u|,\nabla u\rangle
    + \lambda \psi_\varepsilon^{2-\alpha}\textnormal{div}(\Delta u\cdot \psi_\varepsilon^{\alpha-2}\nabla u)
    \\ &-n(1-\lambda)f\langle \nabla u,\nabla|\nabla u|\rangle 
    -n(1-\lambda)|\nabla u|\langle \nabla f,\nabla u\rangle.
    \end{split}
\end{align}
Setting $\lambda=3-n$ yields
\begin{align}
\begin{split}
&-2(\alpha-2)+n\lambda(\alpha-2)-n(1-\lambda)=0,
\end{split}
\end{align}
and thus the coefficients of $\langle\nabla |\nabla u|,\nabla u\rangle$ incur many cancellations.
Then a further reorganization of terms produces the formula
\begin{align}\label{e:riccimainpointwise}
\begin{split}
    &\alpha^{-1}\psi_\varepsilon^{2-\alpha}\Delta\psi_\varepsilon^\alpha-(3-n)\psi_\varepsilon^{2-\alpha}\textnormal{div}(\Delta u\cdot \psi_\varepsilon^{\alpha-2}\nabla u)
\\ =&|\overline{\nabla}^2 u|^2-\frac{n}{n-1}\big|\nabla|\nabla u|
+f\nabla u\big|^2
\\&+\frac{n^2(n-2)^2}{n-1}f^2|\nabla u|^2
    -n(n-2)|\nabla u|\langle \nabla f,\nabla u\rangle
    +\Ric(\nabla u,\nabla u)
    \\
   &+\frac {n\varepsilon^2}{n-1}
\left(
\big|\nabla|\nabla u|
+f\nabla u\big|^2
-(n-1)(n-2)f\langle \nabla |\nabla u|,\nabla u\rangle 
    -f^2|\nabla u|^2
   \right)
   \psi_\varepsilon^{-2}.
   \end{split}
\end{align}

Finally, we multiply \eqref{e:riccimainpointwise} by $\alpha \psi_\varepsilon^{\alpha-2}$, integrate by parts, and study the boundary term. 
Let $\mathbf{n}$ be the outer normal vector on $\partial M^n$, noting that $\mathbf{n}=-\nu$ on $\partial_-M^n$ and $\mathbf{n}=\nu$ on $\partial_+ M^{n}$, then
\begin{align}
\begin{split}
&\int_{M^n} \Delta \psi_\varepsilon^{\frac{n-2}{n-1}}+\frac{(n-3)(n-2)}{n-1}\textnormal{div}(\Delta u\cdot \psi^{-\frac{n}{n-1}}\nabla u) dV
  \\=&\int_{\partial M^n} \partial_{\mathbf{n}}\psi_\varepsilon^{\frac{n-2}{n-1}}+\frac{(n-3)(n-2)}{n-1}\Delta u\cdot \psi_\varepsilon^{-\frac{n}{n-1}}\partial_{\mathbf{n}} udA
  \\=&\int_{\partial M^n}\frac{n-2}{n-1}\psi_\varepsilon^{-\frac{n}{n-1}}|\nabla u|\langle\mathbf{n},\nu\rangle\left(\nabla_{\nu\nu}u+(n-3)\Delta u\right)dA
    \\=&\int_{\partial_- M^n}\frac{n-2}{n-1} \psi_\varepsilon^{-\frac{n}{n-1}}|\nabla u|^2(n(n-2)f-H)dA
    \\ &-\int_{\partial_+M^n}\frac{n-2}{n-1} \psi_\varepsilon^{-\frac{n}{n-1}}|\nabla u|^2(n(n-2)f+H)dA.
    \end{split}
\end{align}
In order to aid passage to the limit $\varepsilon\to0$, observe that
since $\alpha< 2$, $\psi_\varepsilon\ge |\nabla u|$, and $\psi_\varepsilon\ge \varepsilon$ we have
\begin{align}
\begin{split}
   &|\nabla|\nabla u|+f\nabla u|^2-(n-1)(n-2)f\langle\nabla |\nabla u|,\nabla u\rangle
    -f^2|\nabla u|^2
    \\ =& \bigg|\nabla|\nabla u|-\frac{n(n-3)}{2}f\nabla u\bigg|^2-\frac{n^2(n-3)^2}{4}f^2|\nabla u|^2
    \\ \ge &-\frac{n^2(n-3)^2}{4}f^2\psi_\varepsilon^{4-\alpha}\varepsilon^{-(2-\alpha)}.
    \end{split}
\end{align}
Therefore, the following integral formula holds
\begin{align}\label{eq:lemma3.4proof}
\begin{split}
        &\int_{\partial_- M^n}\frac{n-2}{n-1} \psi_\varepsilon^{-\frac{n}{n-1}}|\nabla u|^2(n(n-2)f-H)dA
    \\ &-\int_{\partial_+M^n}\frac{n-2}{n-1} \psi_\varepsilon^{-\frac{n}{n-1}}|\nabla u|^2(n(n-2)f+H)dA\\
    \ge&
        \int_{M^n}\frac{n-2}{n-1} \psi_{\varepsilon}^{-\frac n{n-1}}\left(
     |\overline \nabla^2u|^2-\frac{n}{n-1}|\nabla |\nabla u|+f\nabla u|^2
        \right)dV
        -\int_{M^n}c(n)\varepsilon^{\frac{n-2}{n-1}}f^2dV
        \\
&+
\int_{M^n}\frac{n-2}{n-1} \psi_{\varepsilon}^{-\frac n{n-1}}
\left(
\frac{n^2(n-2)^2}{n-1}f^2|\nabla u|^2
    -n(n-2)|\nabla u|\langle \nabla f,\nabla u\rangle+\Ric(\nabla u,\nabla u)
    \right)dV,
    \end{split}
\end{align}
where $c(n)$ is a nonnegative constant depending on $n$.
Since $f^2$ is integrable, the term involving $c(n)$ converges to zero. Furthermore, the limit may be taken inside the integral of each term in the last line of \eqref{eq:lemma3.4proof} using the dominated convergence theorem.
Lastly, the remaining term involving Hessian components has a nonnegative integrand by Remark \ref{Rm:Meyer integral formula}, and thus may be treated with Fatou's lemma, yielding the desired result.
\end{proof}

\begin{proof}[Proof of Theorem \ref{Meyer boundary}]
Suppose that the width of $(M^n,\partial_\pm M^n,g)$ satisfies
\begin{equation}
w:=d(\partial_-M^n,\partial_+M^n)\ge \arctan \frac{H_-}{n-1}+\arctan\frac{H_+}{n-1}.
\end{equation}
Define a $1$-variable Lipschitz function
\begin{equation} \label{f Meyers}
\tilde{f}(\tau)=
    \begin{cases}
    \frac{n-1}{n(n-2)}\tan(\tau-\arctan\frac{H_-}{n-1})  &\text{if } \tau\le \tilde{\tau},
    \\ 
    L(\tau) &\text{if } \tau\ge \tilde{\tau},
    \end{cases}
    \end{equation}
where $\tilde{\tau}=\max\{\arctan\frac{H_+}{n-1}+\arctan\frac{H_-}{n-1},\frac{\pi}{4}+\frac{1}{2}\arctan\frac{H_-}{n-1}\}$ and $L(\tau)$ is the unique linear function that makes $\tilde{f}(\tau)$ differentiable. 
From equation \eqref{f Meyers},
$\tilde{f}(\tau)$ is strictly increasing with respect to $\tau$.  Furthermore, $\tilde{f}(\tau)$ satisfies 
\begin{equation}\label{tilde f ode}
        \frac{n^2(n-2)^2}{n-1}\tilde{f}^2-n(n-2)\tilde{f}'+n-1\ge 0.
\end{equation}
Denote $f(x)=\tilde{f}(r(x))$, where $r(x)=d(x,\partial M^n_-)$ is the distance function to $\partial_- M^n$, and note that $f$ satisfies $n(n-2)f\le H$ on $\partial_-M^n$ as well as $n(n-2)f\ge -H$ on $\partial_+ M^n$.
Let $u$ be the corresponding spacetime harmonic function in Proposition \ref{p:existence} with $c_-=0$, $c_+=1$.
By Lemma \ref{L:Meyer integral formula}, Remark \ref{Rm:Meyer integral formula}, and equation \eqref{tilde f ode} we have 
\begin{align}   \label{error meyer}
\begin{split}
    0 \ge &
    \int_{\{\nabla u\ne0\}}|\nabla u|^{-\frac n{n-1}}\left(
        \frac{1}{2n-2}\sum_{i\ne j\le n-1} (\nabla_{ii}u-\nabla_{jj}u)^2
     +\frac{n-2}{2n-2}\sum_{i\ne j}|\nabla_{ij}u|^2
        \right)dV \\
        &        +\int_{M^n}|\nabla u|^{2-\frac n{n-1}}(\Ric(\nu,\nu)-(n-1))dV,
\end{split}
\end{align}
where we have used the Cauchy-Schwarz inequality $\langle \nabla f,\nabla u\rangle\le |\nabla f||\nabla u|= \tilde f'|\nabla u|$ and where $\{e_i\}_{i=1}^n$ is an orthonormal frame with $e_n=\nu$.
Therefore, at points where $\nabla u\ne0$ we have
\begin{align}
    \nabla_{ij}u=0\;   \textnormal{\;for\;} i\neq j ,
    \label{off diagonal}
    \quad\quad&
    \nabla_{ii}u=\nabla_{jj}u\quad\text{for $i,j\le n-1$},\\
    \Ric(\nu,\nu)=n-1 ,
    \quad\quad
   & \langle \nabla f,\nu\rangle=|\nabla f|
    \label{parallel} \textnormal{\; where $f$ is differentiable.}
\end{align}

Our first goal is to show that $\nabla u$ is nowhere vanishing.
Equation \eqref{off diagonal} implies that $\nabla_i|\nabla u|=0$, $i=1,2,\dots,n-1$, whenever $\nabla u\ne0$.
Therefore, connected components of level sets consist entirely of either regular or critical points. 
Let $t_0$ be the minimal value such that $\Sigma_{t_0}=\{u=t_0\}$ contains a critical point, and note that by the Hopf Lemma $t_0>0$.
Denote by $\Sigma'_{t_0}$ the connected component of $\Sigma_{t_0}$ which consists of critical points.
Take a small geodesic ball $B_{r}(x_0)\subset \{u<t_0\}$ such that $\partial B_r(x_0)\cap \Sigma'_{t_0}\neq \emptyset$. 
Then $u$ obtains its maximum on $\partial B_r(x_0)\cap\Sigma'_{t_0}$,
which contradicts the Hopf Lemma.
Hence, $\nabla u\ne0$ everywhere and $M^n =[0,1]\times \Sigma_0$ topologically splits as a product. This allows us to write $g=\frac{du^2}{|\nabla u|^2}+g_u$ where $g_u$ are the induced metrics on level sets.

According to \eqref{off diagonal}, we have $d(du/|\nabla u|)=0$. Since the de Rham class defined by $du/|\nabla u|$ is trivial, there exists a new coordinate $s$ with $s=0$ corresponding to $\partial_- M^n$, and such that $ds=du/|\nabla u|$. In fact, $s$ agrees with the distance function $r$ to $\partial_- M^n$. Furthermore, as in the proof of Theorem \ref{t:AFricci}, using \eqref{off diagonal} and the spacetime harmonic equation $\Delta u=-nf|\nabla u|$ produces
\begin{equation}
    \partial_u g_u=\frac{2\nabla^2 u|_{\Sigma_u}}{|\nabla u|}=
    \frac{2\Tr_{g_u}(\nabla^2u|_{\Sigma_u})}{(n-1)|\nabla u|} g_u
    =-\frac{2(\nabla_{\nu\nu}u+nf|\nabla u|)}{(n-1)|\nabla u|}g_u.
\end{equation}
Next, observe that $\nabla_{\nu\nu}u=\partial_\nu|\nabla u|$ is a function of $u$ alone, and the same is true for $f$ since $s=r$. It follows that there is a function $\psi=\psi(r)$ so that $g=dr^2+\psi^2 (r)g_0$ for some metric $g_0$ independent of $r$.

To determine $\psi$, we use standard calculations for the curvature of warped product manifolds, and \eqref{parallel}, to find
\begin{align}
   n-1 =\Ric(\nu,\nu)=&\Ric(\partial_r,\partial_r)
    = -(n-1)\frac{\psi''}{\psi}(r) \label{psi ode}.
\end{align}
The function $\psi$ is also subject to the boundary conditions 
\begin{align}
H_-=(n-1)\frac{\psi'}{\psi}(0),\quad\quad -H_+=(n-1)\frac{\psi'}{\psi}(w).
\end{align}
Therefore  $\psi(r)=a\cos(r-\arctan\frac{H_-}{n-1})$ for some $a>0$, and
\begin{equation}\label{e:ricciweq}
w=\arctan \frac{H_-}{n-1}+\arctan\frac{H_+}{n-1}.
\end{equation}
A computation with the Gauss equations shows that $\Ric (g_0)\ge (n-2)a^2g_0$. Finally, setting $\theta=r-\arctan\frac{H_-}{n-1}$ and $g_{\Sigma}=a^2 g_0$ gives the stated result. 
\end{proof}

\subsection{Cheng's rigidity and proof of Main Theorem \ref{Meyer boundary cor}}

In this section we prove Main Theorem \ref{Meyer boundary cor}, the Bonnet-Myers Theorem, and show Cheng's rigidity. 
The latter two will be treated as a consequence of the main theorem.
The arguments for this result follow a similar logic to the proof of Theorem \ref{Meyer boundary}, though the fact that $M^n$ is open and incomplete poses enormous technical difficulties.
Ideally, we would like to solve the spacetime harmonic equation $\Delta u=-nf|\nabla u|$ on $M^n$ with $f\to\pm \infty$ at $E_\pm$, where $f$ satisfies the differential inequality \eqref{tilde f ode}.
Since the incompleteness of $M^n$ makes this infeasible, we instead solve the spacetime harmonic equation on compact bands $\{M^n_\delta\}_{\delta>0}$ exhausting $M^n$. The mean curvature of $\partial M^n_\delta$ is completely uncontrolled, and adjusting $f$ to overcome this comes at the cost of slightly violating the inequality in \eqref{tilde f ode} in a fixed compact set. 
Decreasing this violation while running through the exhausting bands, and passing to a limit gives rise to a spacetime harmonic function satisfying a fundamental integral identity derived from Lemma \ref{L:Meyer integral formula}.

\begin{proof}[Proof of Main Theorem \ref{Meyer boundary cor}]
Let $\Sigma^{n-1}\subset M^n$ be a closed hypersurface separating $M^n$ into two connected components $M^n_\pm$ where $E_\pm$ is contained in $M^n_{\pm}$.
Suppose that $w_-+w_+\geq\pi$ where we define $w_{\pm}$ as the minimum $\min(\pi,d(\Sigma^{n-1},E_\pm))$, though we will soon see that $w_\pm<\pi$. Consider the signed distance to $\Sigma$ given by $\varrho(x)=\pm d(x,\Sigma)$ when $x\in M^n_\pm$.
Let $\delta>0$ be a small parameter and consider the band $({\widetilde{ M^n_\delta}},\partial_\pm \widetilde{ M_\delta^n},g)$ given by
\begin{align}
    \widetilde {M_\delta^n}=
    \{x\in M^n:\varrho(x)\in[-w_-+\delta,w_+-\delta]\}
\end{align}
where $\partial_\pm \widetilde{M_\delta^n}$ is assigned based on the sign of $\varrho$. As an application of Lemma \ref{l:compactification}, we have that $\widetilde{M_\delta^n}$ is compact. To see this, first note that if $w_- =w_+$ then no further argument is needed. Now suppose without loss of generality that $w_-$ is smaller than $w_+$, and apply Lemma \ref{l:compactification} to the closed $w_- -\delta$ neighborhood $N_1$ of $\Sigma$, showing that it is compact. This allows us to apply extension construction of \cite{Morrow} to find a metric $g'$ on $M^n$ which only differs from $g$ on $M^n_-\setminus N_1$ and makes $d_{g'}(\Sigma^{n-1},E_-)=\infty$. Applying Lemma \ref{l:compactification} to the closed $w_+-\delta$ neighborhood $N_2$ of $\Sigma^{n-1}$ in this new metric shows that $N_2$ is compact. We then conclude that the closed set $\widetilde{M^n_\delta}\subset N_2$ is compact.

Compactness of $\widetilde{M}^n_\delta$ is the crucial property which will allow us to apply Theorem \ref{Meyer boundary} and its proof. First, we will make a small perturbation of $\widetilde{M}^n_\delta$ to a new band $M^n_\delta$ such that $\partial M_\delta$ is smooth. Using the version of Sard's Theorem in \cite{Rifford}, almost all $\delta'$ are such that $\partial \widetilde{M^n_{\delta'}}$ is Lipschitz. We may consider a smooth hypersurface homologous and arbitrarily close to $\partial \widetilde{M^n_{\delta'}}$ by, for instance, running mean curvature flow for a short time \cite{EckerHuisken}. Denote by $M^n_\delta$ the region bounded by these smooth hypersurfaces, which we may assume has width at least $(w_-+w_+)-3\delta$. Furthermore, Theorem \ref{Meyer boundary} implies that the width of $M_\delta^n$ cannot exceed $\pi$. Since these inequalities hold for all $\delta>0$ we conclude that $w_-+w_+=\pi$.

Let us denote with $H_\delta=\sup_{\partial M_\delta^n} |H|$ where $H$ is the mean curvature of $\partial M^n_\delta$. 
Slightly adjusting the construction in Lemma \ref{construct f}
with $a=b=\frac{n(n-2)}{n-1}$ and appropriate $\varepsilon$, we find a function $f_\delta\in \mathrm{Lip}(M_\delta^n)$ satisfying the following
\begin{align}
\begin{split}
    \frac{n^2(n-2)^2}{n-1}f^2_\delta(x)+n-1-n(n-2)|\nabla f_\delta(x)|\ge-C\delta & \quad\text{   in }\mathcal{B},\\
    \frac{n^2(n-2)^2}{n-1}f^2_\delta(x)+n-1-n(n-2)|\nabla f_\delta(x)|\ge 0 &
   \quad \text{    in }M^n_\delta\setminus \mathcal{B},\\
   f_\delta \le- \frac1{n(n-2)} H_\delta\quad\text{ on }\partial_- M_\delta^n,\quad f_\delta\ge\frac1{n(n-2)} H_\delta&\quad \text{ on } \partial_+ M_\delta^n,
\end{split}
\end{align}
where $\mathcal{B}=\{-\frac{w_-}{2}\le \varrho(x)\le \frac{w_+}{2}\}$ and $C$ is a constant independent of $\delta$. Applying Proposition \ref{p:existence} with $f_\delta$, we may consider a spacetime harmonic function $u_\delta$ on $M^n_\delta$ with Dirichlet boundary conditions. Fix a small $\rho>0$ such that $\mathcal{B}\subset M_\rho^n$ and a point $p\in M_\rho^n$. By scaling and adding a constant to $u_\delta$, we arrange for $\sup_{M^n_{\rho}}|\nabla u_\delta|=1$ and $u_\delta(p)=0$. Integrating $du_\delta$ along paths in $M^n_\rho$, we find $|u_\delta|\le \textnormal{diam}(M^n_\rho)$ on $M^n_\rho$. 

Next, we apply the integral formula Lemma \ref{L:Meyer integral formula} and use the boundary conditions of $f_\delta$ to find that
\begin{align}   \label{error meyer1}
\begin{split}
    0 \ge &
    \int_{M^n_\delta} |\nabla u_\delta|^{-\frac n{n-1}}\left(
      |\overline \nabla^2u_\delta |^2-\frac{n}{n-1}|\nabla |\nabla u_\delta|+f_\delta\nabla u_\delta|^2
        \right)dV \\
        &      +\int_{M^n_\delta}|\nabla u_\delta|^{-\frac n{n-1}}(\Ric(\nabla u_\delta,\nabla u_\delta)-(n-1)|\nabla u_\delta|^2)dV
   \\& -C\delta\int_{\mathcal B}(n-2)|\nabla u_\delta|^{2-\frac{n}{n-1}}dV.
\end{split}
\end{align}
The next step is to take the limit $\delta\to0$. Since $|u_\delta|$ and $|\nabla u_\delta|$ are uniformly bounded on $M_\rho^n$, standard Schauder estimates allow us find a subsequential $C^{2,\beta}$ limit, for some $\beta\in(0,1)$ of $u_\delta$ on $M_\rho^n$, which we denote by $u$. The function $u$ solves the spacetime harmonic equation with $f(x)= \frac{n-1}{n(n-2)}\tan (\varrho(x)+\frac{w_--w_+}2)$. 
Notice that $u$ is nonconstant since $\sup_{M^n_{\rho}}|\nabla u|=1$.
By Fatou's Lemma and boundedness of $\|u_\delta\|_{C^{2,\beta}(M^n_\rho)}$, we may take a liminf of \eqref{error meyer1} to find
\begin{align}
\begin{split}
    0 = &
    \int_{ M^n_{\rho}} |\nabla u|^{-\frac n{n-1}}\left(
        |\overline \nabla^2u |^2-\frac{n}{n-1}|\nabla |\nabla u|+f\nabla u|^2
        \right)dV \\
        &
        +\int_{M^n_{\rho}}|\nabla u|^{-\frac n{n-1}}(\Ric(\nabla u,\nabla u)-(n-1)|\nabla u|^2)dV
 \end{split}
\end{align}
and so the conditions \eqref{off diagonal} and \eqref{parallel} hold on $M_\rho^n$.

Since $u$ is nonconstant, the image of $u$ contains a nonempty interval. According to Sard's Theorem, we may consider a regular value $t_0$ of $u$. We claim that $u$ has no critical point within the interior of $M^n_\rho$. If this fails, there is a critical value $t_1$ which is closest value to $t_0$. By symmetry of the following argument, we will assume $t_0<t_1$. As in the proof of Theorem \ref{Meyer boundary}, the component of $\Sigma_{t_1}$ containing a critical point must consist entirely of critical points. Then we may find a closed ball in $u^{-1}([t_0,t_1])$ which intersects the critical component of $\Sigma_{t_1}$ only along the boundary of the ball, contradicting the Hopf Lemma. We conclude $u$ has no interior critical points. 

As in the proof of Theorem \ref{Meyer boundary}, $\nabla u$ is parallel to $\nabla\varrho$ and hence $\varrho$ can be considered as a function of $u$ on $M_\rho^n$. The arguments of the previous theorem show there is a constant $b$ such that
\begin{equation}\label{eq warped}
g=d\varrho^2+\frac{\cos^2(\varrho-b)}{\cos^2 b} g_\Sigma    
\end{equation}
on $M_\rho^n$, where $g_\Sigma$ is the induced metric on $\Sigma$. Since $\rho$ was arbitrary, the splitting \eqref{eq warped} holds along $M_0^n:=\{x\in M^n| \varrho(x)\in (-w_-,w_+)\}$. By rescaling $g_\Sigma$ and shifting the $\varrho$ coordinate, we can conclude that $g=d\theta^2+\cos^2(\theta)g_0$ along $M_0^n$ where $g_0$ is a metric on $\Sigma$ with $\Ric(g_0)\ge (n-2)g_0$. Since $\theta$ ranges over $(-\pi/2, \pi/2)$, no connected Riemannian $n$-manifold with at least two ends can properly and isometrically contain $(M_0^n,d\theta^2+\cos^2(\theta)g_0)$. We conclude that $M_0^n=M^n$ and the proof is complete.
\end{proof}

Next, we use Theorem \ref{Meyer boundary cor} to give a new proof of Bonnet-Myers and the Cheng rigidity result regarding complete manifolds with positive Ricci curvature and maximal diameter.

\begin{corollary}\label{c:myerscheng}
Let $(M^n,g)$, $n\geq 2$ be a complete Riemannian manifold. If $\Ric\ge (n-1)g$, then the diameter satisfies $\mathrm{diam}(M^n,g)\le \pi$. Moreover, equality occurs in the diameter estimate only if $(M^n,g)$ is isometric to the unit round sphere.
\end{corollary}

\begin{proof}
The result is classical, and so we only present arguments in dimensions $n\geq 3$ where the new proof applies. Let $p,q$ be two points in $M^n$ realizing the diameter $\mathrm{diam}(M^n,g)$. 
Fix a sufficiently small $\varepsilon>0$ so that the $\varepsilon$-sphere about $p$ is smooth. Denote this sphere by $\Sigma^{n-1}$ and notice that $d(p,\Sigma^{n-1})=\varepsilon$ and $d(q,\Sigma^{n-1})=w-\varepsilon$. Now we may apply Main Theorem \ref{Meyer boundary cor} to the open manifold $M^n\setminus \{p,q\}$ with the separating hypersurface $\Sigma^{n-1}$. The inequality $w\leq\pi$ follows.

In the case of equality $w=\pi$, Main Theorem \ref{Meyer boundary cor} tells us we have the splitting
\begin{equation}\label{aeoiwfhoia}
(M^n\setminus\{p,q\},g)=\left(\left(-\frac{\pi}{2},\frac{\pi}{2}\right)\times\Sigma^{n-1},d\theta^2+\cos^2\theta g_0\right)
\end{equation}
where $g_0$ is a scaling of the induced metric on $\Sigma^{n-1}$.
It suffices to show that $g_0$ is the unit round metric on $S^{n-1}$. Note, however, that this follows immediately from the fact that the metric \eqref{aeoiwfhoia} must extend smoothly across $p$ and $q$.
\end{proof}

\begin{remark}
Lemma \ref{L:Meyer integral formula} is the only place where we use the assumption $n\ge3$ in the proofs of Theorem \ref{Meyer boundary} and Main Theorem \ref{Meyer boundary cor}.
We would like to point out that in dimension $2$ a similar computation yields 
\begin{align}\label{Gauss-Bonnet}
   \int_{M^2}\Ric(\nu,\nu)dA+ \int_{\partial M^2}Hds\leq0.
\end{align} 
Since $\Ric(\nu,\nu)$ is Gauss curvature, $H$ is geodesic curvature, and the Euler characteristic of a connected $2$-dimensional band cannot exceed $0$, inequality \eqref{Gauss-Bonnet} is a weak version of Gauss-Bonnet's theorem.
Inequality \eqref{Gauss-Bonnet} follows by integrating $\Delta \log|\nabla u|$ for a harmonic function $u$ with constant Dirichlet boundary conditions, and applying the Bochner formula.
\end{remark}


\section{Torical Band Inequality and Rigidity}\label{S:Band}



\begin{proof}[Proof of Theorem \ref{ToricBand}] Set $w=d(\partial_- M^3,\partial_+ M^3)$ and let $\delta\in (0,\tfrac{2\pi}{3})$ be such that
$\frac{4}{3}\mathrm{arctan}(H_0/2)<\frac{2\pi}{3}-\delta$.
Define 
\begin{equation}
w_0=\min\left\{w,\tfrac{2\pi}{3}-\delta\right\},
\end{equation}
and let $\tilde{f}\in C^1(\mathbb{R}_{+})$ be the increasing function given by
\begin{equation}\label{e:toruspotential}
\tilde{f}(\tau)=
\begin{cases}
\tan\left(\frac{3}{2}(\tau-w_0 /2)\right) &\text{ for }0\leq \tau< w_0 ,\\
\frac{3}{2}\sec^2(3w_0 /4) \cdot (\tau-w_0)+\tan(3w_0 /4)&\text{ for } \tau\geq w_0 .
\end{cases}
\end{equation}
Furthermore, denote the distance function to $\partial_-M^3$ by $r(x)=d(x,\partial_- M^3)$, and set $f=\tilde{f}\circ r \in \mathrm{Lip}(M^3)$. With the help of Proposition \ref{p:existence}, we may solve the corresponding spacetime harmonic Dirichlet boundary value problem
\begin{equation}\label{eq:spacetimeharmonic}
\begin{cases}
\Delta u + 3f|\nabla u|=0 & \text{ in }M^3 , \\
u=\pm 1 & \text{ on }\partial_\pm M^3 ,
\end{cases}
\end{equation}
such that the solution $u$ lies in $C^{2,\alpha}(M^3)$.
Using the integral inequality of Lemma \ref{integralformula}, we find that 
\begin{align}\label{eq:bandint1}
\begin{split}
&4\pi\int_{-1}^1\chi(\Sigma_t)dt+2\int_{\partial_- M^3}(2f-H)|\nabla u|dA-2\int_{\partial_+ M^3}(2f+H)|\nabla u|dA\\
\geq&\int_{M^3}\left(\frac{|\overline{\nabla}^2u|^2}{|\nabla u|}+(R+6f^2)|\nabla u|-4\langle \nabla f,\nabla u\rangle\right)dV
\end{split}
\end{align}
where then mean curvature $H$ is with respect to the outward normal. Note that even though $f$ is only Lipschitz, by Radamacher's theorem its derivatives exist almost everywhere, justifying the use of $\nabla f$ in the integral expression. 

First observe that the scalar curvature lower bound, the Cauchy-Schwarz inequality, the fact that $\tilde{f}$ is an increasing function, and a straightforward computation yield
\begin{align}\label{anfighiw}
\begin{split}
(R+6f^2)|\nabla u|-4\langle \nabla f,\nabla u\rangle=&\left[ 6+6\tilde{f}^2 (r) -4\tilde{f}'(r) +(R -6)\right]|\nabla u|\\
&+ 4\tilde{f}'(r) (|\nabla u|-\langle\nabla r,\nabla u\rangle)\\
\geq& (R -6)|\nabla u|+ 4\tilde{f}'(r) (|\nabla u|-\langle\nabla r,\nabla u\rangle)\\
\geq &0,
\end{split}
\end{align}
away from a set of measure zero. 
Since $H_2(M;\mathbb{Z})$ contains no spherical classes it holds that $\chi(\Sigma_t)\leq 0$ for all regular level sets of $u$. 
This last statement is a consequence of the maximum principle, which implies that all components of $\Sigma_t$ are homologically nontrivial. Combining this observation with \eqref{eq:bandint1} and \eqref{anfighiw}, we find that the boundary term of \eqref{eq:bandint1} must be nonnegative.
Inspecting the boundary term, we obtain
\begin{equation}\label{e:kak}
2f-H\leq 2\tilde{f}(0)+H_0 =-2\tan(3w_0/4)+H_0 \quad\quad\text{ on }\partial_- M^3,
\end{equation}
\begin{equation}\label{kakfkafk}
2f+H\geq 2\tilde{f}(w)-H_0 \geq 2\tan(3w_0/4)-H_0  \quad\quad\text{ on }\partial_+ M^3, 
\end{equation}
and therefore
\begin{equation}
w_0\leq \frac{4}{3}\mathrm{arctan}(H_0/2).
\end{equation}
If $w_0=\tfrac{2\pi}{3}-\delta$, then a contradiction is reached in light of the choice of $\delta$. It follows that $w_0=w$, and the desired inequality \eqref{eq:widthinequality} is established. Note also that this implies $H_0>0$. 

We will now address the rigidity statement of the theorem.
Assume that equality is achieved in \eqref{eq:widthinequality}. Then the arguments above give rise to
\begin{equation}\label{fjjjj}
|\nabla^2 u+fg|\nabla u||=|\overline{\nabla}^2 u|=0,\quad\quad\quad R =6,\quad\quad\quad w=\frac{4}{3}\mathrm{arctan}(H_0/2),
\end{equation}
\begin{equation}\label{fjjjjj}
\langle\nabla r,\nabla u\rangle =|\nabla u| \text{ }\text{ a.e.},\quad\quad H=-H_0 \text{ on }\partial M,\quad\quad
\chi(\Sigma_t)=0\text{ for regular values } t\in[-1,1].
\end{equation}
The first equation of \eqref{fjjjj} implies that whenever $|\nabla u|\neq 0$ we have
\begin{equation}
|\nabla \log|\nabla u||\leq\frac{|\nabla^2 u|}{|\nabla u|}= 3\sup_{M^3}|f|\leq C_0,
\end{equation}
where the constant $C_0$ depends on the diameter of $M^3$ and $H_0$. Applying this estimate along curves emanating from $\partial M^3$, where $|\nabla u|>0$ by the Hopf lemma, shows that $|\nabla u|>0$ on all of $M^3$.
The manifold is then topologically a cylinder $M^3 \cong [-1,1] \times T^2$, where we have also used the zero Euler characteristic property of level sets in \eqref{fjjjjj}. Moreover, the metric splits as
\begin{equation}\label{aknhiiowoh}
g=\frac{du^2}{|\nabla u|^{2}} +g_u,
\end{equation}
where $g_u$ is a family of metrics on the 2-torus. Next, observe that if $Y$ is a vector field tangent to a level set $\Sigma_t$ then the first equation of \eqref{fjjjj} yields
\begin{equation}
Y(|\nabla u|)=\nabla^2 u\left( \frac{\nabla u}{|\nabla u|},Y \right)
=-f \langle \nabla u,Y\rangle=0,
\end{equation}
so that $|\nabla u|$ is constant on $\Sigma_t$. Furthermore, denoting $\nu=\frac{\nabla u}{|\nabla u|}$,
\begin{equation}\label{fjqjh9hyj}
\partial_u |\nabla u|=\left\langle \nabla|\nabla u|, \frac{\nabla u}{|\nabla u|^2}\right\rangle=\nabla^2 u\left( \frac{\nabla u}{|\nabla u|},\frac{\nabla u}{|\nabla u|^2} \right)=-f \langle \nu,\nu\rangle=-f.
\end{equation}

From the structure of the metric \eqref{aknhiiowoh} it is clear that the distance function from the zero level set is a function of $u$ alone, that is $r=r(u)$ with $r(0)=0$, $r\in [0,w]$, and $dr=|\nabla u|^{-1} du$.
With a slight abuse of notation, we will write $g_r=g_{u(r)}$, $\Sigma_r=\Sigma_{u(r)}$, and denote the second fundamental form of level sets by $II_r =\tfrac{\nabla^2 u}{|\nabla u|}|_{\Sigma_r}$. Then the first equation of \eqref{fjjjj} produces
\begin{equation}\label{==g0juh}
\frac{1}{2}\partial_{r}g_{r}=II_{r}=-\tilde{f}(r) g_r.
\end{equation}
Keeping in mind the formula for $\tilde{f}$, it follows that 
\begin{equation}
g_r=\cos^{\frac{4}{3}}\!\left(\frac{3}{2}(r-w/2)\right)g_0,
\end{equation}
for some metric $g_0$ on $T^2$. 
Moreover, note that \eqref{==g0juh} shows the mean curvature of level sets is given by $H_r=-2\tilde{f}(r)$. The Riccati equation and two traces of the Gauss equations then imply
\begin{align}
\begin{split}
-2\tilde{f}'=\partial_r H_r =& -|II_r|^2 -\mathrm{Ric}(\nu,\nu)\\
=&-\frac{1}{2}\left(|II_r|^2 +H_r^2 +R\right)+K\\
=&-3\tilde{f}^2 -3 +K,
\end{split}
\end{align}
where $K$ is the Gaussian curvature of $\Sigma_r$. According to the ODE satisfied by $\tilde{f}(r)$ for $r\in[0,w]$, we conclude that $K=0$ and $g_0$ is flat. Lastly, by setting $s=r-w/2\in [-w/2,w/2]$ we obtain the desired form of the metric
\begin{equation}
g=ds^2 +\cos^{\frac{4}{3}}\!\left(\tfrac{3}{2}s \right)g_0.
\end{equation}
\end{proof}


\section{The Waist Inequalities}\label{S:waist}

\subsection{Proof of Theorem \ref{thm:waist}}\label{secwaist}

The argument presented here will be somewhat similar to the strategy employed for the torus band inequality of Theorem \ref{ToricBand}. Let $R \geq R_0>0$ and assume that $w:=\mathrm{diam}(M^3)>\frac{4\pi}{\sqrt{3R_0}}$. Take points $p,q\in M^3$ that realise the diameter so that $d(p,q)=w$, and define the two distance functions $r_p(x)=d(x,p)$, $r_q(x)=d(x,q)$. Since Sard's theorem holds for distance functions \cite{Rifford}, the critical values of $r_p$ and $r_q$ form a set of zero measure, and thus there exists an $r_0 \in (0,w)$ which is regular for $r_p$ and such that $w-r_0$ is regular for $r_q$. Note that since the distance functions may only be Lipschitz, the notion of a critical point must be suitably defined, and regular level sets are only guaranteed to be Lipschitz hypersurfaces in general, see \cite{Rifford} for further details. Consider the function $f\in \mathrm{Lip}(M^3 \setminus\{p,q\})$ given by
\begin{equation}
f(x)=
\begin{cases}
-\frac{2\pi}{3w}\cot\left(\frac{\pi}{w}r_{p}(x)\right) &\text{ if }x\in B_{r_0}(p),\\
-\frac{2\pi}{3w}\cot\left(\frac{\pi}{w}r_0 \right) &\text{ if }x\in M^{3} \setminus\left(B_{r_0}(p)\cup B_{w-r_0}(q)\right),\\
-\frac{2\pi}{3w}\cot\left(\pi-\frac{\pi}{w}r_{q}(x)\right) &\text{ if }x\in B_{w-r_0}(q),
\end{cases}
\end{equation}
and observe that a calculation shows the following holds wherever $f$ is differentiable
\begin{equation}\label{alfnhwowt2}
R_0+6f^2 -4|\nabla f|\geq R_0 -\frac{8\pi^2}{3w^2}\geq \frac{1}{2}R_0,
\end{equation}
where the second inequality follows from the hypothesis $w\geq\frac{4\pi}{\sqrt{3R_0}}$.

For any small $\varepsilon>0$, consider the Riemannian band $(M^3_{\varepsilon},\partial_{\pm}M^3_{\varepsilon},g)$ given by
\begin{equation}
M^3_{\varepsilon}=M^3 \setminus \left(B_{\varepsilon}(p) \cup B_{\varepsilon}(q)\right),\quad\quad
\partial_- M_{\varepsilon}^3 =\partial B_{\varepsilon}(p),\quad\quad \partial_+ M_{\varepsilon}^3 =\partial B_{\varepsilon}(q).
\end{equation}
Restricting our attention to $\varepsilon>0$ small enough to guarantee the smoothness of $\partial B_\varepsilon(p)$ and $\partial B_\varepsilon(q)$, consider the corresponding spacetime harmonic Dirichlet boundary value problem
\begin{equation}\label{eq:spacetimeharmonic100}
\begin{cases}
\Delta u_{\varepsilon} + 3f|\nabla u_{\varepsilon}|=0 & \text{ in }M^3_{\varepsilon}, \\
u_{\varepsilon}=\pm1& \text{ on }\partial_\pm M^3_{\varepsilon},
\end{cases}
\end{equation}
and note that Proposition \ref{p:existence} guarantees the existence of a unique solution in $C^{2,\alpha}(M^3_{\varepsilon})$. The integral inequality of Lemma \ref{integralformula} then gives
\begin{align}\label{eq:bandint1100}
\begin{split}
&4\pi\int_{-1}^{1}\chi(\Sigma_t^{\varepsilon})dt+2\int_{\partial_- M^3_{\varepsilon}}(2f-H)|\nabla u_{\varepsilon}|dA-2\int_{\partial_+ M^3_{\varepsilon}}(2f+H)|\nabla u_{\varepsilon}|dA\\
\geq&\int_{M^3_{\varepsilon}}\left(\frac{|\overline{\nabla}^2u_{\varepsilon}|^2}{|\nabla u_{\varepsilon}|}+(R+6f^2)|\nabla u_{\varepsilon}|-4\langle \nabla f,\nabla u_{\varepsilon}\rangle\right)dV,
\end{split}
\end{align}
where $\Sigma_t^{\varepsilon}$ denotes the $t$-level sets of $u_{\varepsilon}$ and $H$ represents the boundary mean curvature with respect to the outer normal.
Observe that $H=-\frac{2}{\varepsilon}+O(\varepsilon)$ and $f=\mp\frac{2}{3\varepsilon}+O(\varepsilon)$ on the geodesic spheres $\partial B_\varepsilon(p)$ and $\partial B_\varepsilon(q)$ as $\varepsilon\rightarrow 0$, where $\mp$ is associated with $p$ and $q$, respectively. Furthermore, according to Lemma \ref{L:A convergence}, $|\nabla u_{\varepsilon}|=O(1)$ on the boundary spheres as $\varepsilon\to0$. 
Combining these observations with \eqref{alfnhwowt2} and \eqref{eq:bandint1100} produces
\begin{equation}\label{e:penintform}
8\pi\int^{1}_{-1}n_{\varepsilon}(t)dt\geq \frac{R_0}{2}\int_{M^3_{\varepsilon}}|\nabla u_{\varepsilon}|dV +O(\varepsilon),
\end{equation}
where $n_{\varepsilon}(t)$ denotes the number of spherical components of regular level sets $\Sigma_{t}^{\varepsilon}$.

Arguing as in Remark \ref{alkjflkjah}, we see that for each $\varepsilon>0$ the function $n_{\varepsilon}:[-1,1]\rightarrow\mathbb{R}$ is measurable. 
We now claim that $n_{\varepsilon}(t)\leq b_2(M^3) +1$ for almost every $t$ and any sufficiently small $\varepsilon$, where $b_2(M^3)$ denotes the second Betti number. This follows by applying Lemma \ref{topangnha} below to $M^3_\varepsilon$, taking $\mathcal{S}^2$ to be the collection of spherical components of a given regular level set. This lemma is applicable because the maximum principle ensures that no component of $M^3_{\varepsilon}\setminus \mathcal{S}^2$ is bounded entirely by a subcollection of $\mathcal{S}^2$. Since $\partial M^3_\varepsilon$ consists of two spheres, and $H_2(M^3_\varepsilon,\partial M^3_\varepsilon)\cong H_2(M^3)$, the upper bound in Lemma \ref{topangnha} becomes $b_2(M^3)+1$.
Next, apply Lemma \ref{L:A convergence} to show that as $\varepsilon\to0$, $u_{\varepsilon}$ subconverges on compact subsets of $M^3 \setminus\{p,q\}$ to a spacetime harmonic function $u\in C^{2,\beta}(M^{3}\setminus \{p,q\})\cap C^{0,1}(M^3)$ with $\alpha\in[0,\beta)$, such that $u(p)=-1$ and $u(q)=1$.  By combining these facts with Fatou's Lemma, the coarea formula, and \eqref{e:penintform}, we find
\begin{equation}\label{all296t9t0}
\frac{32\pi}{R_0}\left(b_2(M^3) +1\right)\geq \int_{M^3}|\nabla u|dV=\int_{-1}^{1}\mathcal{H}^2(\Sigma_t)dt,
\end{equation}
where $\mathcal{H}^2(\Sigma_t)$ is the $2$-dimensional Hausdorff measure of the $t$-level set $\Sigma_t$ of $u$. The desired inequality \eqref{e:waistest} now follows from \eqref{all296t9t0}. It remains to establish the following topological estimate.

\begin{lemma}\label{topangnha}
Let $M^3$ be a connected orientable compact 3-manifold whose boundary has $b$ connected components. Let $\mathcal{S}^2\subset M^3$ be a collection of disjoint two-sided embedded spheres with the property that no component of $M^3\setminus \mathcal{S}^2$ is bounded entirely by a subcollection of spheres in $\mathcal{S}^2$. Then the number of components in $\mathcal{S}^2$ cannot exceed $\textnormal{rank} H_2(M^3,\partial M^3;\mathbb{Z})+b-1$.
\end{lemma}

\begin{proof}
All (co)homology groups in this proof are understood to have integer coefficients. 
Consider the triple $(M^3,\partial M^3\cup \mathcal{S}^2,\partial M^3)$ and the associated long exact sequence
\begin{equation}\label{e:LEStriple}
\begin{tikzpicture}[>=triangle 60,baseline=(current  bounding  box.center)]
\matrix[matrix of math nodes,column sep={100pt,between origins},row
sep={20pt,between origins},nodes={asymmetrical rectangle}] (s)
{
&|[name=k0]|\phantom{00000000000}0&|[name=ka]| H_3(M^3,\partial M^3) &|[name=kb]| H_3(M^3,\partial M^3\cup \mathcal{S}^2) &\\
|[name=00]|&|[name=01]|&|[name=02]| &|[name=03]| &|[name=04]| \\
|[name=c00]|&|[name=c0]| H_2(\partial M^3\cup \mathcal{S}^2,\partial M^3)&|[name=ca]| H_2( M^3,\partial M^3) &|[name=cb]| H_2(M^3,\partial M^3\cup \mathcal{S}^2) &|[name=cc]| 0\phantom{000000000} \\
};
\draw[->] (k0) edge (ka)
          (ka) edge (kb)
          (c0) edge (ca)
          (ca) edge (cb)
          (cb) edge (cc)
;
\draw[->,black,rounded corners] (kb) -| 
 ($(04.west)-(1.2,0)$) |- ($(01)+(-.3,0)$) -|
($(c00.east)+(1,0)$) |- (c0);
\end{tikzpicture}
\end{equation}

Let $U\subset M^3$ be a thickened copy of $\mathcal{S}^2$, retracting onto $\mathcal{S}^2$ and avoiding $\partial M^3$. 
Also consider the compact manifold $N^3$ obtained by removing $\mathcal{S}^2$ from $M^3$ and attaching two copies of $\mathcal{S}^2$ back to the resulting edges.
Evidently, the pair $(N^3,\partial N^3)$ is homotopy equivalent to $(M^3\setminus \mathcal{S}^2,\partial M^3\cup( U\setminus\mathcal{S}^2))$.
Using excision and Poincar{\'e}-Lefschetz duality, we have the following chain of isomorphisms
\begin{align}\label{e:homologyiso}
\begin{split}
H_3(M^3,\partial M^3\cup \mathcal{S}^2)\cong H_3(M^3,\partial M^3\cup U)&\cong H_3(M^3\setminus\mathcal{S}^2,\partial M^3\cup (U\setminus\mathcal{S}^2))\\
{}&\cong H_3(N^3,\partial N^3)\cong H^0(N^3).
\end{split}
\end{align}
Notice that $k:=\mathrm{rank} H^0(N^3)$ represents the number of components of $M^3\setminus\mathcal{S}^2$. 
We claim that $k\leq b$. 
Indeed, since components of $M^3\setminus \mathcal{S}^2$ cannot be bounded entirely by spheres in $\mathcal{S}^2$, each component of $M^3\setminus \mathcal{S}^2$ contains at least one component of $\partial M^3$. 
Next, due to the fact that the alternating sum of the ranks of abelian groups in a long exact sequence vanishes, \eqref{e:LEStriple} and \eqref{e:homologyiso} imply
\begin{align}
\begin{split}
   \mathrm{rank} H_2(\partial M^3\cup \mathcal{S}^2 ,\partial M^3)=& \mathrm{rank} H_2(M^3,\partial M^3)-\mathrm{rank} H_2(M^3,\partial M^3\cup \mathcal{S}^2)\\
   {}&\qquad+\mathrm{rank}H_3(M^3,\partial M^3\cup \mathcal{S}^2)-\mathrm{rank}H_3(M^3,\partial M^3)\\
   =& \mathrm{rank} H_2(M^3,\partial M^3)-\mathrm{rank} H_2(M^3,\partial M^3\cup \mathcal{S}^2)+k-1\\
   \le& \mathrm{rank} H_2(M^3,\partial M^3)+b-1.
   \end{split}
\end{align}
Since $\mathrm{rank} H_2(\partial M^3\cup \mathcal{S}^2 ,\partial M^3)$ is the number of components of $\mathcal{S}^2$, the lemma follows.
\end{proof}

\subsection{Proof of Theorem \ref{thm:dice}}

Let $(M^3,g)$ be closed with scalar curvature satisfying $R \geq R_0>0$, and set $w_0:=\frac{4\pi}{\sqrt{3R_0}}$. 
Take $p\in M^3$ and consider the distance function $r_p(x)=d(x,p)$. 
As in the previous subsection, we recall that critical values of $r_p$ form a null set according to the version of Sard's theorem given in \cite{Rifford}, and therefore a regular value may be chosen arbitrarily close to $2w_0$; we will assume that this has been done and for simplicity will continue to denote the value by $2w_0$. The geodesic sphere $\partial B_{2w_0}(p)$ then forms a (possibly empty) codimension-1 Lipschitz submanifold. If this surface is empty, set $V_1=B_{2w_0}(p)=M^3$ and there is nothing more to show. If it is nonempty, take a smooth approximating surface $S_1$ which bounds a closed region $B_1$ that is arbitrarily close in Hausdorff distance to $B_{2w_0}(p)$; this approximation may be obtained by running mean curvature flow for a short time \cite{EckerHuisken}. Fix $\delta_1>0$ appropriately small and let $r_1(x)=d(x,S_1)$, then appealing to \cite{Rifford} we may assume that $\tfrac{w_0}{2} +\delta_1$ is a regular value for $r_1$. Now consider the Riemannian band $(M^3_1,\partial_{\pm}M^3_1,g)$ in which $M^3_1=\{x\in M^3\mid r_1(x)\leq \tfrac{w_0}{2} +\delta_1\}$, with $\partial_- M^3_1$ ($\partial_+ M^3_1$) denoting the boundary components having nontrivial (trivial) intersection with $B_1$. As before, the surfaces $\partial_{\pm} M^3_1$ may be replaced with smooth approximations if necessary, and by abuse of notation the relevant band having these smooth approximations as boundary will be denoted in the same way. Note that by definition of the sets involved, $\partial_- M^3_1$ is nonempty. If $\partial_+ M^3_1 =\emptyset$ then $M^3 \subset B_1\cup M^3_1$, so that by setting $V_1=M^3$ the argument is complete. Otherwise, we will seek a surface of controlled area within the band that may be found as a spacetime harmonic function level set. 
 
Let $\varepsilon>0$ be small such that $\tfrac{w_0}{2} -\varepsilon$ is regular for $r_1$, and consider the function $f_{\varepsilon}\in \mathrm{Lip}(M^3_1)$ given by
\begin{equation}
f_{\varepsilon}(x)=
\begin{cases}
-\frac{2\pi}{3w_0}\tan\left(\frac{\pi}{w_0}r_{1}(x)\right) &\text{ if }r_1(x)\leq \tfrac{w_0}{2} -\varepsilon\text{ and }x\in B_1 ,\\
-\frac{2\pi}{3w_0}\cot\left(\frac{\varepsilon\pi}{w_0} \right) &\text{ if }r_1(x)>\tfrac{w_0}{2} -\varepsilon\text{ and }x\in B_1 ,\\
\frac{2\pi}{3w_0}\tan\left(\frac{\pi}{w_0}r_{1}(x)\right) &\text{ if }r_1(x)\leq \tfrac{w_0}{2} -\varepsilon\text{ and }x\notin B_1 ,\\
\frac{2\pi}{3w_0}\cot\left(\frac{\varepsilon\pi}{w_0} \right) &\text{ if }r_1(x)>\tfrac{w_0}{2} -\varepsilon\text{ and }x\notin B_1 .
\end{cases}
\end{equation}
Observe that a calculation yields almost everywhere
\begin{equation}\label{alfnhwowt}
R_0+6f_{\varepsilon}^2 -4|\nabla f_{\varepsilon}|\geq R_0 -\frac{8\pi^2}{3w^2_0}= \frac{1}{2}R_0
\end{equation}
and for $\varepsilon$ fixed sufficiently small we have
\begin{equation}
2f_{\varepsilon}-H<0 \quad\text{on } \partial_- M_1^3,\quad\quad\quad
2f_{\varepsilon}+H>0\quad\text{on } \partial_+ M_1^3,
\end{equation}
where $H$ denotes the boundary mean curvature with respect to the outer normal.
Furthermore, Proposition \ref{p:existence} gives a unique solution in $C^{2,\alpha}(M^3_1)$ to the corresponding spacetime harmonic Dirichlet boundary value problem
\begin{equation}\label{eq:spacetimeharmonic100111}
\begin{cases}
\Delta u_{1} + 3f_{\varepsilon}|\nabla u_{1}|=0 & \text{ in }M^3_{1}, \\
u_{1}=\pm 1& \text{ on }\partial_{\pm} M^3_{1}.
\end{cases}
\end{equation}
Therefore, in a manner similar to that of Section \ref{secwaist}, we may then apply the integral inequality of Lemma \ref{integralformula} together with the coarea formula to obtain
\begin{equation}
8\pi\int^{1}_{-1}N_{1}(t)dt\geq 4\pi\int_{-1}^1 \chi(\Sigma_t^{1})dt\geq \frac{R_0}{2}\int_{M^3_{1}}|\nabla u_{1}|dV =\frac{R_0}{2}\int_{-1}^1 \mathcal{H}^2(\Sigma_t^{1})dt,
\end{equation}
where $N_{1}(t)$ denotes the number of components of regular level sets $\Sigma_{t}^{1}=u_{1}^{-1}(t)$ and $\mathcal{H}^2$ indicates 2-dimensional Hausdorff measure. It follows that there exists a regular value $t_1\in [-1,1]$ such that the `normalized area' of the level set satisfies the bound
\begin{equation}
\mathcal{A}(\Sigma_{t_1}^1):= N_1(t_1)^{-1}\mathcal{H}^2(\Sigma_{t_1}^{1})\leq \frac{16\pi}{R_0}.
\end{equation}
Now set $V_1=\left(B_1 \setminus M^3_1 \right)\cup u_1^{-1}([-1,t_1])$, and note that the entirety of $V_1$ lies within a controlled distance to its boundary $\partial V_1 =\Sigma_{t_1}^1$, that is,
\begin{equation}
\max_{x\in V_1}d(x,\partial V_1)\leq 5w_0
\end{equation}
if $\delta_1$ is small enough.

Next, let $B_{2w_0}(\partial V_1)$ be the set of points in $M^3$ of distance less than or equal to $2w_0$ from $\partial V_1$. As before, using the version of Sard's theorem presented in \cite{Rifford}, we may assume that this number has been replaced with a nearby one which is a regular value for the distance function while keeping the same notation. Set $S_2=\left[\partial B_{2w_0}(\partial V_1) \right]\setminus V_1$ to be the collection of boundary points disjoint from $V_1$. If this surface is empty define $V_2$ to be the closure of $B_{2w_0}(\partial V_1)\setminus V_1$, and there is nothing more to show since 
\begin{equation}\label{alkjfkhlqh}
\max_{x\in V_2}d(x,\partial V_2)\leq 2w_0 .
\end{equation}
If it is nonempty then take a smooth approximating surface, still denoted by $S_2$, which serves as a portion of the boundary for a closed region $B_2$ that is arbitrarily close in Hausdorff distance to $B_{2w_0}(\partial V_1)\setminus V_1$; as before this approximation may be obtained by running mean curvature flow for a short time \cite{EckerHuisken}. Fix $\delta_2>0$ appropriately small and such that $\tfrac{w_0}{2} +\delta_2$ is a regular value for $r_2(x)=d(x,S_2)$. Consider the Riemannian band $(M^3_2,\partial_{\pm}M^3_2,g)$ in which $M^3_2=\{x\in M^3\mid r_2(x)\leq \tfrac{w_0}{2} +\delta_2\}$, with $\partial_- M^3_2$ ($\partial_+ M^3_2$) representing the boundary components having nontrivial (trivial) intersection with $B_2$. If necessary, we replace the surfaces $\partial_{\pm} M^3_2$ with smooth approximations, and will continue to denote the resulting band having smooth boundary with the same notation. Note that by definition of the sets involved, $\partial_- M^3_2$ is nonempty. If $\partial_+ M^3_2 =\emptyset$, then set $V_2=B_2\cup M_2^3$ and the argument is complete, since \eqref{alkjfkhlqh} holds with $3w_0$ on the right-hand side. On the other hand, if this portion of the boundary is nonempty, then as above we may
find a surface $\Sigma_{t_2}^2\subset M^3_2$ that arises as a spacetime harmonic function level set $u_2^{-1}(t_2)$ and satisfies
\begin{equation}
\mathcal{A}(\Sigma_{t_2}^2):= N_2(t_2)^{-1}\mathcal{H}^2(\Sigma_{t_2}^{2})\leq \frac{16\pi}{R_0},
\end{equation}
where $N_2(t_2)$ indicates the number of components. Now set $V_2=\left(B_2 \setminus M^3_2 \right)\cup u_2^{-1}([-1,t_2])$, and observe that the distance between the two level sets satisfies
\begin{equation}
d(\Sigma_{t_1}^1,\Sigma_{t_2}^2)\leq 4w_0
\end{equation}
if $\delta_2$ is small enough.

\begin{figure}
\begin{picture}(0,0)
\put(16,25){$V_1$}
\put(57,34){$V_2$}
\put(121,8){$V_3$}
\put(175,33){$V_3$}
\put(234,6){$V_4$}
\put(229,75){$V_4$}
\put(109,32){\textcolor{darkblue}{$\Sigma_2$}}
\put(221,40){\textcolor{darkblue2}{$\Sigma_3$}}
\put(38,8){\textcolor{darkestblue}{$\Sigma_1$}}
\put(0,70){$(M^3,g)$}
\end{picture}
\includegraphics[totalheight=3cm]{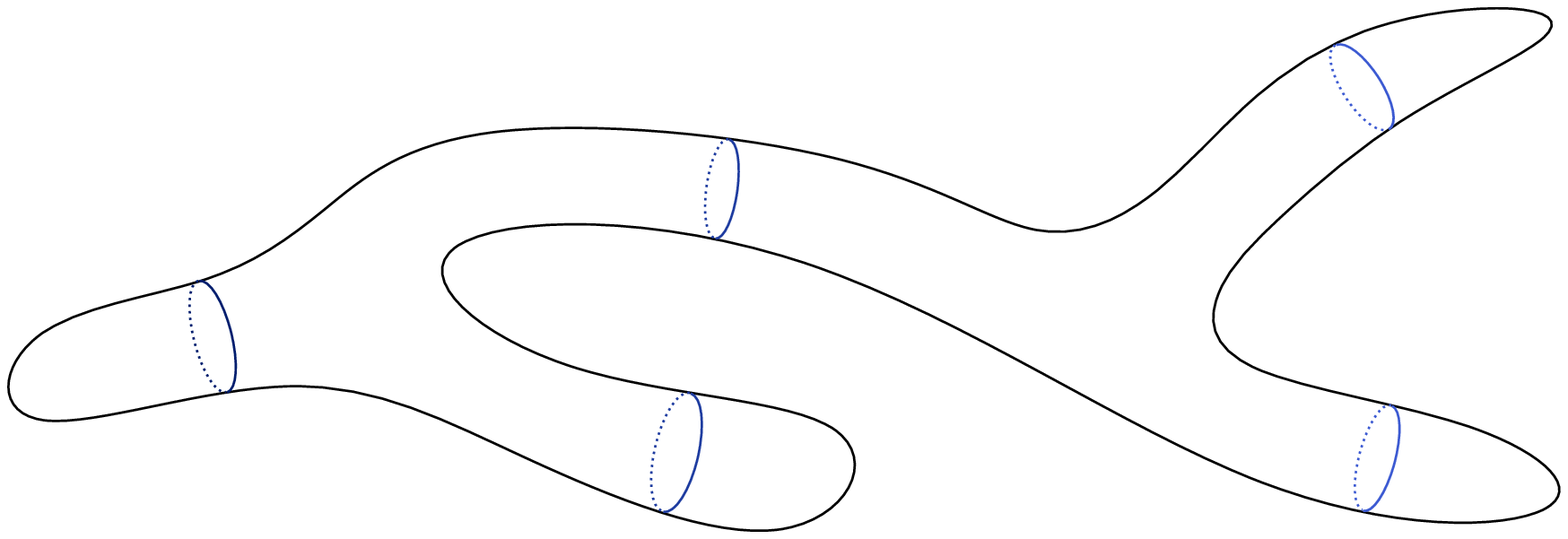}
\caption{The decomposition of $M^3$ into regions $\{V_i\}_{i=1}^{I}$.}
\label{figV}
\end{figure}


This procedure may be continued until $M^3$ is exhausted.  It gives rise to an inductive sequence of closed regions $V_1,\ldots,V_I$ that cover $M^3$, see Figure \ref{figV}. The boundaries $\partial V_i$ of these regions are regular level sets of spacetime harmonic functions, and are thus $C^2$ smooth. Furthermore, members of the sequence only have nontrivial intersection with their neighbors, and the intersection consists of boundary components. More precisely, $V_i \cap V_j =\emptyset$ unless $j=i\pm 1$ and $V_i \cap V_j =\partial V_i \cap \partial V_j$. The intersection surfaces have controlled `averaged area'
\begin{equation}
\mathcal{A}(\partial V_i \cap \partial V_{i+1})\leq \frac{16\pi}{R_0},\quad\quad\quad i=1,\ldots, I-1,
\end{equation}
and they lie within a fixed distance to one another
\begin{equation}
d(\partial V_{i-1} \cap \partial V_i , \partial V_i \cap \partial V_{i+1})\leq 4w_0,\quad\quad\quad i=2,\ldots,I-1,
\end{equation}
with the first and last regions staying within a bounded Hausdorff distance to their boundaries
\begin{equation}
d_H(V_1,\partial V_1)\leq 5w_0,\quad\quad\quad\quad d_H(V_I,\partial V_I)\leq 3w_0.
\end{equation}
Lastly, the maximum number of regions needed to exhaust the manifold may be estimated in terms of diameter as follows
\begin{equation}
w_0 (I-1)\leq \mathrm{diam}(M^3).
\end{equation}

\section{The Lipschitz Constant of Maps From $M^3\to S^3$} 
\label{S:Llarull}

This section is devoted to Theorems \ref{Llarull}, \ref{LlarullBand}, and Main Theorem \ref{LlarullIncomplete}. Before beginning, we introduce a few notations used in this section. Recall from Section \ref{s:statements} that $A[r_1,r_2]$ denotes the annular region in the unit sphere $(S^3,g_{S^3})$ centered about the north pole $N\in S^3$ with inner and outer radii $r_1$ and $r_2$, respectively. Consider a $3$-dimensional Riemannian band $(M^3,\partial_\pm M^3,g)$. We say that a map $\ell:M^3\to A[r_1,r_2]$ has {\emph{nonzero degree}} if $\ell(\partial_-M)\subset \partial B^{S^3}_{r_2}(N)$, $\ell(\partial_+M)\subset\partial B^{S^3}_{r_1}(N)$, and $\ell_*:H_3(M^3,\partial M^3;\mathbb{Z})\to H_3(A[r_1,r_2],\partial A[r_1,r_2];\mathbb{Z})$ is nontrivial. We use $\theta:S^3\to[0,\pi]$ to denote the spherical distance $\theta(x)=d_{S^3}(x,N)$. 


\subsection{Proof of Theorem \ref{LlarullBand}}

It suffices to show that if $H(x)\geq H_{g_{S^3}}(\ell(x))$ for all $x\in \partial M^3$, then the band is isometric to $A[r_1,r_2]$. We will therefore proceed with this mean curvature hypothesis.
Set $f(x)=\cot(\theta\circ\ell(x))$ and let $u\in C^{2,\alpha}(M^3)$ be the corresponding spacetime harmonic function with Dirichlet conditions $u=\pm1$ on $\partial_\pm M^3$ supplied by Proposition \ref{p:existence}. By assumption, the outward mean curvature of $\partial M^3$ satisfies $H+2f\ge 0$ on $\partial_+ M^3 $ and  $H-2f\ge 0$ on $\partial_- M^3$. Moreover,
using the topological assumptions on $M^3$ and Lemma \ref{topangnha}, the number of spherical components of a regular level set $\Sigma_t$ of $u$ cannot exceed $1$ so that $\chi(\Sigma_t)\le 2$. The integral formula of Lemma \ref{integralformula} then produces
\begin{align}\label{alfjhiw}
\begin{split}
&\int_{M^3}\left(\frac{|\overline\nabla^2u|^2}{|\nabla u|}+(R+6f^2)|\nabla u|-4\langle\nabla f,\nabla u\rangle\right)dV 
\\\le& 4\pi\int_{-1}^{1}\chi(\Sigma_t)dt-2\int_{\partial M^3}(H|\nabla u|+2\mathbf{n}(u) f)dA\\
\le &4\pi\int_{-1}^{1}\chi(\Sigma_t)dt\\
\leq & 16\pi,
\end{split}
\end{align}
where $\mathbf{n}$ is the outward unit vector to $\partial M^3$.

First we analyse the bulk integrand.
Using the scalar curvature lower bound, the Lipschitz condition on $\ell$, Cauchy-Schwartz, the co-area formula, and elementary trigonometric identities,
\begin{align} \label{mu-J1}
\begin{split}
&\int_{M^3}\left((R+6f^2)|\nabla u|-4\langle\nabla f,\nabla u\rangle \right)dV\\
\ge&
\int_{M^3}\left(6\csc^2 \left(\theta\circ\ell(x)\right)-4\csc^2(\theta\circ l(x))|D\ell|\right)|\nabla u|dV
\\ \ge&\int_{M^3} 2\csc^2(\theta\circ \ell)|\nabla u|dV\\
=&\int_{-1}^{1}\int_{\Sigma_t}2\csc^2(\theta\circ \ell)dA dt\\
\geq &\int_{-1}^{1}\int_{\ell(\Sigma_t)}2\csc^2 \theta dA_{S^3} dt
\end{split}
\end{align}
where $dA_{S^3}$ denotes the area element of $\ell(\Sigma_t)\subset S^3$ induced from the round metric.
Observe that the last integral over $\ell(\Sigma_t)$ represents double the area of this surface with respect to the metric
\begin{equation}
\csc^2\theta  g_{S^3}=\csc^2 \theta (d\theta^2+\sin^2 \theta g_{S^2})=\csc^2 \theta d\theta^2+g_{S^2}, 
\label{csc}
\end{equation}
which is a product metric on the cylinder $(0,\pi)\times S^2$. The minimum area for a homologically nontrivial surface in this cylinder is given by $4\pi$, which is achieved by cross-sections. Since the degree of $\ell$ is non-zero, and $\Sigma_t$ is homologous to $\partial_-M$, $\ell(\Sigma_t)$ is such a homologically nontrivial surface. It follows that
\begin{equation}\label{fuuuuu}
\int_{-1}^{1}\int_{\ell(\Sigma_t)}2\csc^2 \theta dA_{S^3} dt\geq 16\pi.    
\end{equation}

Combining \eqref{alfjhiw}, \eqref{mu-J1}, and \eqref{fuuuuu} forces all of the above inequalities to be equalities. From \eqref{alfjhiw}, we have $\overline{\nabla}^2u=0$, which allows us to integrate the fact that $|\nabla u|>0$ on $\partial M^3$ inwards and conclude $|\nabla u|\neq 0$ throughout $M^3$. Now we inspect the inequalities in \eqref{mu-J1}. From the first inequality there, $\nabla u$ must be parallel to $\nabla (\theta\circ\ell)$, and from the last inequality, $\ell|_{\Sigma_t}$ is area preserving. In other words,
\begin{equation} \label{l equation}
   {\bigg{\langle}} D\ell\left(\frac{\nabla u}{|\nabla u|}\right), \tilde{\nabla} \theta {\bigg{\rangle}}=-1,\quad |\det (D\ell|_{\Sigma_t})|=1,  
\end{equation}
where $\tilde{\nabla}$ is the gradient from $(S^3,g_{S^3})$. Combined with $|D\ell|\leq1$, we find that $\ell$ is a local isometry. Since $A[r_1,r_2]$ has no nontrivial covers, $\ell$ is globally isometric and the result follows.


\subsection{Proof of Theorem \ref{Llarull}}




We will postpone the proof of the inequality \eqref{e:Llarullquant} until the end and first focus on the rigidity statement. Our strategy is as follows: produce an appropriate exhaustion of $M^3$ by bands, solve a spacetime harmonic equation on these bands, apply the calculations performed in the proof of Theorem \ref{LlarullBand}, and then take a limit. In the step \eqref{alfjhiw}, it was crucial that the Euler characteristic of regular level sets satisfied $\chi(\Sigma_t)\le 2$. In order to arrange for this, some care must be taken in constructing the bands within $M^3$.

To begin, Sard's Theorem guarantees that almost all points of $S^3$ are regular for $\ell$, and in particular we may find antipodal regular values which we choose and denote by $\{N,S\}$.
Enumerate the points in $\ell^{-1}(N)$ by $N_0,N_1,\dots,N_k$ and points of $\ell^{-1}(S)$ by $S_0,S_1,\dots, S_l$. Let $\delta>0$ be a parameter which will eventually be taken to $0$. Let $U_{\delta}^j$ and $V_{\delta}^i$ be open neighborhoods around $N_j$ and $S_i$ such that $\ell(U_\delta^j)=B^{S^3}_{\delta}(N)$, $\ell(V_\delta^i)=B^{S^3}_{\delta}(S)$ for $j=0,1,\dots, k$ and $i=0,1,\dots,l$. Setting aside the first regions in the list, define $\mathcal{N}_{\delta}=\cup_{j=1}^k U^j_\delta$ and $\mathcal{S}_{\delta}=\cup_{i=1}^l V^i_\delta$. Since $\ell$ is regular at $\ell^{-1}(\{S,N\})$, there exists $C_\ell>0$ such that $U^j_{\delta}\subset B_{C_\ell\delta}(N_j)$ and $V^i_{\delta}\subset B_{C_\ell\delta}(S_i)$ for $j=0,1,\dots, k$ and $i=0,1,\dots,l$. See Figure \ref{Fig:Section6} for a depiction of these sets.

\begin{center}
\begin{figure}[h]
\begin{picture}(0,0)
\put(28,123){\textcolor{lightblue}{$U_\delta^1$}}
\put(196,126){\textcolor{lightblue}{$U_\delta^2$}}
\put(113,125){\textcolor{darkblue}{$U_\delta^0$}}
\put(30,49){\textcolor{lightred}{$V_\delta^1$}}
\put(195,33){\textcolor{lightred}{$V_\delta^2$}}
\put(115,33){\textcolor{darkred}{$V_\delta^0$}}
\put(197,71){\textcolor{darkgreen}{$\Sigma_t$}}
\put(104,85){{$(M^3,g)$}}
\put(276,86){$\ell$}
\put(375,133){N}
\put(377,19){S}
\put(370,98){\textcolor{darkgreen}{$\ell(\Sigma_t)$}}
\end{picture}
\includegraphics[scale=.8]{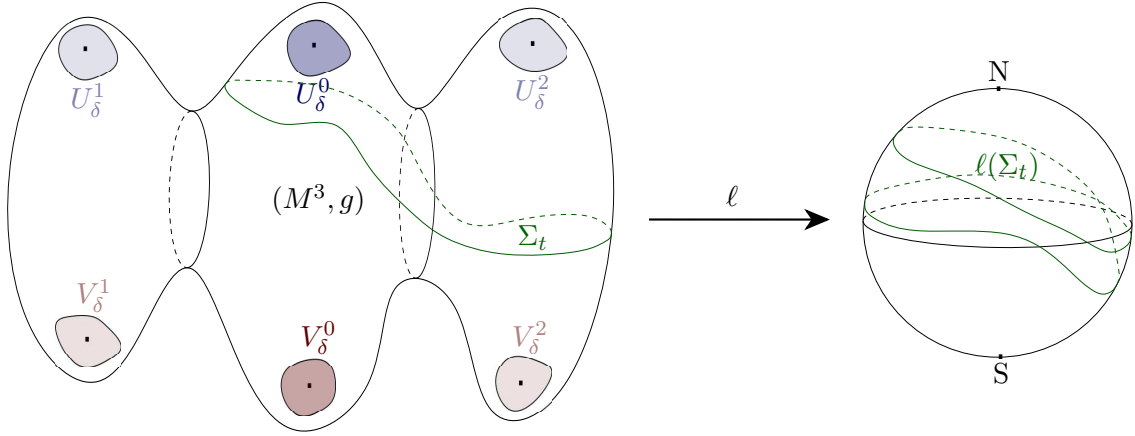}
\caption{
Schematic description of the construction used in the proof of Theorem \ref{Llarull}, including a level set $\Sigma_t$ of $u_{\varepsilon,\delta}$.
}\label{Fig:Section6}
\end{figure}
\end{center}

Fix once and for all a radius $r_1>0$ such that $U^0_{r_1},\cdots, U^k_{r_1}$ and $V^0_{r_1},\cdots, V^l_{r_1}$ are disjoint. For small $\varepsilon$, we will construct functions $f_{\varepsilon,\delta}$ on $M^3\setminus(\overline{U^0_\varepsilon\cup V^0_\varepsilon})$ in a manner similar to Lemma \ref{construct f}. The important difference is that here, $f_{\varepsilon,\delta}$ is defined using $\ell$ and $\theta$ and must blow up only near $N_0$, $S_0$.
For sufficiently small $\varepsilon\in(0,r_1)$, $\delta\in (0,\varepsilon)$, and a fixed $r_0\in(0,r_1)$, define $f_{\varepsilon,\delta}$ by
\begin{equation}
f_{\varepsilon,\delta}(x)=
\begin{cases}
\cot(\psi_\varepsilon^0 \circ \theta\circ \ell(x)) & \text{if } x\in U^0_{r_1}\cup V^0_{r_1}\setminus (\overline{U^0_{\varepsilon}\cup V^0_{\varepsilon}}), \\  
\cot(\theta\circ \ell(x)) & \text{if } x\in M^3 \setminus (\mathcal{N}_{\delta}\cup \mathcal{S}_{\delta}\cup U^0_{r_1}\cup V^0_{r_1}),\\
\cot(\psi_{\delta}\circ \theta\circ \ell (x)) & \text{if } x\in \mathcal{N}_{{\delta}}\cup \mathcal{S}_{\delta},
\end{cases}
\end{equation}
where 
\begin{equation}
\psi_{\delta}(\theta)=
\begin{cases}
\frac{1}{2\delta}\theta^2+\frac{1}{2}\delta & \text{if }\theta\le \delta,\\
\theta & \text{if }\delta\leq\theta\leq \pi-\delta,\\
\pi-[\frac{1}{2\delta}(\pi-\theta)^2+\frac{1}{2}\delta] & \text{if }\theta\ge \pi- \delta,
\end{cases}
\end{equation}
prevents $f_{\varepsilon,\delta}$ from blowing up near $N_1,\dots, N_k, S_1,\dots, S_l$, and 
\begin{equation}
\psi_\varepsilon^0(\theta)=
\begin{cases}
\theta-\varepsilon & \text{if } \varepsilon\le \theta\le r_0,\\ \theta-\varepsilon+\varepsilon\psi(\frac{\theta-r_0}{r_1-r_0}) & \text{if } r_0\le \theta\le r_1,\\ 
\theta & \text{if } r_1\le \theta\le \pi-r_1,\\ 
\theta+\varepsilon \psi(\frac{\theta-\pi+r_1}{r_1-r_0}) & \text{if }
\pi-r_1\le \theta\le \pi-r_0,\\ 
\theta+\varepsilon & \text{if } \pi-r_0\le \theta\le \pi-\varepsilon,
\end{cases}
\end{equation}
forces $f_{\varepsilon,\delta}$ to blow up slightly before reaching $N_0,S_0$. Above, $\psi$ is a fixed smooth cut-off function on $[0,1]$ satisfying $\psi(0)=0$, $\psi(1)=1$, $\psi'(0)=\psi'(1)=0$, and $0\le\psi'(x)\le 2$. This ensures $1\le \tfrac{d}{d\theta}\psi_\varepsilon^{0}\le 1+\frac{2\varepsilon}{r_1-r_0}$. 

Next, we collect some properties of $f_{\varepsilon,\delta}$. Note that $f_{\varepsilon,\delta}\in \textnormal{Lip}_{loc}(M^3\backslash (\overline{U^0_\varepsilon\cup V^0_\varepsilon}))$ and is hence differentiable almost everywhere. Using $\mathrm{Lip}(\ell)\leq1$, the following differential inequalities hold almost everywhere: 
\begin{equation}\label{e:llarull69}
       3+3f_{\varepsilon,\delta}^2-2  |\nabla f_{\varepsilon,\delta}|\ge \csc^2 \theta\quad\quad \text{ in } M^3 \setminus (\mathcal{N}_{\delta}\cup \mathcal{S}_{\delta}\cup U^0_{r_1}\cup V^0_{r_1}),
\end{equation}
\begin{equation} \label{ODE:NS region}
  3+3f_{\varepsilon,\delta}^2-2  |\nabla f_{\varepsilon,\delta}|\ge 3\csc^2[\psi_\delta(\theta)]-  2\delta^{-1}\theta \csc^2[\psi_\delta(\theta)]>0\quad\quad \text{ in } \mathcal{N}_{\delta},
\end{equation}
\begin{equation}
    3+3f_{\varepsilon,\delta}^2-2  |\nabla f_{\varepsilon,\delta}|\ge\csc^2(\theta-\varepsilon)\ge \csc^2 \theta \quad\quad
    \text{ in } U^0_{r_0},
\end{equation}
and 
\begin{align}\label{e:llarull612}
\begin{split}
    3+3f_{\varepsilon,\delta}^2-2  |\nabla f_{\varepsilon,\delta}|
    \ge& 3\csc^2(\psi^0_\varepsilon\circ\theta)- 2\left(1+\frac{2\varepsilon}{r_1-r_0}\right)\csc^2(\psi^0_\varepsilon\circ \theta)
    \\ \ge& \left(1-\frac{4\varepsilon}{r_1-r_0}\right)\csc^2\theta\\ \ge& \csc^2\theta-C_{r_0,r_1}\varepsilon
    \end{split}
\end{align}
in the region $U^0_{r_1}\setminus U^0_{r_0}$.
Similar inequalities may be shown to hold in the regions $V^0_{r_1}$ and $\mathcal{S}_{\delta}$. 

For $\varepsilon_1>\varepsilon$, consider the band $(M^3_{\varepsilon_1},\partial_\pm M^3_{\varepsilon_1})$ where $M^3_{\varepsilon_1}=M^3\setminus (U^0_{\varepsilon_1}\cup V^0_{\varepsilon_1})$, $\partial _-M^3_{\varepsilon_1}=\partial V^0_{\varepsilon_1}$, and $\partial _+M^3_{\varepsilon_1}=\partial U^0_{\varepsilon_1}$. Since $f_{\varepsilon,\delta}$ blows up at $\partial U^0_\varepsilon$ and $\partial V^0_\varepsilon$, and $\ell$ is regular at $\{N,S\}$, we can find $\varepsilon_1\in(\varepsilon,2\varepsilon)$ such that $|H_{\varepsilon_1}|\le 2|f_{\varepsilon,\delta}|$ on $\partial M^3_{\varepsilon_1}$ where $H_{\varepsilon_1}$ is the outward mean curvature of the boundary. Now use Proposition \ref{p:existence} to solve the spacetime harmonic boundary value problem
\begin{equation}
\begin{cases}
    \Delta u_{\varepsilon,\delta}+3f_{\varepsilon,\delta}|\nabla u_{\varepsilon,\delta}|=0&\text{ in }M^3_{\varepsilon_1},\\
    u_{\varepsilon,\delta}=\pm 1& \text{ on }\partial_\pm M^3_{\varepsilon_1}.
    \label{setting pde}
\end{cases}
\end{equation}
Let $\mathbf{n}$ denote the outward unit normal to $\partial M^3_{\varepsilon_1}$. Applying the integral formula \eqref{integralformula}, one finds
\begin{align}\label{integral formula LL}
\begin{split}
    &\int_{M^3_{\varepsilon_1}}\left(\frac{1}{2}\frac{|\overline{\nabla}^2 u_{\varepsilon,\delta}|^2}{|\nabla u_{\varepsilon,\delta}|}+\frac12R|\nabla u_{\varepsilon,\delta}|+3f_{\varepsilon,\delta}^2|\nabla u_{\varepsilon,\delta}|-2\langle\nabla f_{\varepsilon,\delta}, \nabla u_{\varepsilon,\delta}\rangle\right)dV\\ \le& -\int_{\partial M^3_{\varepsilon_1}} (|\nabla u_{\varepsilon,\delta}|H_{\varepsilon_1}+2f_{\varepsilon,\delta}\mathbf{n}(u_{\varepsilon,\delta}))dA+\int^{1}_{-1} 2\pi \chi(\Sigma_t)dt.
\end{split}
\end{align}
On $\partial_\pm M^3_{\varepsilon_1}$, we have $H_{\varepsilon_1} \pm 2f_{\varepsilon,\delta}\ge 0$ and, due to the Dirichlet conditions, $\mathbf{n}(u_{\varepsilon,\delta})=\pm|\nabla u_{\varepsilon,\delta}|$. 
We conclude that the boundary term in \eqref{integral formula LL} is nonpositive. Furthermore, it follows from the previous estimates \eqref{e:llarull69}--\eqref{e:llarull612} that 
\begin{align}\label{integral estimate 4} 
\begin{split}
  &  \int_{M^3_{\varepsilon_1}}\left(\frac{1}{2}R|\nabla u_{\varepsilon,\delta}|+3f_{\varepsilon,\delta}^2|\nabla u_{\varepsilon,\delta}|-2\langle\nabla f_{\varepsilon,\delta}, \nabla u_{\varepsilon,\delta}\rangle\right)dV
\\    \ge& \int_{M^3_{\varepsilon_1}}(3+3f^2_{\varepsilon,\delta}-2|\nabla f_{\varepsilon,\delta}|)|\nabla u_{\varepsilon,\delta}|dV+\int_{M^3_{\varepsilon_1}}\tfrac{1}{2}|\nabla u_{\varepsilon,\delta}|(R-6)dV
\\ \ge& \int_{M^3_{\varepsilon_1}\setminus( \mathcal{N}_{\delta}\cup \mathcal{S}_{\delta})} |\nabla u_{\varepsilon,\delta}|\csc^2\theta dV-C_{r_0,r_1}\varepsilon \int_{\mathcal{B}} |\nabla u_{\varepsilon,\delta}| dV
\\ \ge& \int_{[-1,1]\setminus I_{\varepsilon,\delta}}\int_{\{u_{\varepsilon,\delta}=t\}}\csc^2\theta dA dt
-C_{r_0,r_1}\varepsilon \int_{\mathcal{B}} |\nabla u_{\varepsilon,\delta}|dV
\\ \ge & 4\pi 
\left(2-|I_{\varepsilon,\delta}|\right)
-C_{r_0,r_1}\varepsilon \int_{\mathcal{B}}|\nabla u_{\varepsilon,\delta}| dV,
\end{split}
\end{align}
where $\mathcal{B}= (U^0_{r_1}\setminus U^0_{r_0})\cup(V^0_{r_1}\setminus V^0_{r_0})$ and $I_{\varepsilon,\delta}=\{t\in[-1,1]\mid \{u_{\varepsilon,\delta}=t\}\cap(\mathcal{S}_{\delta}\cup\mathcal{N}_{\delta})\neq \emptyset\}$. Note that we have used the coarea formula, and an estimate similar to \eqref{fuuuuu} in order to arrive at \eqref{integral estimate 4}.
Since $U^i_\delta\subset B_{C_\ell \delta}(N_i)$, $V^i_\delta\subset B_{C_\ell \delta}(S_i)$, the gradient estimate Theorem \ref{T:Gradient Estimate} implies
\begin{equation}\label{measure}
    |I_{\varepsilon,\delta}|
    \le\sum_{i=1}^k 2C_\ell\delta \sup_{U^i_\delta}|\nabla u_{\varepsilon,\delta}|
+ \sum_{i=1}^l 2C_\ell\delta \sup_{V^i_\delta}|\nabla u_{\varepsilon,\delta}|\le C_1\delta,
\end{equation}
where $C_1$ depends on $k$, $l$,  $C_\ell$, and a Ricci curvature lower bound in neighborhoods of $\mathcal{S}_\delta\cup\mathcal{N}_\delta$. 
Moreover by Lemma \ref{topangnha} and $H_2(M^3,\Z)=0$, we have $\chi(\Sigma_t)\le2$.
Hence \eqref{integral formula LL}, \eqref{integral estimate 4}, and \eqref{measure} yield
\begin{align} \label{Hessian LL}
\begin{split}
     \int_{M^3_{\varepsilon_1}}\frac{1}{2}\frac{|\overline{\nabla}^2u_{\varepsilon,\delta}|^2}{|\nabla u_{\varepsilon,\delta}|}dV
     \le &4\pi C_1\delta+
    C_{r_0,r_1}\varepsilon \int_{\mathcal{B}}|\nabla u_{\varepsilon,\delta}| dV
    \\ \le &4\pi C_1\delta+C_{r_0,r_1}\varepsilon|\mathcal{B}|\sup_{\mathcal{B}}|\nabla u_{\varepsilon,\delta}|.
\end{split}
\end{align}

The next step is to take a limit of $u_{\varepsilon,\delta}$.
Fixing $\varepsilon$ and letting $\delta\to 0$, $f_{\varepsilon,\delta}$ are uniformly bounded on any compact subset in $M^3_{\varepsilon_1}\setminus(\ell^{-1}(N)\cup \ell^{-1}(S))$.
Thus, $u_{\varepsilon,\delta}$ subsequently converges to $u_\varepsilon$ in $C^{2,\alpha}$, for some $\alpha\in(0,1)$. 
Moreover, we have $\lim_{\delta\to 0}\sup_\mathcal{B} |\nabla u_{\varepsilon,\delta}|> 0$. Otherwise, $u_\varepsilon$ is a constant by the Hopf Lemma, which contradicts the Dirichlet condition on $\partial M^3_{\varepsilon_1}$. 
Therefore, for any $\varepsilon>0$, we can choose a $\delta=\delta(\varepsilon)$ with $0<\delta<\varepsilon$ such that $\delta<\varepsilon \sup_{\mathcal{B}}|\nabla u_{\varepsilon,\delta}|$. In what follows we will assume that $\delta$ is chosen in this way.

Let $\Omega$ be a connected compact subset of $M^3\backslash (\ell^{-1}(N)\cup \ell^{-1}(S))$ containing $\mathcal{B}$. We scale $u_{\varepsilon,\delta}$ to prevent $u_{\varepsilon,\delta}$ from converging to a constant as $\varepsilon$ and $\delta$ tend to $0$. Namely, fix $p\in \Omega$ and let 
\begin{align}
    \hat{u}_{\varepsilon,\delta}(x):=(\sup_\Omega |\nabla u_{\varepsilon,\delta}|)^{-1}(u_{\varepsilon,\delta}(x)-u_{\varepsilon,\delta}(p)).
\end{align}
Then $\hat{u}_{\varepsilon,\delta}$ is still a solution to 
$\Delta \hat{u}_{\varepsilon,\delta}+f_{\varepsilon,\delta}|\nabla \hat{u}_{\varepsilon,\delta}|=0$, and thus integral formulas such as \eqref{Hessian LL} still hold for $\hat{u}_{\varepsilon,\delta}$.  
On $\Omega$, we have $|\nabla \hat{u}_{\varepsilon,\delta}|\le 1$ and $|\hat{u}_{\varepsilon,\delta}|\le \textnormal{diam}(\Omega)$. 
Since $f_{\varepsilon,\delta}$ is uniformly bounded on $\Omega$, we have uniform $C^{2,\alpha}$ estimates for $\hat{u}_{\varepsilon,\delta}$. 
By passing a subsequence, $\hat{u}_{\varepsilon,\delta}$ converges to $u$ in $C^{2,\beta}$ for $\beta\in(0,\alpha)$ with $f_{\varepsilon,\delta}(x)\to f(x):= \cot(\theta\circ \ell(x))$ on $\Omega$ for $\varepsilon\to 0$. 
Because $\sup_\Omega |\nabla\hat{u}_{\varepsilon,\delta}|=1$, we also have $\sup_{\Omega}|\nabla u|=1$.
Thus, $u$ is not a constant function.
Define $\Omega_i=\{x\in\Omega: |\nabla u(x)|\ge i^{-1}\}$. 
For $\varepsilon$ small enough, we obtain $|\nabla \hat{u}_{\varepsilon,\delta}|\ge \frac{1}{2i}$ in $\Omega_i$.
Moreover
\begin{equation}
    \lim_{\varepsilon\to0}\frac{|\overline{\nabla}^2 \hat{u}_{\varepsilon,\delta}|^2}{|\nabla \hat{u}_{\varepsilon,\delta}|}(x)=\frac{|\overline{\nabla}^2 u|^2}{|\nabla u|}(x), \quad\textnormal{\;for all\;} x\in\Omega_i.
\end{equation}
Using equation \eqref{Hessian LL} and applying Fatou's lemma produces
\begin{equation}
  0=\liminf_{\varepsilon\to 0} \int_{\Omega}\frac{|\overline{\nabla}^2 \hat{u}_{\varepsilon,\delta}|^2}{|\nabla \hat{u}_{\varepsilon,\delta}|}dV\ge
 \liminf_{\varepsilon\to 0} \int_{\Omega_i}\frac{|\overline{\nabla}^2 \hat{u}_{\varepsilon,\delta}|^2}{|\nabla \hat{u}_{\varepsilon,\delta}|}dV\ge
    \int_{\Omega_i}\frac{|\overline{\nabla}^2 u|^2}{|\nabla u|}dV. 
\end{equation}
Therefore, $\overline{\nabla}^2 u=\nabla^2 u+f|\nabla u|g=0$ on $\Omega_i$, for any $i\in \N$. Suppose $x_0\in \Omega$ satisfies $|\nabla u|(x_0)=1$. For any $q\in \Omega$, let $\gamma:[0,1]\to \Omega$ be a curve connecting $x_0$ and $q$, then $\nabla_{\gamma'}|\nabla u|=-f\nabla_{\gamma'} u$ holds wherever $|\nabla u|\neq 0$. The ODE for $|\nabla u|$ implies $|\nabla u|(q)\ne0$, and therefore, $|\nabla u|\neq 0$ on $\Omega$. To finally see that $(M^3,g)$ must be the round sphere, we use the fact that $\ell$ satisfies the same equation as \eqref{l equation}
\begin{equation}
    {\bigg{\langle}} D\ell\left(\frac{\nabla u}{|\nabla u|}\right), \tilde{\nabla} \theta {\bigg{\rangle}}=-1,\quad |\det (D\ell|_{\Sigma_t})|=1, \textnormal{\;on\;} \Omega. 
    \label{LL equality}
\end{equation}
Consequently, after taking a exhaustion of $M^3$, $\ell$ locally isometric on $M^3$. Since $\ell$ is also a proper, it is a covering map onto $S^3$. It follows that $\ell$ is a global isometry.

The proof will be complete upon establishing the general inequality \eqref{e:Llarullquant} in Theorem \ref{Llarull}. Its proof requires only slight modifications to the above argument, which will now be explained. In the sequence of inequalities \eqref{integral estimate 4}, we no longer throw away the term involving $R-6$. Next, instead of considering the rescaled functions $\hat{u}_{\varepsilon,\delta}$, we directly estimate and take a limit of $u_{\varepsilon,\delta}$, in a similar fashion to the final steps in the proof of Theorem \ref{thm:waist}. 
A standard diagonal argument utilizing the uniform $C^0$ control of $u_{\varepsilon,\delta}$ from the maximum principle shows that a subsequence of $u_{\varepsilon,\delta}$ converges in $C^{2,\beta}$, $\beta\in(0,\alpha)$, to a spacetime harmonic function $u$ on any subset of $M^3$ which is compactly contained in the complement of $\ell^{-1}(N)\cup\ell^{-1}(S)$. Furthermore, in order to facilitate the interchange of limit and integral, notice that by applying Lemma \ref{L:A convergence} near $N_0$, $S_0$ and Theorem \ref{T:Gradient Estimate} near $N_1,\dots, N_k$ and $S_1,\dots,S_l$, we obtain a global uniform gradient bound for $u_{\varepsilon,\delta}$ throughout $M^3_{\varepsilon_1}$.
It should by pointed out that the hypothesis \eqref{eq:condition f} of Lemma \ref{L:A convergence} with $C\geq \frac{n-1}{n}=\frac{2}{3}$ may be confirmed directly from the definition of $f_{\delta,\varepsilon}$ with the aid of a Taylor expansion of cotangent together with the Lipschitz norm restriction for $\ell$, while the hypothesis of Theorem \ref{T:Gradient Estimate} involving $C_0$ is satisfied by virture of \eqref{e:llarull69}-\eqref{e:llarull612}.
Moreover, the proof of Lemma \ref{L:A convergence} implies something slightly stronger, namely $\pm u_{\varepsilon,\delta}\geq\tfrac12$ in a small but uniform neighborhood of $N_0$ and $S_0$, respectively, which crucially implies that $u$ is nonconstant. Combined with the arguments in the previous paragraph, 
we may take a limit of the integral inequality \eqref{integral formula LL} as $\varepsilon,\delta\to0$ to find
\begin{equation}
    \int_{M^3}|\nabla u|(6-R)dV\geq\int_{M^3}\frac{{\big{|}}\nabla^2u+\cot(\theta\circ\ell)|\nabla u|\;g{\big{|}}^2}{|\nabla u|}dV.
\end{equation}
Note that the uniform gradient bounds allow for an application of the dominated convergence theorem to give the left-hand side of this inequality, whereas the right-hand side is obtained by Fatou's lemma as before.



\subsection{Proof of Main Theorem \ref{LlarullIncomplete}}

Recall that $\theta(x)=d_{{S^3}}(x,N)$ and let $S^3_\varepsilon=\{x\in S^3:\varepsilon\le\theta(x)\le\pi-\varepsilon\}$.
Then $\{S^3_\varepsilon\}_{\varepsilon>0}$ is an exhaustion of $S^3\setminus\{N,S\}$.
Fix a constant $0<r_0<\frac{\pi}{2}$.
Similar to the previous proof and Lemma \ref{construct f}, there exist functions $f_\varepsilon\in \textnormal{Lip}(S^3_\varepsilon)$ such that
\begin{align}
\begin{split}
    3+3f^2_\varepsilon-2|\nabla f_\varepsilon|\ge \csc^2\theta \quad\quad  & \text{almost everywhere in $S^3_\varepsilon\backslash S^3_{r_0}$},
 \\ 3+3f^2_\varepsilon-2|\nabla f_\varepsilon|\ge\csc^2\theta-C\varepsilon \quad\quad &\text{almost everywhere in $ S^3_{r_0}$},
 \\ H+2f_\varepsilon\ge 0 \quad\quad&\text{on $\{\theta(x)=\varepsilon\}$}, \\
H- 2f_\varepsilon\geq 0 \quad\quad&\text{on $\{\theta(x)=\pi-\varepsilon\}$},\\
f_\varepsilon\to \cot\theta\quad\quad&\text{in }C^0_{loc}(S^3\setminus\{N,S\}) \text{ as }\varepsilon\to0,
 \end{split}
\end{align}
where $H$ is with respect to the unit outer normal and the constant $C>0$ is independent of $\varepsilon$.
Let $u_\varepsilon$ be the solution to the spacetime harmonic equation
\begin{align}
\begin{cases}
   \Delta u_\varepsilon+3f_{\varepsilon}|\nabla u|=0 \quad\quad&\text{in $S_\epsilon^3$}, \\
   u_\varepsilon=1 \quad\quad&\text{on $\{\theta=\varepsilon\}$},\\
   \;u_\varepsilon=-1 \quad\quad&\text{on $\{\theta=\pi-\varepsilon\}$}.
   \end{cases}
\end{align}
As in the previous argument above line \eqref{Hessian LL}, the regular level sets $\Sigma_t$ of $u_\varepsilon$ satisfy $\chi(\Sigma_t)\leq2$, and we obtain the same estimates as in \eqref{integral estimate 4}. 
We can now implement the arguments at the end of the proof of Theorem \ref{Llarull}, taking $\ell$ to be the identity map on $S^3\setminus\{N,S\}$.  It follows that $\ell$ is an isometry, yielding the desired result.





\section{2-Ricci Positive Bands}\label{S:HopfSphere}

In this section we study $3$-dimensional bands with positive $2$-Ricci curvature. In the preceding three sections, the integral inequality of Lemma \ref{integralformula} was fundamental in analysing the positive scalar curvature condition. The following lemma provides a modification of this integral inequality suited to the positive $2$-Ricci curvature condition.

\begin{lemma}\label{l:2Ricidentity}
Let $(M^3,\partial_\pm M^3,g)$ be a $3$-dimensional Riemannian band, and let $f\in \mathrm{Lip}(M^3)$.  If $u\in C^{2,\alpha}$, $\alpha\in(0,1)$ solves boundary value problem \eqref{e:bandspacetimeharmoniceq}, then 
\begin{align}\label{e:2ricciint}
\begin{split}
   & \int_{\partial_-M^3}|\nabla u|\left(\frac{ 3}{2}f-H\right)dA-\int_{\partial_+M^3}|\nabla u|\left(\frac{3 }{2}f+H\right)dA+\int_{-1}^1 4\pi\chi(\Sigma_t)dt\\
    \geq & \int_{-1}^{1}\int_{\Sigma_t}\left(\frac{|\nabla_\Sigma|\nabla u||^2+(\nabla_{\nu\nu}u+\frac32 f|\nabla u|)^2}{|\nabla u|^2}
    +\left(R-\mathrm{Ric}(\nu,\nu)+\frac{9}{4}f^2-\frac32 \langle \nabla f,\nu\rangle\right)\right)dA dt
\end{split}
\end{align}
where $\nu=\frac{\nabla u}{|\nabla u|}$ on regular level sets $\Sigma_t$, and $H$ is the mean curvature of the boundary with respect to the unit outward normal.
\end{lemma}

\begin{proof}
The calculation is similar to the proof of Lemma \ref{integralformula}. The key here is the following unusual application of the Gauss and Codazzi equations, first known to the present authors from \cite{Zhu}. Along regular $u$-level sets
\begin{equation}\label{e:riccigauss}
\Ric(\nu,\nu)=R-\Ric(\nu,\nu)+H^2-|II|^2-2K,
\end{equation}
where $II$, $H$, and $K$ denote the second fundamental form, mean curvature, and Gauss curvature respectively, of the level sets.
Combining \eqref{e:riccigauss} with the fact that $II=\tfrac{\nabla^2 u|_\Sigma}{ |\nabla u|}$ and $\Delta u+3f|\nabla u|=0$, yields
\begin{align}\label{e:hopfbochner1}
\begin{split}
    &\Ric(\nu,\nu)\\
    =&R-\Ric(\nu,\nu)+\frac1{|\nabla u|^2}(\Delta u-\nabla_{\nu\nu}u)^2-\frac{|\nabla^2 u|_\Sigma|^2}{|\nabla u|^2}-2K\\
    =&R-\Ric(\nu,\nu)+\frac1{|\nabla u|^2}(\Delta u-\nabla_{\nu\nu}u)^2-\frac1{|\nabla u|^2}(|\nabla^2u|^2-2|\nabla |\nabla u||^2+(\nabla_{\nu\nu}u)^2)-2K\\
     =&R-\Ric(\nu,\nu)+\frac1{|\nabla u|^2}\left((\Delta u)^2-2\Delta u \nabla_{\nu\nu}u-|\nabla^2u|^2+2|\nabla |\nabla u||^2\right)-2K\\
    =&R-\Ric(\nu,\nu)+9f^2+\frac6{|\nabla u|}f \nabla_{\nu\nu}u-\frac1{|\nabla u|^2}(|\nabla^2u|^2-2|\nabla |\nabla u||^2)-2K
    \end{split}
\end{align}
along regular level sets.

In order to deal with critical points of $u$, let $\delta$ be a positive parameter and consider the quantity $\varphi=\sqrt{|\nabla u|^2+\delta}$. Calculating with the Bochner formula,
\begin{align}
\begin{split}
    \Delta \varphi=&\frac{1}{\varphi} \left[\frac{1}{2}\Delta|\nabla u|^2-\frac{|\nabla u|^2}{\varphi^2}|\nabla|\nabla u||^2 \right]\\
    \ge&\frac1{\varphi}(|\nabla^2u|^2+\Ric(\nabla u,\nabla u)-|\nabla |\nabla u||^2+\langle \nabla u,\nabla \Delta u\rangle)\\
    =&\frac1{\varphi}(|\nabla^2u|^2+\Ric(\nabla u,\nabla u)-|\nabla |\nabla u||^2-3|\nabla u|\langle \nabla u,\nabla f\rangle-3f|\nabla u|\nabla_{\nu\nu}u) 
     \label{varphi}.
     \end{split}
\end{align}
Combining \eqref{e:hopfbochner1} and \eqref{varphi} gives
\begin{align}\label{e:hopfbochner2}
\begin{split}
    \Delta \varphi\ge&\frac1{\varphi}(|\nabla^2u|^2-|\nabla |\nabla u||^2-3|\nabla u|\langle \nabla u,\nabla f\rangle-3f|\nabla u|\nabla_{\nu\nu}u)\\
&+\frac{|\nabla u|^2}{\varphi}\left(R-\Ric(\nu,\nu)+9f^2+\frac6{|\nabla u|}f \nabla_{\nu\nu}u-\frac1{|\nabla u|^2}(|\nabla^2u|^2-2|\nabla |\nabla u||^2)-2K\right)\\
=&\frac1{\varphi}(|\nabla |\nabla u||^2-3|\nabla u|\langle \nabla u,\nabla f\rangle+3f|\nabla u|\nabla_{\nu\nu}u)+\frac{|\nabla u|^2}{\varphi}\left(R-\Ric(\nu,\nu)+9f^2-2K\right).
\end{split}
\end{align}
Moreover, inserting the identity
\begin{align}
    \frac92f^2|\nabla u|=-\frac32f\Delta u=-\frac32\div(f\nabla u)+\frac32\langle \nabla f,\nabla u\rangle
\end{align}
into \eqref{e:hopfbochner2} leads to 
\begin{align}\label{e:hopfbochner3}
\begin{split}
\Delta\varphi\ge&\frac1{\varphi}\left(|\nabla |\nabla u||^2-\frac32|\nabla u|\langle \nabla u,\nabla f\rangle+3f|\nabla u|\nabla_{\nu\nu}u\right)
\\&+\frac{|\nabla u|^2}{\varphi}\left(R-\Ric(\nu,\nu)+\frac92f^2-2K\right)-\frac{3|\nabla u|}{2\varphi}\div(f\nabla u).
\end{split}
\end{align}
Finally, we calculate the square
\begin{equation}
    \left(\nabla_{\nu\nu}u+\frac32f|\nabla u|\right)^2=(\nabla_{\nu\nu}u)^2+\frac94f^2|\nabla u|^2+3f|\nabla u|\nabla_{\nu\nu}u,
\end{equation}
and combine with \eqref{e:hopfbochner3} to arrive at the primary pointwise identity
\begin{align}\label{eq:K}
\begin{split}
\Delta\varphi\ge&\frac1{\varphi}\left(|\nabla_\Sigma|\nabla u||^2+\left(\nabla_{\nu\nu}u+\frac32f|\nabla u|\right)^2-\frac32|\nabla u|\langle \nabla u,\nabla f\rangle\right)\\
&+\frac{|\nabla u|^2}{\varphi}\left(R-\Ric(\nu,\nu)+\frac94f^2-2K\right)-\frac{3|\nabla u|}{2\varphi}\div(f\nabla u).
\end{split}
\end{align}

The next step is to integrate \eqref{eq:K}, and take $\delta\to0$. Since \eqref{eq:K} only holds along regular level sets, this process is delicate. However, a similar process is carried out in \cite{HKK} and \cite{Stern}, so we will be brief. Let $\mathcal{A}\subset[-1,1]$ be an open set containing the critical values of $u$ and let $\mathcal{B}\subset[-1,1]$ be its compliment. Integrate \eqref{eq:K} over $u^{-1}(\mathcal{B})$ to find
\begin{align} \label{integral 1}
\begin{split}
&\int_{u^{-1}(\mathcal{B})}\left(\Delta\varphi+\frac{3|\nabla u|}{2\varphi}\div(f\nabla u)\right)dV
\\ \ge&\int_{u^{-1}(\mathcal{B})}\frac1{\varphi}\left(|\nabla_\Sigma|\nabla u||^2+\left(\nabla_{\nu\nu}u+\frac32f|\nabla u|\right)^2\right)dV
\\&+\int_{u^{-1}(\mathcal{B})}\frac{|\nabla u|^2}{\varphi} \left(R-\Ric(\nu,\nu)+\frac{9}{4}f^2 -2K-\frac{3}{2}\langle\nabla f,\nu\rangle\right)dV.
\end{split}
\end{align}
To help control the integral over $u^{-1}(\mathcal{A})$, we return to $\Delta\varphi$ and estimate it in a different way. Along regular level sets we may use \eqref{varphi} and $|\nabla ^2u|_\Sigma|^2\ge \frac{1}{2}|\Tr(\nabla^2u|_\Sigma)|^2$ to obtain
\begin{align}\label{e:badset1}
\begin{split}
\Delta \varphi\ge&\frac1{\varphi}(|\nabla^2u|^2+\Ric(\nabla u,\nabla u)-|\nabla     |\nabla u||^2-3|\nabla u|\langle \nabla u,\nabla f\rangle-3f|\nabla                 u|\nabla_{\nu\nu}u)\\ 
\ge &\frac{1}{\varphi}(|\nabla^2u|_\Sigma|^2-C_1|\nabla u|^2-3|\nabla u|^2 |\nabla f|-3f|\nabla u|\nabla_{\nu\nu}u)\\ 
\ge & \frac{1}{\varphi}\left(\frac{1}{2}(\nabla_{\nu\nu} u+3f|\nabla u|)^2-3f|\nabla u|\nabla_{\nu\nu}u\right)-(C_1+3|\nabla f|)|\nabla u|\\ 
\ge& -C_2|\nabla u|,
\end{split}
\end{align}
where $C_1$ and $C_2$ are positive constants depending only on $|\Ric|$ and $|\nabla f|$. Similarly, the other integrand on the left-hand side of \eqref{integral 1} can be controlled by
\begin{align}\label{e:badset2}
    \div(f\nabla u)= f \Delta u +\langle \nabla f, \nabla u\rangle
     \ge  -3 f^2 |\nabla u|-|\nabla f||\nabla u|
    \ge -C_3|\nabla u|,
\end{align}
where $C_3>0$ only depends on $|\nabla f|$ and $f$. As a consequence of the coarea formula and Sard's Theorem, we may restrict attention to regular level sets and apply \eqref{e:badset1}, \eqref{e:badset2}, to find
\begin{align}
\begin{split}
\int_{u^{-1}(\mathcal{A})} \left(\Delta\varphi+\frac{3}{2}\div(f\nabla u)\right)dV=& \int     _{\mathcal{A}}\int_{\Sigma_t}\frac{1}{|\nabla u|}\left(\Delta\varphi+\frac{3}{2}\div(f\nabla u)\right)dAdt
\\ 
\ge &-\int_{\mathcal{A}} \int_{\Sigma_t} \left(C_2+\frac{3C_3}{2}\right)dAdt \\
=&-\int_{u^{-1}(\mathcal{A})}\left(C_2+\frac{3C_3}{2}\right)|\nabla u|dV.
\end{split}
\end{align}
Therefore, a further application of \eqref{e:badset1}, \eqref{e:badset2} shows that
\begin{align}
\begin{split}
   & \int_{u^{-1}(\mathcal{B})}\left(\Delta\varphi+\frac{3|\nabla u|}{2\varphi}\div(f\nabla u)\right)dV \\ 
 \le& \int_{M^3} \left(\Delta\varphi+\frac{3}{2}\div(f\nabla u)\right)dV + \int_{u^{-1}(\mathcal{A})} \left(C_2+\frac{3C_3}{2}\right)|\nabla u|dV\\
 &+\int_{u^{-1}(\mathcal{B})} \frac{3C_3}{2}\left(1-\frac{|\nabla u|}{\varphi}\right)|\nabla u|dV.
    \label{third term}
    \end{split}
\end{align}

Let us now integrate the first term on the right-hand side in \eqref{third term} by parts, apply \eqref{integral 1}, and rearrange the inequality to find
\begin{align}
\begin{split}
& \int_{\partial M^3} \left(\frac{|\nabla u|}{\varphi}\mathbf{n}(|\nabla u|)+\frac{3}{2}f\mathbf{n}( u) \right)dA\\
  \ge & \int_{u^{-1}(\mathcal{B})}\frac1{\varphi}\left(|\nabla_\Sigma|\nabla u||^2+\left(\nabla_{\nu\nu}u+\frac32f|\nabla u|\right)^2\right)dV\\
  &+\int_{u^{-1}(\mathcal{B})}\frac{|\nabla u|^2}{\varphi}\left(R-\Ric(\nu,\nu)+\frac{9}{4}f^2-2K-\frac{3}{2}\langle \nabla f,\nu\rangle\right)dV\\
  &-\int_{u^{-1}(\mathcal{A})} \left(C_2+\frac{3C_3}{2}\right)|\nabla u|dV -\int_{u^{-1}(\mathcal{B})} \frac{3C_3}{2}\left(1-\frac{|\nabla u|}{\varphi}\right)|\nabla u| dV,
  \end{split}\label{e:ibpgood}
\end{align}
where $\mathbf{n}$ denotes the unit outward normal to $\partial M^3$.
This inequality will be applied with a sequence $\mathcal{A}_i\subset[-1,1]$ such that $\lim_{i\to\infty}|\mathcal{A}_i|=0$, which is permissible by Sard's Theorem, and with $\mathcal{B}_i=\mathcal{A}^c_i$. Notice that first taking $\delta\to0$ ensures that the last term of \eqref{e:ibpgood} tends to zero. Furthermore, by the coarea formula and the Gauss-Bonnet Theorem we find
\begin{align}
    \lim_{\delta\to0}\int_{u^{-1}(\mathcal B_i)}\frac{|\nabla u|^2}\varphi K dV =\int_{\mathcal B_i} 2\pi\chi(\Sigma_t)dt.
\end{align}
Now observe that Lemma \ref{topangnha} implies that the Euler characteristic of a regular level set $\chi(\Sigma_t)$ is bounded from above uniformly in $t$. This allows us to apply the Reverse Fatou's Lemma. To explain this step, consider the function $F_i(t)$ which takes the value $\chi(\Sigma_t)$ if $t\in\mathcal{B}_i$, and $0$ otherwise. Arguing as in Remark \ref{alkjflkjah}, the functions $F_i$ are measurable, and hence
 \begin{align}\label{e:limiteul}
     \limsup_{i\to \infty}\int_{\mathcal{B}_i}\chi(\Sigma_t)dt =\limsup_{i\to \infty}\int_{-1}^1  F_i(t)dt
     \le \int^1_{-1}\limsup_{i\to\infty} F_i(t) dt
     = \int^1_{-1}  \chi(\Sigma_t)dt.
 \end{align}
To finish, take a limsup in $i$ of \eqref{e:ibpgood} and use Fatou's Lemma for the first term on the right-hand side, as well as the dominated convergence theorem for all remaining integrals. 
Lastly, noting that $\mathbf{n}=\pm\nu$ on $\partial_\pm M^3$ in the boundary term, and using the spacetime harmonic equation as in the proof of Lemma \ref{L:Meyer integral formula}, we arrive at the desired integral identity.
\end{proof}

\subsection{Proof of Main Theorem \ref{HopfSphereBoundary}}

Let $(M^3,\partial_\pm M^3,g)$ be as in the hypotheses of Main Theorem \ref{HopfSphereBoundary}, and suppose that the width $w=d(\partial_-M^3,\partial_+M^3)$ satisfies $w\geq\arctan(H_0/2)$. If $H_0\leq0$, one may follow the arguments below with $f\equiv 0$ to obtain a contradiction. Therefore, we will now assume that $H_0>0$ and choose $f$ in a different way. Namely set  
\begin{equation}
    \tilde{f}(\tau)=\begin{cases}
   \frac{4}{3} \tan(2\tau-\text{arctan}(H_0/2))&\text{ if }\tau\leq \arctan(H_0/2),\\
    L(\tau)&\text{ otherwise},
    \end{cases}
\end{equation}
where $L$ is the unique $1$-variable linear function making $\tilde{f}$ a $C^1$ function, and define the Lipschitz function $f(x)=\tilde{f}(r(x))$ where we use the notation $r(x)=d(x,\partial_-M^3)$. This choice of $f$ ensures that
\begin{align} \label{e:2Ricband1}
    4+\frac{9}{4}f^2-\frac{3}{2}|\nabla f|\ge 0
\end{align}
holds almost everywhere. Observe that $H\pm \frac{3}{2}f\ge 0$ on $\partial_\pm M^3$, and this is a strict inequality at some point of $\partial M^3$ unless $r\equiv\arctan(H_0/2)$ on $\partial_+M^3$ and $H\equiv -H_0$ holds across $\partial M^3$. As usual, let $u$ be the unique spacetime harmonic function associated with this $f$ such that $u=\pm 1$ on $\partial_{\pm} M$, which is guaranteed by Proposition \ref{p:existence}. 
Since it is assumed that $M^3$ has no spherical classes, the Euler characteristic of any homologically nontrivial surface is nonpositive, and in particular $\chi(\Sigma_t)\leq 0$ for all regular level sets of $u$.

By combining all of the above observations, Lemma \ref{l:2Ricidentity} implies that
\begin{equation}\label{e:2RicBand2}
    0\geq \int_{-1}^1\int_{\Sigma_t}\left(\frac{|\nabla_\Sigma|\nabla u||^2+(\nabla_{\nu\nu}u+\frac32 f|\nabla u|)^2}{|\nabla u|^2}+R-\mathrm{Ric}(\nu,\nu)-4\right)dAdt,
\end{equation}
where we have used the Cauchy-Schwarz inequality $\langle\nabla f,\nu\rangle\le |\nabla f|$. 
Since $R-\mathrm{Ric}(\nu,\nu)$ is at least $4$, the right-hand side of \eqref{e:2RicBand2} is nonnegative. It follows that the boundary term of Lemma \ref{l:2Ricidentity} vanishes and, as discussed above, the width estimate $w\leq\arctan(H_0/2)$ follows.

Now assume that $\mathrm{Ric}\geq2g$, and $w=\arctan(H_0/2)$. Let us collect all the information gained from attaining equality in the inequalities leading to the width estimate above. Inspecting the integrand of \eqref{e:2RicBand2}, we find that 
\begin{align}\label{e:hopfbandrigid}
    \nabla_\Sigma|\nabla u|\equiv 0,\qquad \nabla_{\nu\nu}u+\frac32f|\nabla u|\equiv 0,\qquad R-\Ric(\nu,\nu)=4,
\end{align}
whenever $|\nabla u|\neq0$. As is typical at this stage, $\nabla u$ is actually nonzero everywhere. Indeed, by the Hopf Lemma and Dirichlet conditions, $\mathbf{n}(u)=\pm|\nabla u|$ is nonzero on $\partial_\pm M^3$. To see that $\nabla u\neq0$ on the interior of $M^3$, the first two equations of \eqref{e:hopfbandrigid} can be used to consider an ODE for $|\nabla u|$ along a curve connecting a boundary point to any interior point, allowing us to conclude that it is impossible for $|\nabla u|$ to become zero. Moving along, from equality in line \eqref{e:2RicBand2}, we know that $\nabla f$ is parallel to $\nabla u$ wherever $\nabla f$ exists. It follows that $\nabla r=\nu$ almost everywhere. Since $\nu$ is $C^{1,\alpha}$, $r$ cannot have any critical points. Lastly, since the boundary term of \eqref{e:2ricciint} must vanish, we know that the boundary mean curvature takes the constant value $H\equiv -H_0$.

Now we investigate the consequences of the above information. 
Consider an orthonormal frame $\{e_1,e_2,e_3\}$ where $e_3=\nabla r$. Greek indices will be reserved for $e_1,e_2$, while Latin indices will be used when referring to all three vector fields.
Because $\nabla r=\nu$, it holds that $u$ can be viewed as a function of $r$. Let $H_r$ be the mean curvature of the level set of $r$ with respect to $\nabla r$, then using $\nabla_{\nu\nu} u+\frac{3}{2}f|\nabla u|=0$, we obtain
\begin{equation} \label{Hr}
    H_r=\frac{\Delta u-\nabla_{\nu\nu} u}{|\nabla u|}=-\frac{3}{2}f=-2\tan \left(2r-\arctan\frac{H_0}{2}\right).
\end{equation}
Moreover, since $\Ric\ge2g$ and $R-\mathrm{Ric}_{33}=4$, we find that $\Ric_{11}=2$ and $\Ric_{22}=2$. 
It follows that
\begin{equation}
  2(1+\varepsilon^2)\le  \Ric(e_\alpha+\varepsilon e_3, e_\alpha+\varepsilon e_3)=2+ 2\varepsilon \Ric_{\alpha 3}+\varepsilon^2\Ric_{33}
\end{equation}
for any $\varepsilon\in \R$, which forces $\Ric_{\alpha 3}=0$ for $\alpha =1,2$. A similar argument shows that $\Ric_{12}=0$. 

Next, the contracted second Bianchi identity and $R-\Ric_{33}=4$ imply that
\begin{equation}\label{e:riccicalc1}
    \nabla_3\Ric_{33}+\nabla_\alpha\Ric_{3\alpha}=\frac12e_3(R)=\frac12 e_3(\Ric_{33}).
\end{equation}
Furthermore, using the facts that $\langle \nabla_{e_3}e_3,e_3\rangle=0$ and $\Ric_{3\alpha}=0$, we have
\begin{align}\label{e:riccicalc2}
    \nabla_3 \Ric_{33}= e_3(\Ric_{33})-2\Ric (\nabla_{e_3}e_3, e_3)=e_3(\Ric_{33}).
\end{align}
Since $\langle \nabla_{e_\alpha}e_3,e_\alpha\rangle=-\langle \nabla_{e_\alpha}e_\alpha,e_3\rangle=H_r$, the previous calculations of the Ricci tensor produce
\begin{align}\label{e:riccicalc3}
\begin{split}
    \nabla_\alpha \Ric_{3\alpha}=&e_\alpha (\Ric_{3\alpha})-\Ric(\nabla_{e_\alpha}e_3,e_\alpha)- \Ric(e_3,\nabla_{e_\alpha}e_{\alpha})
    \\=&-2H_r+H_r\Ric_{33}.
\end{split}
\end{align}
Combining \eqref{e:riccicalc1}, \eqref{e:riccicalc2}, and \eqref{e:riccicalc3} then shows that $\Ric_{33}$ satisfies the equation
\begin{align}\label{h ODE}
    \partial_r\Ric_{33}+2H_r\Ric_{33}=4H_r.
\end{align}

It will now be established that $\Ric_{33}$ is constant on level sets. To see this, first observe that
\begin{align}
\begin{split}
    \frac{1}{2}e_1(\Ric_{33})=&\frac{1}{2}e_1(R)
    \\=&\nabla_{e_3}\Ric_{13}+\nabla_{\alpha}\Ric_{1\alpha}
    \\=&-\Ric(\nabla_{e_3}e_1,e_3)-\Ric(e_1,\nabla_{e_3}e_3)-\Ric(\nabla_{e_\alpha}e_1,e_\alpha)-\Ric(e_1,\nabla_{e_\alpha}e_\alpha).
\end{split}
\end{align}
Since
\begin{equation}
    \langle e_1,\nabla_{e_3}e_3\rangle=\left\langle e_1,\nabla_{e_3}\frac{\nabla u}{|\nabla u|}\right\rangle=\frac{\nabla^2 u(e_1,e_3)}{|\nabla u|}=\frac{\nabla_{e_1}|\nabla u|}{|\nabla u|}=0,
\end{equation}
it follows from calculations above that $\Ric(\nabla_{e_3}e_1,e_3)=0$ and $\Ric(e_1,\nabla_{e_3}e_3)=0$.
Furthermore
\begin{align}
    \Ric(\nabla_{e_\alpha}e_1,e_\alpha)+\Ric(e_1,\nabla_{e_\alpha}e_\alpha)=
    \Ric_{22}\langle \nabla_{e_2}e_1,e_2\rangle+
    \Ric_{11}\langle e_1,\nabla_{e_2}e_2\rangle
    =2\nabla_{e_2}\langle e_1,e_2\rangle=0.
\end{align}
Therefore $e_1(\Ric_{33})=0$, and similarly $e_2(\Ric_{33})=0$, yielding the desired constancy.

Using \eqref{Hr}, \eqref{h ODE}, and $\Ric_{33}\ge 2$, we now obtain $\Ric_{33}=2+c_0\sec^2(2r-\arctan\frac{H_0}{2})$ for some constant $c_0\ge0$. 
Subtracting the identities
\begin{align}
2=\Ric_{11}&=\sigma(e_1,e_2)+\sigma(e_1,e_3),\\
2=\Ric_{22}&=\sigma(e_1,e_2)+\sigma(e_2,e_3),
\end{align}
produces a coincidence of sectional curvatures $\sigma(e_1, e_3)=\sigma(e_2, e_3)$ and, subsequently $\Ric_{33}=2\sigma(e_1, e_3)$. Introducing the shifted coordinate $\rho=r-\frac{1}{2}\arctan\frac{H_0}{2}$, we will abuse notation and use $\Sigma_\rho$ to denote the level sets of $\rho$. As a consequence of the above, we obtain a simple expression for the sectional curvature 
\begin{equation}\label{e:sec13}
\sigma(Y, e_3)=\frac12 \Ric_{33}=1+\frac{1}{2}c_0\sec^2(2\rho),\quad\quad  Y\in T\Sigma_\rho.
\end{equation}

Following \cite[Alternative Proof of Corrollary 1.7]{Zhu}, we consider the Riccati equation
\begin{equation}\label{Ricatti}
   \partial_\rho (\nabla^2 \rho)+(\nabla^2 \rho)^2=- R(\cdot,\nabla \rho)\nabla \rho
\end{equation}
where, crucially, $\nabla^2\rho$ is viewed as an endomorphism of $T^*M^3$ and $(\nabla^2 \rho)^2$ is a composition of endomorphisms.
Let $\lambda_1\le \lambda_2$ denote the principle curvatures of $\Sigma_\rho$, considered as functions on $M^3$. By  \eqref{e:sec13} and \eqref{Ricatti}, their restrictions to an integral curve of $\nabla\rho$ satisfy 
\begin{equation}\label{e:eigenvalueeq}
    \lambda_\alpha'+\lambda_\alpha^2=-1-\frac{1}{2}c_0\sec^2(2\rho),\quad \alpha=1,2,
\end{equation}
where $'$ denotes the $\partial_\rho$ derivative.
Note that \eqref{Hr} implies
\begin{equation}\label{e:meancurv}
\lambda_1+\lambda_2=H_r=-2\tan (2\rho),
\end{equation}
and thus $\lambda_1'+\lambda_2'=-4\sec^2(2\rho)$. 
Summing \eqref{e:eigenvalueeq} with both $\alpha=1,2$, we are lead to
\begin{equation}\label{e:secondfundamental}
    |II|^2=\lambda_1^2+\lambda_2^2=4\sec^2(2\rho)-2-c_0\sec^2(2\rho).
\end{equation}
As a consequence of \eqref{e:meancurv} and \eqref{e:secondfundamental}, $c_0\leq2$. 
Equations \eqref{e:meancurv} and \eqref{e:secondfundamental} also allow one to algebraically solve for the principal curvatures
\begin{equation} \label{e:eigenvalue express}
\lambda_1=-\tan(2\rho)-\sqrt{1-\frac{c_0}{2}}\sec(2\rho),\quad\quad     \lambda_2=-\tan(2\rho)+\sqrt{1-\frac{c_0}{2}}\sec(2\rho),   
\end{equation}
as long as $c_0\le2$.
At this point we can apply the twice contracted Gauss equations on the surfaces $\Sigma_\rho$, to compute the Gauss curvature
\begin{equation}\label{sec7Gaussian}
\begin{split}
    K=&\frac{R}{2}-\Ric_{33}-\frac{|II|^2}{2}+\frac{H_r^2}{2}\\
    =&2-\frac12\left(2+c_0\sec^2(2\rho)\right)
    -\frac12\left(4\sec^2(2\rho)-2-c_0\sec^2(2\rho)\right)+2\tan^2 (2\rho)
    \\=&2+2\tan^2 (2\rho)-2\sec^2(2\rho)=0.
\end{split}
\end{equation}

It will be useful to consider two cases. If $c_0=2$, then \eqref{e:eigenvalue express} shows that each level set $\Sigma_\rho$ is umbilic. Then $g$ splits as $g=d\rho^2+g_\rho$, and we have the equation $\partial_\rho g_\rho=-2\tan(2\rho)g_\rho$. Integrating this equation yields $g_\rho=\cos(2\rho)g_0$, where $g_0$ is a flat metric on the torus. It follows that, upon passing to the universal cover, $(M^3,g)$ takes the form \eqref{e:g'metric} for $\delta=1$.

For the rest of the proof, assume that $c_0\in[0,2)$. In this case, $\lambda_1< \lambda_2$ and the level sets are never umbilic. Let us now work on the universal cover $\widetilde{M^3}$ where the level sets $\Sigma_\rho$ are lifted to planar surfaces $\widetilde{\Sigma_\rho}$. On $\widetilde{M^3}$, we may choose a new global orthonormal frame, also denoted by  $\{e_1,e_2,e_3=\nabla\rho\}$, such that $e_\alpha$ is the principal direction corresponding to $\lambda_\alpha$, for $\alpha=1,2$. Such a global frame exists since $\widetilde{\Sigma_\rho}$ is contractible and thus the eigenspaces of $II$ form trivial bundles.
Since $\widetilde{\Sigma_0}$ is nonumbilic, near any point we may find a local coordinate chart $\{x^1,x^2,\rho\}$ so that $\partial_1,\partial_2$ are orthogonal to $\nabla \rho$ and, along $\widetilde{\Sigma_0}$, $x^1$ and $x^2$ are lines of curvature coordinates, which is to say  
$\partial_1=c_1 e_1$ and $\partial_2= c_2 e_2$ for some smooth functions $c_1$ and $c_2$ defined locally on $\widetilde{\Sigma_0}$. 
In these coordinates, $g_\rho$ and $II$ satisfy the following evolution equations, written in terms of symmetric $(0,2)$-tensors: 
\begin{align}\label{e:gIIequations}
    \begin{cases}
   \partial_\rho II-II^2=-\left(1+\frac{1}{2}c_0\sec^2(2\rho)\right) g_\rho ,\\ 
   \partial_\rho g_\rho=2II ,\\ 
   g_0=c_1^2(dx^1)^2+c_2^2(dx^2)^2 ,\\ 
   II|_{\Sigma_0}=\lambda_1 c_1^2(dx^1)^2+\lambda_2c_2^2(dx^2)^2 . 
   \end{cases}
\end{align}
Solving \eqref{e:gIIequations} yields the diagonalized expressions
\begin{equation}
    g_\rho=\bar{c}_1^2(d x^1)^2+\bar{c}^2_2 (dx^2)^2,\quad\quad II=\lambda_1 \bar{c}_1^2 (dx^1)^2+\lambda_2 \bar{c}_2^2(dx^2)^2,
\end{equation}
where $\bar c_\alpha(\rho,x_1,x_2)=\phi_\alpha(\rho)c_\alpha(x_1,x_2)$ and the function $\phi_\alpha(\rho)$ satisfies
\begin{equation}
   \partial_\rho\phi_\alpha=\lambda_\alpha\phi_\alpha,\qquad
   \phi_\alpha(0)=1.
\end{equation} 
Using the formula for $\lambda_\alpha$, one finds the following explicit expressions
\begin{align}\label{phi alpha1}
\begin{split}
    \phi_1= &\left[\sec(2\rho)+\tan(2\rho)\right]^{-\frac{1-\Upsilon}{2}}\cos^\frac{1}{2}(2\rho)
    =2^{\frac{1-\Upsilon}{2}}\cos^{1-\Upsilon}(\rho+\frac{\pi}{4})\cos^{\frac{\Upsilon}{2}}(2\rho),\\ 
    \phi_2= &\left[\sec(2\rho)+\tan(2\rho)\right]^{\frac{1-\Upsilon}{2}}\cos^\frac{1}{2}(2\rho)
    =2^{\frac{1-\Upsilon}{2}}\sin^{1-\Upsilon}(\rho+\frac{\pi}{4})\cos^{\frac{\Upsilon}{2}}(2\rho),
\end{split}
\end{align}
where $\Upsilon=1-\sqrt{1-\frac{c_0}{2}}.$


The rest of the proof is devoted to showing that $e_1$ and $e_2$ are coordinate vector fields on $\widetilde{\Sigma_0}$. After accomplishing this, we may choose $x^\alpha$ so that $c_\alpha\equiv1$, giving the desired form of $g$. To this end, on $\widetilde{M^3}$ consider the smooth functions $a,b$ defined by $[e_1,e_2]=ae_1+be_2$.
Denoting the connection on $\widetilde{\Sigma_\rho}$ by  $\hat{\nabla}$, we compute
\begin{equation}
    \hat{\nabla}_{e_1} e_2=a e_1,\quad\hat{\nabla}_{e_2}e_1=-b e_2,\quad \hat{\nabla}_{e_1} e_1=-ae_2,\quad \hat{\nabla}_{e_2}e_2=b e_1. 
\end{equation}
Using these expressions, we may compute the level set curvature $\langle\hat{R}(e_1,e_2)e_1,e_2\rangle$ in terms of $a$ and $b$. Since the level sets are flat, it follows that
\begin{align}
    \begin{split}
      0=&\langle\hat{\nabla}_{e_1}\hat{\nabla}_{e_2}e_1-\hat{\nabla}_{e_2}\hat{\nabla}_{e_1} e_1-\hat{\nabla}_{[e_1,e_2]}e_1,e_2\rangle
    \\=&\langle\hat{\nabla}_{e_1}(-b e_2)-\hat{\nabla}_{e_2}(-ae_2)-\hat{\nabla}_{a e_1+b e_2} e_1, e_2\rangle
    \\=& e_1(-b)+e_2(a)+a^2+b^2.
    \end{split}\label{Gauss ab}
\end{align}
On the other hand, $\partial_\alpha=\bar c_\alpha e_\alpha$ and so
\begin{equation}
    0=[\partial_1,\partial_2]=[\bar c_1 e_1,\bar c_2e_2]= \bar c_1(e_1(\bar c_2))e_2-\bar c_2(e_2(\bar c_1))e_1+\bar c_1\bar c_2(ae_1+b e_2).
\end{equation}
Collecting coefficients in the above produces $\partial_{1}\bar c_2=  \bar c_1e_1(\bar c_2)=-b\bar c_1\bar c_2$ and $\partial_{2}\bar c_1=\bar c_2e_2(c_1)=a\bar c_1\bar c_2$.  Since $\bar c_\alpha(\rho,x_1,x_2)=\phi_\alpha(\rho)c_\alpha(x_1,x_2)$, we obtain
\begin{equation}
    a(\rho,x_1,x_2)=\phi_2^{-1}(\rho)a(0,x_1,x_2),\quad\quad 
    b(\rho,x_1,x_2)=\phi_1^{-1}(\rho)b(0,x_1,x_2).
\end{equation}
Combining this with equation \eqref{Gauss ab} implies
\begin{equation}\label{e:gausspart2}
\begin{split}
    0=&\bar c_1^{\;-1} \partial_{1}(-b)+\bar c_2^{\;-1} \partial_{2}(a)+a^2+b^2
    \\=& \phi_1^{-2}(\rho) c_1^{-1}(x_1,x_2)\partial_{1}(-b(0,x_1,x_2))+\phi_1^{-2}(\rho)b^2(0,x_1,x_2)
    \\&+\phi_2^{-2}(\rho) c_2^{-1}(x_1,x_2)\partial_{2}(a(0,x_1,x_2))+\phi_2^{-2}(\rho)a^2(0,x_1,x_2).
\end{split}
\end{equation}
As $\phi_1\neq \phi_2$, equation \eqref{e:gausspart2} yields the following simple identities for $a$ and $b$ along $\widetilde{\Sigma_0}$, namely
\begin{equation}
    e_1(b)=b^2,\quad\quad
    e_2(a)=-a^2.
\end{equation}
If $a\neq0$ at some point on $\widetilde{\Sigma_0}$, then $e_2(a^{-1})=1$. Because the integral curve $\gamma(t)$ of $e_2$ exists for all time, there exists a point $\gamma(t_0)$ such that $a^{-1}(\gamma(t_0))=0$, which is a contradiction. A similar argument applies to $b$, andhence $a=b\equiv0$. In particular, $[e_1,e_2]=0$ holds on $\widetilde{\Sigma_0}$, finishing the proof.




\subsection{Proof of Main Theorem \ref{t:2RicIncomplete}}

The proof will be similar to that of Main Theorem \ref{Meyer boundary cor}.  Let $\Sigma^{2}\subset M^3$ be the closed surface separating $M^3$ into two connected components $M^3_\pm$, where $E_\pm$ is contained in $M^3_{\pm}$.
Suppose that $w_-+w_+\geq\pi/2$ where $w_{\pm}$ is the minimum $\min\{\pi/2,d(\Sigma^{2},E_\pm)\}$, though we will soon see that the later quantity is always smaller. Consider the signed distance to $\Sigma^2$ given by $\varrho(x)=\pm d(x,\Sigma^2)$, when $x\in M^3_\pm$.
For $\delta>0$, consider the band $({\widetilde{ M^3_\delta}},\partial_\pm \widetilde{ M_\delta^3},g)$ given by
\begin{align}
    \widetilde {M_\delta^3}=
    \{x\in M^3 \mid \varrho(x)\in[-w_-+\delta,w_+-\delta]\},
\end{align}
where the assignment $\partial_\pm M^3_\delta$ respects $E_\pm$. Importantly, $\widetilde{M_\delta^3}$ is compact as a consequence of Lemma \ref{l:compactification}. Let $\widehat{M^3_\delta}$ be the union of $\widetilde{M_\delta^3}$ with the compact components of $M^3\setminus \widetilde{M^3_\delta}$. Notice that each component of $M^3\setminus \widehat{M^3_\delta}$ contains at least one end. By inspecting the long exact sequence of the pair $(M^3,\widehat{M_\delta^3})$ and using the fact that the top homology group of an open manifold is trivial, we find that the inclusion $H_2(\widehat{M^3_\delta})\to H_2(M^3)$ is injective. It follows that there are no spherical classes in $H_2(\widehat{M^3_\delta})$. Moreover, since $\Sigma^2$ separates the nonempty collections $E_\pm$, it may be verified that at least one component of each $\partial_\pm \widetilde{M^3_\delta}$ remains in $\partial \widehat{M^3_\delta}$, and that the distance within $\widehat{M^3_\delta}$ from $\Sigma^2$ to these components is unchanged. As in the proof of Main Theorem \ref{Meyer boundary cor}, there is a small perturbation of $\widehat{M_\delta^3}$ to a band $(M^3_\delta,\partial_\pm M^3_\delta)$ with smooth boundary, no spherical homology, and width at least $(w_-+w_+)-3\delta$. In light of the width estimate in Main Theorem \ref{HopfSphereBoundary}, one can conclude $w_-+w_+=\pi/2$.

Now assume that $\Ric\geq 2g$. Define $H_\delta=\sup_{\partial M_\delta^3} |H|$, where $H$ is the mean curvature of $\partial M^3_\delta$. Slightly adjusting Lemma \ref{construct f} with $a=3/4$, $b=3/8$, and appropriate $\varepsilon$, we find a function $f_\delta\in \mathrm{Lip}(M_\delta^3)$ satisfying the following
\begin{align}
\begin{split}
   4+\frac{9}{4}f_\delta^2-\frac{3}{2}|\nabla f_\delta(x)|\ge-C\delta & \quad\text{in }\mathcal{B},\\
    4+\frac{9}{4}f_\delta^2-\frac{3}{2}|\nabla f_\delta(x)|\ge 0 &\quad
   \text{in }M^n_\delta\setminus \mathcal{B},\\
   3f_\delta \le- 2 H_\delta\quad\text{ on }\partial_- M_\delta^3,\quad 3f_\delta\ge2 H_\delta&\quad \text{on } \partial_+ M_\delta^3,
\end{split}
\end{align}
where $\mathcal{B}=\{-\frac{w_-}{2}\le \varrho(x)\le \frac{w_+}{2}\}$ and $C$ is independent of $\delta$. The existence result Proposition \ref{p:existence} with $f_\delta$, yields a spacetime harmonic function $u_\delta$ on $M^3_\delta$ with Dirichlet boundary conditions. Fix a small $\rho>0$ such that $\mathcal{B}\subset \widetilde{M_\rho^3}$, and a point $p\in \widetilde{M_\rho^3}$. By scaling and adding a constant to $u_\delta$, we arrange for $\sup_{\widetilde{M^3_{\rho}}}|\nabla u_\delta|=1$, and $u_\delta(p)=0$. Integrating $du_\delta$ along paths shows that $|u_\delta|\le \textnormal{diam}(\widetilde{M^3_{\rho}})$ on $\widetilde{M^3_{\rho}}$. 

Next, apply the integral formula Lemma \ref{l:2Ricidentity} and use the boundary conditions of $f_\delta$ to find 
\begin{align}   \label{error hopf}
\begin{split}
0\geq&\int_{\underline{u}_{\delta}}^{\overline{u}_\delta}\int_{\Sigma_t}\left(\frac{|\nabla_\Sigma|\nabla u_\delta||^2+(\nabla_{\nu\nu}u_\delta+\frac32 f_\delta|\nabla u_\delta|)^2}{|\nabla u_\delta|^2}+\left(R-\mathrm{Ric}(\nu,\nu)-4\right)\right)dAdt\\
   {}&-C\delta\int_{\mathcal{B}}|\nabla u_\delta|dV,
\end{split}
\end{align}
where $\underline{u}_\delta=\min_{M^3_\delta}u_\delta$ and $\overline{u}_\delta=\max_{M^3_\delta}u_\delta$.
Since it has been arranged that $|u_\delta|$ and $|\nabla u_\delta|$ are uniformly bounded on the compact set $\widetilde{M_\rho^3}$, Schauder estimates allow us find a subsequential $C^{2,\beta}$-limit of $u_\delta$ on $\widetilde{M_\rho^3}$ as $\delta\to0$, which will be denoted by $u$. The function $u$ solves the spacetime harmonic equation with $f(x)=\frac{4}{3} \tan (2\varrho(x)+w_--w_+)$, and must be nonconstant since $\sup_{\widetilde{M^3_{\rho}}}|\nabla u|=1$. We will reuse the notations $\Sigma_t$ and $\nu$ when referring to the level sets of $u$ and their unit normals, respectively. By Fatou's Lemma and boundedness of $\|u_\delta\|_{C^{2,\beta}}$, taking the limit of \eqref{error hopf} produces
\begin{align}
\begin{split}
0=\int_{\underline{u}}^{\overline{u}}\int_{\Sigma_t\cap \widetilde{M^3_\rho}}\left(\frac{|\nabla_\Sigma|\nabla u||^2+(\nabla_{\nu\nu}u+\frac32 f|\nabla u|)^2}{|\nabla u|^2}+\left(R-\mathrm{Ric}(\nu,\nu)-4\right)\right)dAdt,
 \end{split}
\end{align}
where $\underline{u}=\min_{\widetilde{M^3_\rho}}u$ and $\overline{u}=\max_{\widetilde{M^3_\rho}}u$. It follows that the 
conditions \eqref{e:hopfbandrigid} hold on $\widetilde{M_\rho^3}$ wherever $\nabla u\neq0$, and all regular level sets of $u$ are tori. In fact since $\nabla_\Sigma|\nabla u|\equiv0$, components of level sets of $u$ are either entirely regular or critical. 

The argument at this stage in the proof of Main Theorem \ref{Meyer boundary cor} now shows that $u$ has no critical point within the interior of $\widetilde{M^3_\rho}$.
Consequently, $\nabla u$ is parallel to $\nabla\varrho$ and hence $\varrho$ can be considered as a function of $u$ on $\widetilde{M_\rho^3}$. The conditions \eqref{e:hopfbandrigid} then hold across $\widetilde{M^3_\rho}$, and all level sets are tori. This allows us to directly apply the argument in the proof of Main Theorem \ref{HopfSphereBoundary}, and conclude that $g$ has a very particular form on $\widetilde{M^3_\rho}$. To describe this, first introduce the shifted coordinate $s=\varrho-\frac{w_+-w_-}{2}\in(-\pi/4, \pi/4)$. The arguments of the previous section imply that the universal cover of $(\widetilde{M^3_\rho},g)$ splits as $g=ds^2+\phi_\alpha^2(s)(dx^\alpha)^2$, where $x^1,x^2$ are global coordinates on the lifts of the level sets of $u$ and where $\phi_1,\phi_2$ are the functions given in \eqref{phi alpha1}. Since this splitting is independent of the parameter $\rho$, we find that the universal cover of the region $N^3:=\{x\in M^3\colon\varrho(x)\in(-w_-,w_+)\}$ splits as 
\begin{equation}
   \left( \left(-\frac{\pi}{4},\frac{\pi}{4}\right)\times\mathbb{R}^2, ds^2+\phi^2_\alpha(s)(dx^\alpha)^2\right).
\end{equation}
Since $(N^3,g)$ cannot be extended to a Riemannian manifold with at least two ends, we conclude that $N^3=M^3$, completing the proof.

\subsection{Corollaries \ref{HopfSphere} and \ref{c:linkCor}}

Equipped with the open width estimate of Main Theorem \ref{t:2RicIncomplete}, we are now ready to discuss its applications.

\begin{proof}[Proof of Corollary \ref{HopfSphere}] 
Suppose $\Sigma^2\subset M^3$ is an embedded surface of genus at least $1$. 
Since $M^3$ and $\Sigma^2$ are orientable, the normal bundle of $\Sigma^2$ is two-sided. In particular, the open manifold 
\begin{equation}
    N^3=\{x\in M^3\colon d(x,\Sigma^2)<\textit{Inj}_n(\Sigma^2)\}
\end{equation}
has two ends. Evidently $N^3$ is topologically $\Sigma^2\times\mathbb{R}$, and thus cannot support spherical homology classes. The normal injectivity radius upper bound follows from Main Theorem \ref{t:2RicIncomplete} when applied to the separating surface $\Sigma^2$ in $N^3$. 

Now assume that $\mathrm{Ric}\geq 2g$ and $\textit{Inj}_n(\Sigma^2)\geq\pi/4$. The rigidity statement in Main Theorem \ref{t:2RicIncomplete} implies that the universal cover of $(N^3,g)$ is isometric to $((-\frac\pi4,\frac\pi4)\times \R^2,g_\Upsilon)$ for some $\Upsilon\in[0,1]$, and that $\Sigma^2$ is a torus.
Since $N^3$ lies within the smooth closed manifold $(M^3,g)$ and the curvature of $g_\Upsilon$ is unbounded for $\Upsilon>0$, we must have $\Upsilon=0$. We conclude that $(N^3,g)$ is round and, since $g_0$ collapses along its ends, the closure of $N^3$ is the entire original manifold $M^3$.
Consequently, the universal cover of $(M^3,g)$ is the round sphere.

We claim that the torus $\Sigma^2\subset M^3$ is a Heegaard surface. To see this, observe that $\Sigma^2$ lifts to a two-sided embedded flat and minimal surface in $S^3$, which must therefore be the Clifford torus, see \cite{Brendle2}. In particular, $M^3\setminus \Sigma^2$ has two components, each covered by a solid torus. It follows that $\Sigma^2$ separates $M^3$ into two solid tori.
Now, because lens spaces are the only nontrivial quotients of $S^3$ admitting a genus $1$ Heegaard surface (see for instance \cite[page 5]{Heegaard}), $(M^3,g)$ itself must be isometric to a round sphere or a round sens space.
\end{proof}

\begin{proof}[Proof of Corollary \ref{c:linkCor}]
First we show there are no spherical classes in $H_2(M^3\setminus (K_1\cup K_2);\mathbb{Z})$. We argue by contradiction and assume that such a class exists. This means that there is an embedded and homologically nontrivial sphere $\mathcal{S}^2\subset M^3\setminus (K_1\cup K_2)$. Since $M^3$ is a rational homology sphere, and codimension $1$ integer homology of oriented manifolds is torsion free, the surface $\mathcal{S}^2$ bounds a region $U\subset M^3$. The $2$-sidedness of $\mathcal{S}^2$ and nontriviality of its class in $H_2(M^3\setminus (K_1\cup K_2);\mathbb{Z})$ imply that exactly one of the two knots must lie entirely in $U$. Meanwhile, an argument using the Mayer–Vietoris sequence and the triviality of $H_1(\mathcal{S}^2;\mathbb{Q})$, shows that both $U$ and $M^3\setminus U$ are rational homology $3$-balls. It follows that, up to integer multiples, each knot bounds an integer $2$-chain contained entirely in its respective rational homology $3$-ball, and therefore does not intersect the other knot. This contradicts the linking condition, and we conclude that $M^3\setminus(K_1\cup K_2)$ satisfies the homological hypothesis of Main Theorem \ref{t:2RicIncomplete}.

Let $\Sigma^2$ be the boundary of a small distance neighborhood of $K_1$, and notice that $d(K_1,K_2)= d(K_1,\Sigma^2)+d(K_2,\Sigma^2)$.
The upper bound on the distance between $K_1$ and $K_2$ now follows by applying Main Theorem \ref{t:2RicIncomplete} to $M^3\setminus (K_1\cup K_2)$ with the separating surface $\Sigma^2$. Finally, arguing as in the proof of Corollary \ref{HopfSphere} yields the rigidity statement in Corollary \ref{c:linkCor}.
\end{proof}

\subsection{A counterexample}\label{s:counterexample}

We conclude this section with explicit examples demonstrating the necessity of the $\mathrm{Ric}\geq 2g$ hypothesis in the rigidity statements of Main Theorems \ref{HopfSphereBoundary} and \ref{t:2RicIncomplete}, as well as Corollaries \ref{HopfSphere} and \ref{c:linkCor}. In particular, these examples show why a closed $3$-manifold with $2$-Ricci curvature at least $4$, which achieves equality in \eqref{radius}, need not be covered by the class $((-\frac\pi4,\frac\pi4)\times \mathbb{R}^2,g_{\scriptscriptstyle{\Upsilon}})$ without further imposing that $\Ric\ge 2g$. Here is the 1-parameter family of examples. For $t\in(-\frac{\pi}{4},\frac{\pi}{4})$, $x,y\in [0,2\pi)$, and a parameter $\delta\in[0,1]$, consider the doubly warped product metric
\begin{equation}
    g=dt^2+\phi^2dx^2+\psi^2dy^2, 
\end{equation}
where
\begin{equation}
    \phi(t)=e^{\frac{\delta}{2}(\sin 2t-1)}\cos\left(t+\frac{\pi}{4}\right),\quad\quad \psi(t)=e^{-\frac{\delta}{2}(\sin 2t+1)} \sin\left(t+\frac{\pi}{4}\right).
\end{equation}
Notice that $\phi(t)=\psi(-t)$.
Formally, $g$ is a metric on $(-\frac{\pi}{4},\frac{\pi}{4})\times[0,2\pi)\times[0,2\pi)$, but a straightforward calculation of the derivatives of $\phi$ and $\psi$ along $t=\pm\frac\pi4$ shows that $g$ actually extends to a smooth metric on $S^3$, viewed as the union of two solid tori.
A calculation shows that $\Ric$ is diagonalized in the $(t,x,y)$-coordinates, and that
\begin{align}
    \Ric(\partial_t, \partial_t)=&-\frac{\phi''}{\phi}-\frac{\psi''}{\psi}=2+4\delta-2\delta^2\cos^2(2t),
   \\ \Ric(\phi^{-1}\partial_x,\phi^{-1}\partial_x)=&-\frac{\phi''}{\phi} -\frac{\phi'\psi'}{\phi\psi}  =2+4\delta \sin 2t,
   \\ \Ric(\psi^{-1}\partial_y,\psi^{-1}\partial_y)=&-\frac{\psi''}{\psi} -\frac{\phi'\psi'}{\phi\psi}=2-4\delta \sin 2t.
\end{align}
Evidently, the Ricci curvature is not bounded below by $2$ when $0< \delta \le 1$. However, the $2$-Ricci curvature is always bounded below by $4$, and the width of these open manifolds attains the maximum value $\pi/2$. This shows saturation in the band-width inequalities of Main Theorem \ref{t:2RicIncomplete}, and Corollaries \ref{HopfSphere} and \ref{c:linkCor}.
Moreover, it may be easily checked that by removing thickened neighborhoods of the `Hopf link' $\{t=\pm\pi/4\}$, saturation also occurs in Main Theorem \ref{HopfSphereBoundary}.


\appendix 

\section{A Barrier Construction}

\begin{lemma}\label{L:A convergence}
Let $(M^n,g)$ be a closed $n$-dimensional Riemannian manifold. Fix distinct points $p,q\in M^n$, a constant $C\ge\frac{n-1}n$, and suppose that $f\in\mathrm{Lip}_{loc}(M^n\setminus \{p,q\})$ satisfies
\begin{align}\label{eq:condition f}
    f(x)=\begin{cases}
    \frac{C}{d(x,p)}+{O}(d(x,p))& \textnormal{ as }x\to p,
    \\  
    -\frac{C}{d(x,q)}+{O}(d(x,q))&  \textnormal{ as }x\to q.
    \end{cases}
\end{align}
For each small $\varepsilon>0$, let $u_\varepsilon\in C^{2,\alpha}(M^n\backslash (B_\varepsilon(p)\cup B_\varepsilon(q)))$, $\alpha\in (0,1)$ be the solution of the spacetime harmonic Dirichlet problem
\begin{align}\label{eq:appendix}
\begin{cases}
\Delta u_\varepsilon+nf |\nabla u_\varepsilon|=0 &\textit{\;in\;} M^n\backslash (B_\varepsilon(p)\cup B_\varepsilon(q)), \\
u_\varepsilon=1 &\textit{\;on\;}\partial B_\varepsilon(p),\\
u_\varepsilon=-1 &\textit{\;on\;}\partial B_\varepsilon(q).
\end{cases}
\end{align}
Then there exists a constant $C_1$ independent of $\varepsilon$ such that 
\begin{equation}\label{e:gradientapp}
    |\nabla u_{\varepsilon}|\leq C_1\quad\quad\text{on $\partial B_{\varepsilon}(p)\cup \partial B_{\varepsilon}(q)$}.
\end{equation}
Furthermore, $\{u_\varepsilon\}_{\varepsilon>0}$ subconverges 
to a spacetime harmonic function $u\in C^{2,\beta}(M^n \setminus \{p,q\})\cap C^{0}(M^n)$, with $\beta\in(0,\alpha)$, which satisfies $u(p)=1$ and $u(q)=-1$. If $n=3$, then additionally there is a uniform gradient bound $|\nabla u_{\varepsilon}|\leq C_1$ on $M^3 \setminus (B_{\varepsilon}(p) \cup B_{\varepsilon}(q))$, and the limit function is globally Lipschitz. 
\end{lemma}

\begin{proof}
We will construct subsolutions $\underline u_\varepsilon$ with a $C^1$ bound that is uniform for all small $\varepsilon>0$.
More precisely let $r(x)=d(x,p)$ and define 
\begin{align}
    \underline u_\varepsilon=a_\varepsilon-br+br^2,
\end{align}
where $a_\varepsilon$ and $b$ are positive constants to be determined.
To ensure that $\underline u_\varepsilon$ agrees with $u_\varepsilon$ on $\partial B_\varepsilon(p)$, we choose $a_\varepsilon=1+b\varepsilon-b\varepsilon^2$. 
Next, recall that the mean curvature of geodesic spheres $\partial B_{r}(p)$ has an expansion $H_r=\frac{n-1}{r}+O(r)$ for small radii. Then, using also the assumption \eqref{eq:condition f} it follows that
\begin{equation}\label{aofinlinqh}
 \left|f-\frac{C}{r}\right|\le \frac{1}{2n} \quad\quad\text{and}\quad\quad \left|H_r-\frac{n-1}{r}\right|\le \frac{1}{2} \quad \quad\textnormal{\;\;in}\!\!\quad B_{r_0}(p),
\end{equation}
for some fixed $r_0\in(0,\frac{1}{4})$ less than the injectivity radius at $p$ and independent of $\varepsilon$. 
From now on it will be assumed that $\varepsilon<\frac{r_0}{4}$. In order to achieve
$\underline{u}_\varepsilon(r_0)\leq -1$, which according to the maximum principle is additionally a lower bound for $u_\varepsilon$, we choose $b= \frac{4}{r_0}$ since then
\begin{equation} \label{b r0}
    \underline{u}_\varepsilon(r_0)= a_{\varepsilon}-br_0+br_0^2=1-b(r_0-\varepsilon-r_0^2+\varepsilon^2)\le 1-b\left(\frac{3}{4}r_0-\varepsilon \right)\le 1-\frac{1}{2}r_0b= -1.
\end{equation} 
To see that $\underline{u}_\varepsilon$ is a subsolution of equation \eqref{eq:appendix} on $B_{r_0}(p)\setminus B_{\varepsilon}(p)$, note that
$\partial_r\underline{u}_\varepsilon=-b+2br<0$, $|\nabla \underline{u}_\varepsilon|=-\partial_r \underline{u}_\varepsilon$,
and by assumption $C\geq\frac{n-1}{n}$ which imply
\begin{align}
\begin{split}
    \Delta \underline{u}_\varepsilon +nf|\nabla \underline{u}_\varepsilon|=&\partial_r^2 \underline{u}_\varepsilon +H_r\partial_r \underline{u}_\varepsilon -nf\partial_r \underline{u}_\varepsilon\\
    \ge&(-b+2br)\left(\frac{n-1}{r}-\frac{nC}{r} +1\right)+2b\\
    \ge&-b+2br+2b>0,
\end{split}
\end{align}
where in the second to last inequality \eqref{aofinlinqh} was used.
The comparison principle may now be applied to find that $u_\varepsilon\ge\underline{u}_\varepsilon\ge1-b(r-\varepsilon)$ on the annulus $B_{r_0}(p)\setminus B_\varepsilon(p)$. Since $u_{\varepsilon}$ and $\underline{u}_{\varepsilon}$ agree at the smaller radius, the desired boundary gradient estimate at $\partial B_\varepsilon(p)$ now follows
\begin{equation}
|\nabla u_\varepsilon|= -\partial_ru_\varepsilon\le -\partial_r\underline{u}_\varepsilon=b-2b\varepsilon\le b.
\end{equation}
Similarly, we obtain a uniform boundary gradient estimate for $u_\varepsilon$ at $\partial B_\varepsilon(q)$, which establishes \eqref{e:gradientapp}.

Consider now the issue of convergence. Since $|u_{\varepsilon}|\leq 1$, interior elliptic estimates combined with a diagonal argument show that $u_\varepsilon$ subconverges to a spacetime harmonic function function $u$ uniformly on compact subsets in $C^{2,\beta}(M\setminus\{p,q\})$, for any $\beta\in(0,\alpha)$.  
Moreover, the barrier estimates give $1\geq u_\varepsilon\ge 1-br$ on $B_{r_0}(p)\setminus B_\varepsilon(p)$, and hence the limit function $u$ satisfies the same uniform estimate on $B_{r_0}(p)$. It follows that $u(x)\to1$ as $x\to p$, and similarly $u(x)\to-1$ for $x\to q$. Lastly, this estimate also implies that the solution is Lipschitz at $p$ and $q$.

In fact, the solution is globally Lipschitz in dimension 3. To see this,
we will combine the proof of Theorem \ref{T:Gradient Estimate} below, to achieve a uniform gradient estimate for $u_\varepsilon$ on $M^3 \setminus (B_{\varepsilon}(p) \cup B_{\varepsilon}(q))$. Let $p_1\in M^3 \setminus (B_{\varepsilon}(p) \cup B_{\varepsilon}(q))$ be such that $\rho=\tfrac12 d(p_1,\partial B_\varepsilon(p))\le\frac{r_0}{2}$. Note that since $1-u_{\varepsilon}$ is positive in $B_{\rho}(p_1)$, we may replace $u_\varepsilon$ by $1-u_\varepsilon$ in \eqref{F estimate}. Furthermore, using the notation of the proof of Theorem \ref{T:Gradient Estimate},
it holds that $F(x_1)\le c\rho$ for some uniform constant $c$. Then \eqref{(B.10)} implies that
\begin{equation}
|\nabla u_\varepsilon|(x)\le \frac{4}{3}\rho^{-2}\cdot c\rho(1-u_\varepsilon(x)), \quad\quad\text{ for } x\in B_\frac{\rho}{2}(p_1).
\end{equation}
Moreover $d(x,\partial B_\varepsilon(p))\le \frac{5}{2}\rho$ for $x\in B_\frac{\rho}{2}(p_1)$, and therefore the maximum principle argument above yields $1-u_\varepsilon (x)\le b(r-\varepsilon)\le \frac{5}{2}b\rho$. It follows that
\begin{equation}
    |\nabla u_\varepsilon|(x)\le \frac{4}{3}c\rho^{-1}(1-u_\varepsilon(x))\le\frac{10}{3}cb, \quad\quad\text{ for } x\in B_\frac{\rho}{2}(p_1).
\end{equation}
Similar arguments apply also near the point $q$. Hence, $|\nabla u_\varepsilon|$ is uniformly bounded on the complement of $B_\varepsilon(p)\cup B_\varepsilon(q)$, independent of $\varepsilon$.  
\end{proof}

\section{Gradient Estimate for Spacetime Harmonic Functions}\label{A: gradient estimate}

This section is devoted to the proof of Theorem \ref{T:Gradient Estimate}, which is a modified version of the Cheng-Yau gradient estimate \cite{schoen1994lectures}.
We begin with a preliminary local estimate.

\begin{lemma}\label{L:appendix B}
Let $(M^3,g)$ be a $3$-dimensional Riemannian manifold with $\Ric\ge -\Lambda g$, $\Lambda>0$. Let $f\in C^{1,\alpha}(M^3)$, and let $u\in C^{2,\alpha}(M^3)$, $\alpha\in(0,1)$ be a positive solution of $\Delta u+3f|\nabla u|=0$.
If $p\in M^3$ is a point such that $\nabla u(p)\ne0$,
then $\phi=\frac{|\nabla u|}{u}$ is $C^{2,\alpha}$ near $p$ and at this point the following inequality holds 
\begin{align}
    \Delta\phi \ge & -\left(\Lambda+3  |\nabla f|-\frac{9}{2}f^2\right )\phi-\frac{\langle \nabla\phi, \nabla u\rangle}{u}+\frac{1}{2}\phi^3+3f\phi^2.
\end{align}
\end{lemma}

\begin{proof}
Let $p\in M^3$ be a point with $\nabla u\ne0$, and let $\{x^i\}_{i=1}^3$ be a normal coordinate system around $p$ with $\partial_1=\frac{\nabla u}{|\nabla u|}$ at $p$.
Note that since $f\in C^{1,\alpha}(M^3)$, and $\nabla u\ne0$, $u\in C^{3,\alpha}(M^3)$ by Schauder estimates.
Using the notations $u_{ij}:=\nabla_{ij} u$, $f_i:=\partial_i f$, we compute at $p$
\begin{align}
\begin{split}
    \frac{1}{2}\Delta (|\nabla u|^2)=&|\nabla^2u|^2+\langle \nabla u,\nabla \Delta u\rangle+\Ric(\nabla u,\nabla u)
    \\ \ge& |\nabla^2u|^2-\langle \nabla u, \nabla (3f|\nabla u|)\rangle-\Lambda |\nabla u|^2
    \\=& |\nabla^2u|^2-3|\nabla u|^2f_1-3f|\nabla u|u_{11}-\Lambda |\nabla u|^2
    \\ \ge &
    |\nabla^2u|^2+3f|\nabla u|\left(3f|\nabla u|+\sum_{i\neq 1}u_{ii}\right)-(\Lambda+3 |\nabla f|) |\nabla u|^2.
    \end{split}
\end{align}
Combining the above inequality with the identity $\frac{1}{2}\Delta(|\nabla u |^2)=|\nabla u|\Delta |\nabla u|+|\nabla |\nabla u||^2$, we obtain
\begin{align}\label{e:appBest1}
\begin{split}
    |\nabla u|\Delta|\nabla u|\ge& \sum_{j\neq1}  u_{1j}^2+\sum_{i,j\neq 1}u_{ij}^2+3f|\nabla u|\sum_{i\neq 1}u_{ii}-(\Lambda+3  |\nabla f|-9f^2)|\nabla u|^2
   \\ \ge&  \sum_{j\neq1} u_{1j}^2+ \frac{1}{2}\left(\sum_{i\neq 1}u_{ii}+3 f|\nabla u|\right)^2-\left(\Lambda+3  |\nabla f|-\frac{9}{2}f^2\right)|\nabla u|^2
   \\ \ge& \frac{1}{2} |\nabla |\nabla u||^2-\left(\Lambda+3  |\nabla f|-\frac{9}{2}f^2\right)|\nabla u|^2,
   \end{split}
\end{align}
where we have made use of the equation $\Delta u+3f|\nabla u|=0$ in the last inequality. Combining \eqref{e:appBest1} with Cauchy-Schwarz implies that
\begin{align}
\begin{split}
    \Delta\phi=&\frac{\Delta|\nabla u|}{u}-\frac{\langle 2\nabla|\nabla u|, \nabla u\rangle}{u^2}+\frac{2|\nabla u|^3}{u^3}-|\nabla u|\frac{\Delta u}{u^2}
    \\=&\frac{\Delta|\nabla u|}{u}-\frac{\langle 2\nabla |\nabla u|, \nabla u\rangle}{u^2}+2\phi^3+\frac{3f|\nabla u|^2}{u^2}
    \\ \ge & 
    \frac{1}{2}\frac{|\nabla|\nabla u||^2}{u|\nabla u|}-\left(\Lambda+3  |\nabla f|-\frac{9}{2}f^2\right)\phi-\frac{\langle \nabla\phi, \nabla u\rangle}{u}-\frac{\langle \nabla |\nabla u|, \nabla u\rangle}{u^2}+\phi^3+3f\phi^2
    \\ \ge & -\left(\Lambda+3  |\nabla f|-\frac{9}{2}f^2\right)\phi-\frac{\langle \nabla\phi, \nabla u\rangle}{u}+\frac{1}{2}\phi^3+3f\phi^2 ,
    \end{split}
\end{align}
which finishes the proof.
\end{proof}

Before stating the gradient estimate, we need some notation. Given a Lipschitz function $f$ on a Riemannian manifold $(M^n,g)$ and a point $p\in M^n$, we denote
\begin{equation}
    L_f(p)=\limsup_{\varepsilon\rightarrow 0}\mathrm{Lip}_{B_{\varepsilon}(p)}(f).
\end{equation}

\begin{theorem}\label{T:Gradient Estimate}
Let $(M^3,g)$ be an $3$-dimensional Riemannian manifold. Suppose $p\in M^3$ and $\rho>0$ satisfy $B_\rho(p)\cap\partial M^3=\emptyset$.
Let $C_0>0$ be a constant and let $u\in C^{2,\alpha}(M^3)$, $\alpha\in(0,1)$ be a solution of $\Delta u+3f|\nabla u|=0$, where $f\in \mathrm{Lip}(M^3)$ does not change sign and satisfies $\frac{9}{2}f^2-3 L_f\ge -C_0$  on $B_\rho(p)$. Then 
\begin{equation}
   \sup_{B_{\frac{\rho}{2}}(p)} |\nabla u|\le C(1+\sup_{B_\rho(p)} |u|),
    \label{gradient esitmate}
\end{equation}
where the constant $C$ depends only on $\rho$, $C_0$, and the lower bound for Ricci curvature in $B_\rho(p)$.
\end{theorem}

\begin{proof}
Let us first consider the case $f\in C^{1,\alpha}(B_\rho(p))$.
Following the proof of \cite[Theorem 3.1]{schoen1994lectures}, we introduce the notation $F=(\rho^2-r^2)\phi$ where $\phi=\frac{|\nabla u|}u$ and $r(x)= d(x,p)$.
Replacing $u$ by $u+1+\sup_{B_{\rho}(p)}|u|$, we may assume that $u>0$ which makes $\phi$ and therefore $F$ well-defined.
Observe that $F=0$ on $\partial B_\rho(p)$, and that $F\ge 0$ in $ B_\rho(p)$. Let $x_1\in B_\rho(p)$ a point where $F$ attains its maximum value in $ B_\rho(p)$.
We may assume that $x_1$ lies in the interior of $B_\rho(p)$ and that $F$ is strictly positive at this point, since otherwise $u$ must be constant.

Suppose first that the maximal point $x_1$ does not belong to the cut locus $\text{Cut}(p)$.
Since $x_1$ lies in the interior of $B_\rho(p)$, we have $\nabla F=0$ and $\Delta F\le 0$ at this point.
The former relation implies $(\rho^2-r^2)\nabla \phi=\phi\nabla r^2$, while the latter relation together with Lemma \ref{L:appendix B} yields
\begin{align}\label{gradient 1}
\begin{split}
0\ge&\frac{\Delta F}{(\rho^2-r^2)\phi}
\\=& \frac{\Delta\phi}{\phi}- \frac{2\langle \nabla r^2, \nabla   \phi\rangle}{(\rho^2-r^2)\phi}-\frac{\Delta r^2}{\rho^2-r^2}
\\ \ge &
-\Lambda-3|\nabla f|+\frac{9}{2}f^2- \frac{\langle\nabla \phi, \nabla u\rangle}{u\phi}+\frac{1}{2}\phi^2+3f\phi
- \frac{2\langle\nabla r^2, \nabla   \phi\rangle}{(\rho^2-r^2)\phi}-\frac{C_1}{\rho^2-r^2}
\\ = & -\Lambda-3|\nabla f|+\frac{9}{2}f^2- \frac{2r\langle\nabla r, \nabla u\rangle}{u(\rho^2-r^2)}+\frac{1}{2}\phi^2+3f\phi- \frac{8r^2}{(\rho^2-r^2)^2}-\frac{C_1}{\rho^2-r^2}
\\ \ge& -\Lambda-3|\nabla f|+\frac{9}{2}f^2- \frac{2r\phi}{\rho^2-r^2}+\frac{1}{2}\phi^2+3f\phi
- \frac{8r^2}{(\rho^2-r^2)^2}-\frac{C_1}{\rho^2-r^2},
\end{split}
\end{align}
where $\Lambda$ is the lower bound of $\Ric$ in $B_\rho(p)$ and by \cite[Corollary 1.2]{schoen1994lectures} we have
\begin{align}
\Delta r^2\le 6+4r\sqrt{\frac{\Lambda}{2}}\le 6+4\rho\sqrt{\frac{\Lambda}{2}}=:C_1.
\end{align}
Recall that $f$ does not change sign by assumption.
If $f\le 0$, we may apply the arguments below with obvious modifications to the function $u_0=2\sup_{B_{\rho}(p)}|u|+2-u$ which satisfies $\Delta u_0=nf|\nabla u_0 |$.
Therefore, we will assume that $f\ge0$ in what follows. 
Since $\frac{9}{2}f^2-3|\nabla f|\ge -C_0$ and $\phi(x_1)> 0$, at $x_1$ we have
\begin{align}
\label{F estimate}
\begin{split}
    0\ge& \frac{1}{2}\phi^2-\frac{2r}{\rho^2-r^2}\phi- \frac{8r^2}{(\rho^2-r^2)^2}-\frac{C_1}{\rho^2-r^2}-\Lambda-C_0
    \\ =& \frac{1}{(\rho^2-r^2)^2}\left[\frac{1}{2}F^2-2r F-8r^2-C_1(\rho^2-r^2)-(\Lambda+C_0)(\rho^2-r^2)^2\right].
    \end{split}
\end{align}
Therefore, $F(x_1)\le C_2$ where $C_2$ depends on $\Lambda$, $C_0$, $C_1$, and $\rho$.
Furthermore, since $x_1$ is a point where $F=(\rho^2-r^2)\frac{|\nabla u|}u$ attains its maximum, we have 
\begin{align}\label{(B.10)}
    |\nabla u|(x)\le\frac43 \rho^{-2}F(x_1) u(x)\le C(1+\sup_{B_\rho(p)} |u|)
\end{align}
for $x\in B_{\frac\rho2}(p)$, where $C$ is a constant depending only on $\Lambda$, $C_0$, and $\rho$, but not on $f$.


Next, suppose that $x_1\in \text{Cut}(p)$.
Following the proof in \cite{schoen1994lectures}, we let $\gamma$ be a minimizing geodesic connecting $x_1$ and $p$. 
For $\epsilon>0$ sufficiently small, let $\bar{p}$ be the point on $\gamma$ with $d(p,\bar{p})=\epsilon$, and define 
\begin{equation}
\bar{F}(x)=\left[\rho^2-(\bar{r}(x)+\epsilon)^2\right]\phi, \quad\quad\quad  \bar{r}(x):=d(x,\bar{p}).
\end{equation}
Since $\bar{r}(x)+\epsilon\ge r(x)$ it holds that $\bar{F}(x)\le F(x)$. Moreover $\bar{F}(x_1)=F(x_1)$, because $\gamma$ is minimizing. Thus, $x_1$ is also a maximal point of $\bar{F}(x)$.
Noting that $\bar{r}(x)$ is smooth at $x_1$, the above computations hold with $\bar F$ instead of $F$.

Finally, if $f$ is merely Lipschitz we may approximate $f$ by $f_\varepsilon\in C^{\infty}$ such that 
\begin{equation}\label{aonfoiqwnohinqh}
  |f(x)-f_\varepsilon(x)|\le \varepsilon, \quad\quad \quad |\nabla f_\varepsilon(x)|\le \mathrm{Lip}_{B_{\varepsilon}(x)}(f)+\varepsilon,
\end{equation}
for all $x\in B_\rho(p)$. Such an approximation $f_\varepsilon$ can be constructed as in \cite[Theorem 2.2]{CRi}, and satisfies $f_{\varepsilon}\geq -\varepsilon$ since $f$ is nonnegative; this property is sufficient to carry out the arguments above. Furthermore, since \eqref{aonfoiqwnohinqh} implies that $L_{f_{\varepsilon}}\leq L_f +\tfrac{1}{10}$, we have $\frac{9}{2}f^2_\varepsilon-3 L_{f_\varepsilon}\ge -C_0-1$ for $\varepsilon>0$ sufficiently small. Let $u_\varepsilon$ be the solution to the Dirichlet problem
\begin{align}
    \begin{cases}
    \Delta u_\varepsilon+3f_\varepsilon|\nabla u_\varepsilon|=0&\text{ in } B_\rho(p),\\
    u_\varepsilon=u&\text{ on }\partial B_\rho(p).
    \end{cases}
\end{align}
Observe that the above gradient estimate applied to $u_\varepsilon$ produces 
\begin{equation}\label{aefjiiohn}
   \sup_{B_{\frac{\rho}{2}}(p)} |\nabla u_\varepsilon|\le C(1+\sup_{ B_\rho (p)}|u_\varepsilon|)
   =C(1+\sup_{ \partial B_\rho(p)}|u|),
\end{equation}
where the last equality follows from the maximum principle. Moreover, as in \cite[Section 4]{HKK} the solutions $u_{\varepsilon}$ are uniformly controlled in $C^{2,\beta}(B_\rho(p))$ for $\beta\in(\alpha,1)$. By uniqueness, it follows that $u_{\varepsilon}$ subconverges to $u$ in $C^{2,\alpha}(B_{\rho}(p))$. The desired result now follows from \eqref{aefjiiohn}.
\end{proof}

\section{Open Riemannian Manifolds}\label{s:openappendix}

In this section, we will present some remarks on the nonstandard topic of open Riemannian manifolds, which are the central objects in Main Theorems \ref{Meyer boundary cor} and \ref{t:2RicIncomplete}. To begin, consider an open manifold $M$, and suppose that we are given a sequence of exhausting compact sets $\{K_i\}_{i=1}^\infty$ satisfying $K_i\subset K_{i+1}$ and $\bigcup_{i=1}^\infty \text{int}(K_i)=M$.
An \textit{end} of $M$ is a collection of connected components $\{U_i\}_{i=1}^\infty$ $U_i\subset (M\setminus K_i)$ satisfying $U_{i+1}\subset U_i$ up to the equivalence relation: $\{U_i\}_{i=1}^\infty\sim\{U'_i\}_{i=1}^\infty$ if $\{U_i\}_{i=i_0}^\infty=\{U'_i\}_{i=i'_0}^\infty$ for some integers $i_0$ and $i'_0$, see for instance \cite[Section 13.4]{ends} for further details. A path $\gamma\subset M$ is said to intersect an end $E$ if $\gamma\cap U_i\neq0$ for all $i$ and any representative $\{U_i\}_{i=1}^\infty$ of $E$.

An open manifold $M$ is said to be {\em{hemicompact}} if it supports an exhaustion by compact sets $\{K_i\}_{i=1}^\infty$ with the property that any given compact set lies within the interior of some $K_i$. If $M$ is hemicompact, there is a canonical compactification, known as the Freudenthal compactification and denoted by $\mathcal{F}(M)$, see for instance \cite{Freudenthal,Peschke}. As a set, $\mathcal{F}(M)=M\cup \mathcal{E}(M)$ where $\mathcal{E}(M)$ denotes the space of $M$'s ends. The topology of $\mathcal{F}(M)$ is given by the open sets of $M$ along with sets $U\cup \{e\}$ where $e\in\mathcal{E}(M)$ is an end and $U$ is an open set of $M$ contained in some representative of $e$. 
Every locally compact Lindel\"{o}f space is hemicompact \cite[page 84]{Minguzzi}, where we recall that a Lindel\"{o}f space is a topological space in which every open cover has a countable subcover. Since a second-countable space is Lindel\"{o}f, all manifolds considered in this paper are hemicompact, see \cite[Theorem 1.3.1]{PetersenM}.

Below, Lemma \ref{l:compactification} can be considered as a variant of the Hopf-Rinow Theorem and has great technical importance to the proofs of Main Theorems \ref{Meyer boundary cor} and \ref{t:2RicIncomplete}. 
First, a technical lemma.

\begin{lemma}\label{l:completeness}
Let $(M,g)$ be an open Riemannian manifold with ends $\mathcal{E}(M)$. If $q\in M$ and $s<d(q,\mathcal{E}(M))$, then the exponential map $\exp_q$ is defined on the closed ball $\overline{B_s(0)}\subset T_q M$.
\end{lemma}

\begin{proof}
We will first show that the closed geodesic ball $\overline{B_s(q)}$ is complete. Let $\{p_i\}_{i=1}^\infty$ be a Cauchy sequence in $\overline{B_s(q)}$. 
Since $\mathcal{F}(M)$ is compact, there is a limit point $p_0\in\mathcal{F}(M)$ of $\{p_i\}_{i=1}^\infty$. If this limit is an end, $p_0\in \mathcal{E}(M)$, then we may represent $p_0$ by a sequence $\{U_i\}_{i=1}^\infty$ of connected open sets in $M$ satisfying $U_i\supset U_{i+1}$ for all $i\in\mathbb{N}$. Since each $U_i$ contains infinitely many points of $\{p_j\}_{j=1}^\infty$, by passing to a subsequence we may assume that $p_i\in U_i$ and $d(p_i,p_j)\le 2^{-i}$ for $i\le j$. We proceed by constructing a short curve from $\overline{B_s(q)}$ to the end $p_0$. For $i\in\mathbb{N}$, let $\gamma_i$ be a curve connecting $p_i$ to $p_{i+1}$ with length $|\gamma_i|\le 2d(p_i,p_{i+1})\le 2^{1-i}$. Denote by $\Gamma_i$ the concatenation $\cup_{k\ge i}\gamma_k$. Then $p_j\in \Gamma_i\cap U_j$ for any $j\ge i$, and $|\Gamma_i|\le 2^{2-i}$. Therefore $d(\overline{B_s(q)},\mathcal{E}(M))\le 2^{2-i}$ for any $i\in \N$, contradicting the property that $s<d(q,\mathcal{E}(M))$. Thus $p_0$ must lie in $M$, and in fact $p_0\in \overline{B_s(q)}$ due to closedness. It follows that $\overline{B_s(q)}$ is complete.

To show that $\exp_q$ is defined on all of $\overline{B_s(0)}\subset T_q M$, we proceed by contradiction and assume that there is a smallest radius $s_0\in(0,s)$ where this fails. Then there exists an inextendable geodesic $\gamma$ emanating from $q$ which is defined only for times $t\in[0,s_0)$. Take a sequence of times $t_i$ within this interval converging to $s_0$, and note that $\gamma(t_i)$ forms a Cauchy sequence in $\overline{B_s(q)}$. By completeness, $\gamma(t_i)$ converges to a point in $\overline{B_s(q)}$. Then by standard local existence arguments, this geodesic may be extended for a short time beyond $s_0$, which is a contradiction. 
\end{proof}

\begin{lemma}\label{l:compactification}
Let $(M,g)$ be an open hemicompact Riemannian manifold with ends $\mathcal{E}(M)$. Suppose $\mathcal{K}\subset M$ is a compact set. If $s<d(\mathcal{K},\mathcal{E}(M))$, then the distance neighborhood $V_s:=\{x\in M| \;d(x,\mathcal{K})\leq s\}$ is compact.
\end{lemma}

\begin{proof} 
We first argue that it suffices to show the Lemma holds in the case where $\mathcal{K}$ is a single point.  Let $\delta>0$ such that $s+2\delta< d(\mathcal{K},\mathcal{E}(M))$. 
Take a $\delta$-net $\{q_1,\cdots,q_m\}$ of $\mathcal{K}$, and note $V_s\subset \cup_{i=1}^m B_{s+2\delta}(q_i)$, where $B_{s+2\delta}(q_i)$ denotes the open geodesic ball around $q_i$ of radius $s+2\delta$. 
Next, observe that $V_s$ is closed in $M$ since it is the preimage of a closed set under the function $x\mapsto d(x,\mathcal{K})$, which is continuous by the triangle inequality.
Therefore, if the closed set $\cup_{i=1}^m\overline{B_{s+2\delta}(q_i)}$ is compact, then $V_s$ is a closed subset of a compact set which implies that $V_s$ is compact. 
Hence, we need only show that the closure $\overline{B_s(q)}$ is compact for any point $q\in M$ and number $s$ satisfying $s<d(q,\mathcal{E}(M))$. 

Given such a $q$ and $s$, the exponential map $\exp_q$ is defined on $\overline{B_t(0)}\subset T_q M$ by Lemma \ref{l:completeness} and we may consider the set
\begin{equation}
    \mathcal{A}=\{t\in[0,s]|\; \exp_q: \overline{B_t(0)}\subset T_q M \to \overline{B_{t}(q)} \subset M \textnormal{ is surjective}\}.
\end{equation}
We claim that if $t_0\in\mathcal{A}$, then the ball $\overline{B_{t_0}(q)}$ is compact. Indeed, notice that the exponential map $\mathrm{exp}_q:\overline{B_{t_0}(0)}\subset T_q M \to M$ is continuous, a fact which follows from the well-posedness of the initial value problem for geodesics in addition to the positive distance imposed between $\overline{B_{t_0}}$ and $\mathcal{E}(M)$. 
It follows that $\overline{B_{t_0}(q)}$ is the continuous image of a compact set and is therefore compact.

The proof will be complete upon establishing that $\mathcal{A}$ is the entire interval $[0,s]$. Clearly, $0\in \mathcal A$, and we proceed by showing the component of $\mathcal{A}$ containing $0$ is both open and closed, starting with the latter. 
Suppose $[0,t)\subset \mathcal{A}$ and let $p\in M$ such that $d(p,q)=t$. 
Then for any $i\in \N$,  there exists a curve $\gamma_i$ connecting $p$ and $q$ with $|\gamma_i|\le t+\frac{1}{i}$. 
Let $\varepsilon>0$ small enough so that $B_\varepsilon(p)$ is a normal neighborhood of $p$. 
Let $p_i\in \gamma_i\cap \partial B_\varepsilon(p)$, then $d(p_i,q)\le t+\frac{1}{i}-\varepsilon$. Since $\partial B_\varepsilon(p)$ is compact, there is a limit point $p_0\in\partial B_\varepsilon(p)$ of $\{p_i\}_{i=1}^\infty$. 
Since $d(p_0,q)\le t-\varepsilon\in\mathcal{A}$, there exists a minimizing geodesic $\alpha$ connecting $p_0$ and $q$. 
Let $\beta$ be the minimizing geodesic connecting $p_0$ and $p$. 
Denote by $\gamma$ the concatenation of $\alpha$ and $\beta$ which is a piecewise smooth geodesic between $p$ and $q$. 
Because $|\gamma|=|\alpha|+|\beta|\le t-\varepsilon+\varepsilon=t$ and $d(p,q)=t$, $\gamma$ is in fact smooth.
Since $p$ was arbitrary, $t\in \mathcal{A}$.

We next show that the component of $\mathcal A$ containing $0$ is open. Suppose that $[0,t]\subset \mathcal A$. Since $\overline{B_t(q)}$ is compact, the injectivity radius $\varepsilon:=\inf_{x\in \overline{B_t(q)}}\mathrm{inj}_x$ is positive.
Our goal is to show that $[0,t+\varepsilon)\in\mathcal A$.
Let $p'\in M$ be such that $t<d(q,p')<t+\varepsilon$.
For $i\in\mathbb{N}$, let $\gamma_i$  be a unit speed curve from $q$ to $p'$ with length $|\gamma_i|<d(q,p') +\frac1i$. 
Consider the points $p_i'=\gamma_i(t)$ for $i\in\mathbb{N}$.
Since $\overline{B_{t}(q)}$ is compact by the assumption $t\in\mathcal A$, we may pass to a subsequence to obtain a limit $p_i'\to p_0'\in \overline{B_{t}(q)}$. Note that $d(p_i',p')\leq d(q,p')-t+1/i$ and so $d(p'_0,p')\leq d(q,p')-t<\varepsilon$.
Since $t\in \mathcal A$, we can find a geodesic $\alpha $ connecting $q$ and $p_0'$ with length $|\alpha|\le t$.
By the definition of $\varepsilon$, we may find a geodesic $\beta$ connecting $p_0'$ to $p'$ with length $|\beta|=d(p_0',p')\le d(q,p')-t$.
Concatenating $\alpha$ and $\beta$ we obtain a piecewise smooth geodesic $\gamma$ connecting $p'$ and $q$, with $|\gamma|\le d(q,p')$. It follows that $\gamma$ must in fact be smooth everywhere.
Hence, the connected component of $\mathcal A$ containing $0$ is open. This component must then agree with $[0,s]$, yielding the desired result.
\end{proof}


\end{document}